\documentclass[12pt]{article}
\usepackage{amsfonts}
\usepackage{amsmath}
\usepackage{amssymb}
\usepackage{amsthm}
\usepackage[english]{babel}
\usepackage{fontspec}
\usepackage[margin=1in]{geometry}
\usepackage{graphicx}
\usepackage[colorlinks]{hyperref}
\usepackage[utf8]{inputenc}
\usepackage{mathabx}
\usepackage{mathrsfs}
\usepackage{mathtools}
\usepackage{titlesec}
\usepackage{verbatim}
\usepackage{wasysym}
\usepackage{xcolor}

\numberwithin{equation}{section}

\colorlet{linkequation}{red}

\newcommand*{\SavedEqref}{}
\let\SavedEqref\eqref
\renewcommand*{\eqref}[1]{%
  \begingroup
    \hypersetup{
      linkcolor=linkequation,
      linkbordercolor=linkequation,
    }%
    \SavedEqref{#1}%
  \endgroup
}

\titleformat{\paragraph}
{\normalfont\normalsize\bfseries}{\theparagraph}{1em}{}
\titlespacing*{\paragraph}
{0pt}{3.25ex plus 1ex minus .2ex}{1.5ex plus .2ex}

\newcommand{\abs}[1]{\left| #1\right|}
\newcommand{\norm}[1]{\left|\left|#1\right|\right|}
\providecommand{\wt}[1]{\widetilde{#1}}
\providecommand{\wh}[1]{\widehat{#1}}

\DeclareMathOperator{\Vol}{Vol}

\DeclareMathOperator{\tr}{tr}

\DeclareMathOperator*{\argmin}{arg\,min}

\DeclareMathOperator{\EL}{\textnormal{\textsterling}\!}

\newcommand{\Rm}{\mathrm{Rm}}
\newcommand{\Ric}{\mathrm{Ric}}
\newcommand{\scal}{\mathrm{scal}}
\newcommand{\p}[1]{{\left( #1\right)}}
\newcommand{\g}[2]{\left\langle #1,#2\right\rangle}
\newcommand{\R}{\mathbb{R}}
\newcommand{\mfg}{\mathfrak{g}}
\newcommand{\mfh}{\mathfrak{h}}
\newcommand{\mfk}{\mathfrak{K}}
\newcommand{\mcb}{\mathcal{B}}

\newcommand{\eps}{\varepsilon}
\newcommand{\Lap}{\Delta}

\newtheorem{definition}[equation]{Definition}
\newtheorem{example}[equation]{Example}
\newtheorem{proposition}[equation]{Proposition}
\newtheorem{theorem}[equation]{Theorem}
\newtheorem{corollary}[equation]{Corollary}
\newtheorem{lemma}[equation]{Lemma}
\newtheorem{remark}[equation]{Remark}
\newtheorem{claim}[equation]{Claim}

\title{Static Vacuum Spacetimes with $\Lambda<0$ as Attractors of the Ricci-Harmonic Flow}
\author{Rasmus Jouttij\"arvi, Klaus Kr\"oncke, and Louis Yudowitz}
\date{\today}

\begin{document}
\maketitle

\begin{abstract}
    We prove dynamical stability and instability theorems for asymptotically hyperbolic static solutions of Einstein's equation with $\Lambda<0$, viewed as self-similar solutions of the Ricci-harmonic flow. More precisely, we show that static metrics are dynamically stable if and only if a positive mass type theorem holds for nearby metrics. Our key tool is a new variant of the expander entropy for the Ricci-harmonic flow.
\end{abstract}

\tableofcontents

\section{Introduction}\label{intro}
    The Ricci flow is a weakly parabolic evolution equation for Riemannian metrics on a manifold, and has been a major tool for solving numerous open conjectures in Riemannian geometry over the past few decades. There are also various extensions and variants of the Ricci flow in the literature. For possible applications in general relativity, Bernhard List  introduced in \cite{list_paper} the extended Ricci flow system
    \begin{align}\label{eq:Listflow}
        \begin{cases}
            \partial_t g(t) &= -2\Ric_{g(t)} + 2\alpha \cdot d\varphi(t) \otimes d\varphi(t)\\
            \partial_t \varphi(t) &= -\Delta \varphi(t),
        \end{cases}
   \end{align}

    \noindent where $\alpha$ is a fixed coupling constant. Here, the evolution of the family of Riemannian metrics $g(t)$ is coupled to the evolution of a time-dependent family of real-valued functions $\varphi(t)$. This extended system is usually called the \textbf{Ricci-harmonic flow} in the literature.
    While one of the motivations for studying the Ricci flow is the possible discovery of new Ricci-flat metrics as limits when $t\to\infty$, the Ricci-harmonic flow has been introduced as a possible tool to find new static solutions of Einstein's vacuum equations. These are solutions to
    \begin{align}\label{eq:static_vac_eq}
        \begin{cases}
            \nabla^2V=V\Ric_g\\ \Delta V=0,
        \end{cases}
    \end{align}
    where $g$ is a Riemannian metric and $V$ is a positive function on a manifold $M$.
    We call $(M,g,V)$ a \textbf{static triple} and $\p{g,V}$ a \textbf{static pair}. Recall that \eqref{eq:static_vac_eq} is equivalent to the statement that the static Lorentzian metric $h=-V^2dt^2+g$ is Ricci-flat. Note that positivity of $V$ ensures that $h$ is well-defined as a Lorentzian metric. By the substitution $\varphi=-\ln{V}$, one sees that a static triple is a self-similar solution of the Ricci-harmonic flow \eqref{eq:Listflow}.
    A further generalisation, where the Ricci flow is coupled to the harmonic map heat flow for manifold-valued functions, has been introduced by M\"{u}ller in \cite{reto_paper}. Our attention, however, will be focused on the case of real-valued functions.\\
                
    \noindent In this paper,  we will study a normalised variant of the Ricci-harmonic flow which is given by
    \begin{align}\label{eq:Listflow_withLambda}
        \begin{cases}
            \partial_t g(t) &= -2\p{\Ric_{g(t)} + n g(t) - d\varphi(t) \otimes d\varphi(t)}\\
            \partial_t \varphi(t) &= -\p{\Delta \varphi(t) - n},
        \end{cases}
    \end{align}
    where $n$ is the dimension of the underlying manifold. Note that the flows \eqref{eq:Listflow} and \eqref{eq:Listflow_withLambda} are homothetically equivalent to each other. That is, solutions of these two equations can be related to each other by suitable rescaling of the components and the time variable.

    Our motivation to study \eqref{eq:Listflow_withLambda} is finding static solutions of \textbf{Einstein's equations} with cosmological constant $\Lambda<0$. By imposing the normalisation $\Lambda=-n$, it is given by 
    \begin{align}\label{eq:static_eq}
        \begin{cases}
            \nabla^2V=\p{\Ric_g+ng}V\\ 
            \Delta V=-nV,
        \end{cases}
    \end{align}
    where the function $V$ is again required to be positive. These equations are equivalent to demanding $\Ric_h=-n h$ for the static Lorentzian metric $h=-V^2dt^2+g$.
    Similarly to the previous case, one can check by substituting $\varphi=-\ln{V}$ that a static triple $(M,g,V)$ with cosmological constant $\Lambda=-n$ is a self-similar solution of the normalised Ricci-harmonic flow \eqref{eq:Listflow_withLambda}. See the discussion in Section \ref{static_metric_sec} for more details and perspectives.

    The static equations \eqref{eq:static_eq} with $\Lambda=0$ are quite rigid. A classic result of Bunting and Masood-ul-Alam (\cite{BM}) states that Euclidean space together with a constant positive function is the only solution of \eqref{eq:static_eq} such that $(M,g)$ is a complete asymptotically flat manifold without boundary. If one allows a boundary, the only other solution is the Schwarzschild family of black holes.
    On the other hand, the solution space for $\Lambda<0$ exhibits a much richer structure. In \cite{ACD}, Anderson, Chruściel, and Delay constructed an infinite-dimensional family of nontrivial asymptotically hyperbolic static triples with $\Lambda=-n$.
    
    The purpose of this article is to determine under which conditions such static triples are dynamically stable as self-similar solutions of  \eqref{eq:Listflow_withLambda}. 
    In doing so, we will reveal an interesting and deep connection to the positive mass theorem. Such a relation was conjectured by Ilmanen for the ADM-mass in the ALE setting and partly confirmed in \cite{HHS}. See also \cite{DO}. In their recent work \cite{KY_PE_stability}, the second and third authors showed that dynamical stability of Poincar\'e-Einstein manifolds is equivalent to a local positive mass theorem involving a new mass-like invariant introduced by Dahl, McCormick and the second author in \cite{DKM}. The present article is inspired by these results.

    In our setting, we use a natural definition of mass associated to a static triple, following the approach in Chruściel-Herzlich \cite{CH}, see also \cite{michel}.
    Let $(M,\wh{g},V)$ be a static triple with $\Lambda=-n$. Furthermore, assume that the Riemannian manifold $(M,\wh{g})$ is asymptotically hyperbolic, in the sense that it is conformally compact and the sectional curvature converges to $-1$ at the conformal boundary. See Section \ref{sec_AH_metrics} for a more formal statement.
    For a Riemannian manifold $(M,g)$ such that the difference $h:=g-\wh{g}$ decays suitably at infinity, we define its \textbf{mass} with respect to $(M,\wh{g},V)$ as
    \begin{align*}\label{eq:AH_mass}
        m_{\wh{g},V}\p{g}:=\lim_{R\to \infty}\int_{\partial B_R}V\g{\mathrm{div}_{\wh{g}}g-d\tr_{\wh{g}}g}{\wh{\nu}}_{\wh{g}}+\tr_{\wh{g}}h\; \wh{\nu}\p{V}-h\p{dV,\wh{\nu}}\; d\wh{A},
    \end{align*}
    where $\wh{\nu}$ is the outward-pointing unit normal. This definition of mass can be derived from the linearisation of the scalar curvature as follows: Let $\omega\in\Omega^1\p{M}$ be a one-form such that
    \begin{align*}
        V\cdot D_{\wh{g}}\scal(h)=\mathrm{div}_{\wh{g}}\omega +\g{(D_{\wh{g}}\scal)^*(V)}{h}_{\wh{g}},
    \end{align*}
    where $(D_{\wh{g}}\scal)^*$ is the formal adjoint of $D_{\wh{g}}\scal$. Then the integrand defining $m_{\wh{g},V}\p{g}$ is simply $\omega(\nu)$. Note furthermore that the static potential satisfies $(D_{\wh{g}}\scal)^*(V)=0$ so that Taylor expanding yields
    \begin{align*}
        V(\scal_g-\scal_{\wh{g}})=V\cdot D_{\wh{g}}\scal(h)+V\cdot \mathcal{O}_2(\abs{h}^2)=\mathrm{div}_{\wh{g}}\omega +V\cdot \mathcal{O}_2(\abs{h}^2),
    \end{align*}
    where $u = \mathcal{O}_2\p{\abs{h}^2}$ means $u$ and its first two derivatives have such decay. This shows that $m_{\wh{g},V}\p{g}$ is finite if both $V(\scal_g-\scal_{\wh{g}})$ and the remainder term $V\cdot \mathcal{O}_2(\abs{h}^2)$ are integrable. In the asymptotically hyperbolic setting, the static potential $V$ usually behaves like $e^r$ at infinity, so that the mass is finite if $\abs{g-\wh{g}}=\mathcal{O}_2\p{e^{-\frac{n}{2}+\varepsilon}}$ for some $\varepsilon > 0$  and $(\scal_g-\scal_{\wh{g}})e^r\in L^1\p{M}$.
    \begin{remark}
        It is straightforward to generalise this definition to metrics $g$ on different manifolds, provided they are asymptotic to $\wh{g}$ through a diffeomorphism at infinity. Here, we are always working with metrics on a fixed manifold, so we omit this case.
    \end{remark}
    \begin{remark}
        Later in this paper, we consider the mass as an expression in $\varphi=-\ln{V}$ and $h=g-\wh{g}$ for practical purposes. Furthermore, we prefer to work with the sign convention of Besse's book \textnormal{\cite{besse}} for the divergence and set $\delta_{\wh{g}}=-\mathrm{div}_{\wh{g}}$. By slight abuse of notation, we will denote the mass with these substitutions by $m_{\wh{g}}\p{h,\varphi}$ from Section \ref{sec_entropy} onwards.
    \end{remark}
    Recall that, on hyperbolic space $\p{\mathbb{H}^n,\wh{g}=g_{\mathbb{H}^n}}$, the kernel of $(D_{\wh{g}}\scal)^*$ is $(n+1)$-dimensional, so that we get an $(n+1)$-dimensional space of static potentials. The kernel can be obtained by embedding $\mathbb{H}^n$ isometrically into Minkowski space $\R^{1,n}$ and restricting the position vector field $\vec{x}=(x^0,x^1,\ldots, x^n)$ to $\mathbb{H}^n$. The restrictions of the component functions $x^i$ span the kernel of $(D_{\wh{g}}\scal)^*$. For this reason, the mass with respect to the hyperbolic metric is defined as an $\p{n+1}$-dimensional vector $\vec{m}_{\wh{g},\vec{x}}\p{g}$, whose components are given by
    \begin{align*}
        m_{\wh{g},x^i}\p{g},\qquad 0\leq i\leq n.
    \end{align*}
    
    Interpreted as a vector in $\R^{1,n}$, the positive mass theorem by Chrusciel-Herzlich in \cite{CH} asserts that if the metric $(M,g)$ is a complete spin manifold, is asymptotic to $\wh{g}=g_{\mathbb{H}^n}$, and satisfies the inequality $\scal_g\geq-n(n-1)$, then the vector $\vec{m}_{\wh{g},\vec{x}}\p{g}$ is future directed timelike or zero, and the latter case occurs if and only if $g$ is isometric to $\wh{g}$.
    To align this statement with our framework, we note that it is equivalent to the following assertion: for each positive function $V\in \ker((D_{\wh{g}}\scal)^*)$, we have ${m}_{\wh{g},V}\p{g}\geq 0$ and equality holds if and only if $g$ is isometric to $\wh{g}$.
    
    For a general static triple $(M,\wh{g},V)$, one cannot expect such a positive mass theorem. On a fixed manifold $M$, one might, however, have a local positive mass property in the following sense:
    \begin{definition}
       Let $(\wh{g},V)$ be an asymptotically hyperbolic static pair. We say that it admits the \textbf{local positive mass property} if for all nearby metrics $g$ with $\scal_g \geq -n\p{n-1}$, we have 
        \begin{align*}
            {m}_{\wh{g},V}\p{g}\geq 0.
        \end{align*}
   \end{definition}
   
    We are going to show that this property is equivalent to dynamical stability of the pair $(\wh{g},-\ln{V})$ as a self-similar solution of the Ricci-harmonic flow. This link is generated by a new entropy functional $\mu_{\mathrm{AH},\wh{g}}(g,\varphi)$ for the pair $(g,\varphi)$, which we define in Section \ref{sec_entropy_defn}. Up to diffeomorphism, the Ricci-harmonic flow can be interpreted as a gradient flow of this entropy (in a weighted $L^2$-sense).
    In the asymptotically hyperbolic setting, subtracting the boundary integral defining ${m}_{\wh{g},V}$ is the right normalisation at infinity to make the entropy well-defined. Such a normalisation has been used in other geometric situations (see \cite{DKM,DO}), in order to define suitable entropies. One of the delicate tasks to deal with here is the fact that the function $\varphi$ will not be bounded, but converge to $-\ln{V}$ at a suitable rate.
    An important cornerstone for the main statements of this paper is the following result which we state roughly here and more precisely in Section \ref{sec_local_pos_mass}:
    \begin{theorem}\label{mainthm_local_positive_mass}
        Let $\left(M,\wh{g},V\right)$ be an asymptotically hyperbolic static triple of dimension $n \geq 3$. Then the following are equivalent:
        \begin{itemize}
            \item[(i)] The pair $(\wh{g},-\ln{V})$ is a local maximiser of $\mu_{\mathrm{AH},\wh{g}}$,
            \item[(ii)] The pair $(\wh{g},V)$ satisfies the local positive mass property.
        \end{itemize}
    \end{theorem}

    We also mention that while static metrics each give rise to an $\left(n+1\right)$-dimensional Poincar\'e--Einstein metric, thereby reducing some of our proofs to those in previous work by the second and third authors (\cite{KY_PE_stability}), many estimates turn out to require far more work. This is due to the need to ensure the warped product structure of the $\left(n+1\right)$-dimensional Poincar\'e--Einstein metric is maintained alongside certain curvature conditions. This is a delicate matter and ended up requiring a surprising amount of work compared with the general Poincar\'e--Einstein case.    
    
    Rough statements of the main two results of this article are as follows.
    \begin{theorem}\label{mainthm_stability} (Dynamical stability)
        Let $\left(M^n,\wh{g},V\right)$ be an asymptotically hyperbolic static triple of dimension $n \geq 3$. If $(\wh{g},-\ln{V})$ is a local maximiser of $\mu_{\mathrm{AH},\wh{g}}$, then
        for every $L^2 \cap L^{\infty}$-neighbourhood $\mathcal{U}$ of $(\wh{g},-\ln{V})$, there exists an $H^\ell$-neighbourhood $\mathcal{V} \subset \mathcal{U}$ (with $\ell > \frac{n}{2} + 2$)
        such that the Ricci flow starting at any metric in  $\mathcal{V}$ exists for all times and converges polynomially
        (modulo diffeomorphisms)  to a static pair in $\mathcal{U}$.
    \end{theorem}
    
    \begin{theorem}\label{mainthm_instability} (Dynamical instability)
        Let $\left(M^n,\wh{g},V\right)$ be an asymptotically hyperbolic static triple of dimension $n \geq 3$. If $(\wh{g},-\ln{V}) =: \left(\wh{g}, \wh{\varphi}\right)$ is not a local maximiser of $\mu_{\mathrm{AH},\wh{g}}$, then there exists a non--trivial ancient Ricci--harmonic flow $\left\{(g(t),\varphi(t))\right\}_{t\in\left(-\infty,0\right]}$ that converges (modulo diffeomorphisms) to $(\wh{g},\wh{\varphi})$ as $t \to -\infty$.
    \end{theorem}

    We refer the reader to Section \ref{stability_sec} for the proofs, as well as more detailed and formal statements of the results. Observe that the converse implications also hold due to the monotonicity of $\mu_{\mathrm{AH},\wh{g}}$ along the flows. In particular, each static metric is either dynamically stable or dynamically unstable, and this entirely depends on the local behaviour of $\mu_{\mathrm{AH},\wh{g}}$.
    
    Summarising and combining these results and the above discussion, we establish the following equivalences for static solutions:
    \begin{align*}
        \text{dynamical stability}&\Leftrightarrow \text{positive mass property for nearby metrics},\\
        \text{dynamical instability}&\Leftrightarrow \text{failure of the positive mass property for nearby metrics}.
    \end{align*}
    
    \begin{example}
        Hyperbolic space is well-known to be strictly linearly stable in the sense that the second variation of the entropy is strictly negative in directions orthogonal to diffeomorphisms. 
        Independently, the positive mass property follows from the positive mass theorem in \textnormal{\cite{CH}}.
        There is also an infinite-dimensional family of static triples discovered by  Anderson-Chruściel-Delay \textnormal{\cite{ACD}}. These examples arise as static fillings for conformal boundaries close to the round sphere. By continuity, the second variation will still be negative for sufficiently small perturbations, which implies the positive mass property for nearby metrics and dynamical stability  due to the results of this article.
    \end{example}
    
    An essential ingredient of the dynamical stability and instability theorems is the following \L ojasiewicz-Simon inequality for  $\mu_{\mathrm{AH},\wh{g}}$:
    \begin{theorem}\label{mainthm_loj_ineq}
        Given a static triple $\left(M^n,\wh{g},V\right)$ with $n \geq 3$, for $\ell > \frac{n}{2} + 2$ there is an $H^\ell$-neighbourhood $\mathcal{U}$ of $\p{\wh{g},\wh{\varphi}} := \p{\wh{g}, -\ln{V}}$, a positive constant $C > 0$, and an exponent $\theta \in \left(0,1\right]$ such that for all $\p{g, \varphi}\in \mathcal{U}$, we have the inequality
        \begin{equation*}
           \abs{\mu_{\mathrm{AH},\wh{g}}\p{g, \varphi} - \mu_{\mathrm{AH},\wh{g}}\p{\wh{g}, \wh{\varphi}}}^{2-\theta} \leq C\norm{\nabla\mu_{\mathrm{AH},\wh{g}}\p{g, \varphi}}^2_{L^2_{\wh{\varphi}}}.
        \end{equation*}
    \end{theorem}
    The way Theorem \ref{mainthm_stability} and Theorem \ref{mainthm_instability} are deduced from Theorem \ref{mainthm_loj_ineq} follows the approach of the second and third authors' earlier work in \cite{KY_PE_stability} for Poincar\'e-Einstein metrics. In general, the stability proofs in the asymptotically hyperbolic situation are remarkably similar to previous proofs for compact manifolds without boundary.\\
    
    The paper is structured as follows. Section \ref{prelims} contains basic definitions and results about asymptotically hyperbolic metrics and static metrics. We also introduce the weighted spaces we will use throughout the paper and present an argument which will show that, under certain decay assumptions on a non-static metric, we can associate to it a generalised 'potential'. In Section \ref{sec_entropy}, we introduce our expander entropy and show it is well-defined and finite. In Section \ref{sec_variations} we compute the first and second variations of the entropy, as well as some related results which include a derivation of the associated stability operator and monotonicity of the entropy under the Ricci-harmonic flow. Section \ref{sec_local_pos_mass} is devoted to the proof of our local positive mass result (Theorem \ref{mainthm_local_positive_mass}). Section \ref{stability_sec} contains the proofs of the \L ojasiewicz--Simon inequality (Theorem \ref{mainthm_loj_ineq}), as well as the stability (Theorem \ref{mainthm_stability}) and instability (Theorem \ref{mainthm_instability}) results. Appendix \ref{appendix} contains some variational formulas and results we rely on throughout the paper.
     
    \vspace{3mm}
    \noindent\textbf{Acknowledgements.} The authors would like to thank the G\"oran Gustafsson Foundation for financial support.

\section{Preliminaries}\label{prelims}
        This section contains some basic definitions, results, and identities we will rely on throughout the paper. In particular, we formally introduce asymptotically hyperbolic metrics and static metrics, as well as relationships between $n$-dimensional static metrics and the associated $\left(n+1\right)$-dimensional negative Einstein metrics.
    
    \subsection{Asymptotically Hyperbolic Metrics}\label{sec_AH_metrics}
        \begin{definition} 
            Let $N^n$ be a compact manifold with boundary and set $M=N\setminus \partial N$. A function $\rho:N\to [0,\infty)$ is called \textbf{boundary defining}, if 
                $$\rho|_M>0,\qquad \rho|_{\partial N}=0,\qquad d\rho|_{\partial N}\neq 0$$
        \end{definition}

        \begin{definition}
            $(M,g)$ is called \textbf{conformally compact of class $C^{\ell, \alpha}$} if there exists a $C^{\ell,\alpha}$-metric $\overline{g}$ on $N$ and a boundary defining function $\rho$, such that $g=\rho^{-2}\overline{g}$.  If $\abs{d\rho}^2_{\overline{g}} \cong 1$ on $\partial N$, $g$ is called \textbf{asymptotically hyperbolic (AH)}. We call an Einstein AH metric \textbf{Poincar\'e--Einstein (PE)}. Note this is necessarily a negative Einstein metric as the sectional curvatures converge to $-1$ at spatial infinity.
        \end{definition}

        From now on we will implicitly assume the AH manifolds we consider are of class $C^{\ell,\alpha}$ for $\ell > \frac{n}{2} + 2$.
    
        \begin{lemma}
            Let $g$ be an AH metric with boundary defining function $\rho$ and compactified metric $\overline{g} := \rho^2g$. Then
            \begin{equation}\label{rhodeltarho}
                \rho\Delta_{\overline{g}}\rho=-\frac{\scal_g+n(n-1)\abs{d\rho}_{\overline{g}}^2}{2(n-1)}+\frac{\rho^2\scal_{\overline{g}}}{2(n-1)}
            \end{equation}
        \end{lemma}

        \begin{proof}   
            Using a standard formula for the Ricci curvature of conformally related metrics (see, for instance, \cite[Lemma 2.1]{GL}), we have
            \begin{align*}
                \Ric_g ={}&-(n-1)\rho^{-2}\abs{d\rho}^2_{\overline{g}}\cdot \overline{g}-\rho^{-1}\p{\Delta_{\overline{g}} \rho\cdot \overline{g}-(n-2)\nabla^2_{\overline{g}}\rho}+\Ric_{\overline{g}}.
            \end{align*}
        
            Taking the trace with respect to $g=\rho^{-2}\overline{g}$ yields \eqref{rhodeltarho}.
        \end{proof}

        \begin{remark}\label{distance_function}
            Note that, if $\rho$ is a boundary defining function on an AH manifold $\left(M,g\right)$, then $r = -\ln\p{\rho}$ is equivalent to a distance function on $\left(M,g\right)$. The volume form of an AH metric satisfies $dV_g\sim e^{(n-1)r}\; dr\;dA_{\overline{g}_r}$, where $dA_{\overline{g}_r}$ is the area form of the compactified metric on $\partial B_r$. As a consequence, a function $f\sim e^{-\beta r}$ is $L^2$-integrable if and only if $\beta>\frac{n-1}{2}$.
        \end{remark}

        \begin{remark}
            Let $g$ be an AH metric with boundary defining function $\rho$. Let $\varphi:M\to \R$ be a function satisfying 
                $$\abs{\varphi-\ln\rho}\to 0\qquad \&\qquad \abs{d\varphi}_g\to 1$$
            as one approaches $\partial N=\rho^{-1}(0)$. Then $\rho_\varphi:=e^{\varphi}$ is another boundary defining function, with the same conformal infinity:
            $$\p{\rho_\varphi^2g|_{\partial N},\partial N}=\p{\overline{g}|_{\partial N},\partial N}.$$
        \end{remark}

        \begin{definition}\label{aux_metric_defn}
            For a pair $(g,\varphi)$ as above, we define the associated \textbf{auxiliary metric} by 
                $$\mathfrak{g}:=e^{-2\varphi}\; d\tau^2+g\qquad \text{on}\qquad \mathfrak{M}=\mathbb{S}^1\times M,$$
            where $\varphi$ is extended trivially along the $\mathbb{S}^1$ factor. The auxiliary metric is itself AH, with boundary defining function $\rho_\varphi$ and conformal infinity 
                $$\p{\rho_\varphi^2\mathfrak{g}|_{\mathbb{S}^1\times \partial N},\mathbb{S}^1\times \partial N}=\p{d\tau^2+\overline{g}|_{\partial N},\mathbb{S}^1\times \partial N}.$$

            We will call a function/$1$-form/symmetric $2$-tensor \textbf{$\tau$-independent} if its restriction to the $\mathbb{S}^1$ factor is constant. 
        \end{definition}

        The importance of the final part of Definition \ref{aux_metric_defn} is that throughout the paper we will switch between viewing functions, forms, and tensors as either living on M or as $\tau$-independent objects living on $\mathfrak{M}$, depending on what is most convenient.        

    \subsection{Static Metrics and Weighted H\"older and Sobolev Spaces}\label{static_metric_sec}
        We now formally define the main objects we are interested in: static metrics.
        
        \begin{definition} 
            Let $\p{M^n,\wh{g}}$ be an asymptotically hyperbolic Riemannian metric. We call the metric $\wh{g}$ \textbf{static}, if there exists a non-trivial solution $V$ of
            \begin{equation}\label{static1}
                \p{D_{\wh{g}}\scal}^*\p{V}=\p{\Delta_{\wh{g}} V} \wh{g}+\nabla^2_{\wh{g}}V-V \Ric_{\wh{g}}=0.
            \end{equation}

            A static metric has constant scalar curvature (see Remark \ref{static_csc} below). In particular, $\scal_{\wh{g}}=-n(n-1)$ since $\wh{g}$ is asymptotically hyperbolic. Therefore, \eqref{static1} is equivalent to the system
                $$\begin{cases}
                \nabla^2_{\wh{g}}V=V\p{\Ric_{\wh{g}}+n\wh{g}}\\ \Delta_{\wh{g}} V=-nV.
                \end{cases}$$
    
            We call $V$ a \textbf{static potential} for $\wh{g}$, $\left(\wh{g}, V\right)$ a \textbf{static pair}, and $\left(M,\wh{g},V\right)$ a \textbf{static triple}.
        \end{definition}
        
            If the potential $V$ is positive near the conformal boundary, it is positive everywhere by a maximum principle argument. It then makes sense to consider the equations of staticity under the transformation $\wh{\varphi} = -\ln{V}$:
            \begin{equation}\label{static_equations}
                \begin{cases}
                    \nabla^2_{\wh{g}}\wh{\varphi}-d\wh{\varphi}\otimes d\wh{\varphi}+\Ric_{\wh{g}}+n\wh{g}=0\\ \Delta_{\wh{g}} \wh{\varphi}+\abs{d \wh{\varphi}}_{\wh{g}}^2=n.
                \end{cases}
            \end{equation}
            
            The reader should be aware, that when we refer to $\wh\varphi$ as a static potential, $\p{\wh{g},\wh\varphi}$ a static pair, or $\p{M,\wh{g},\wh\varphi}$ as static triple, it should always be interpreted as the potential $V=e^{-\wh{\varphi}}$ in the above sense.

        \begin{remark}\label{static_ricci}
          If $(\wh{g},V)$ is a static pair, both the  $(n+1)$-dimensional Lorentzian metric 
                $$h=-V^2\; d\tau^2+\wh{g}$$
                as well as the $(n+1)$-dimensional Riemannian metric 
                      $$h=V^2\; d\tau^2+\wh{g}$$
            solve the equation $\Ric_h=-nh$. In particular, the latter metric is Poincaré-Einstein.
        \end{remark}

        \begin{remark}\label{static_csc} 
            Taking the divergence of \eqref{static1}, we obtain
            $$0=-\nabla_{\wh{g}}\Delta_{\wh{g}} V+\Delta_{\wh{g}} \nabla_{\wh{g}} V+\Ric_{\wh{g}}\p{ \nabla_{\wh{g}} V,\cdot}-V \delta_{\wh{g}}\Ric_{\wh{g}}=\frac{V}{2}\nabla_{\wh{g}} \scal_{\wh{g}},$$
            where the second equality is a result of using the Ricci formula and the contracted Bianchi identity. It follows that a static metric must have constant scalar curvature, at least on each connected component of $M$.             
        \end{remark}

        \begin{example} 
                        Consider the standard hyperbolic metric on $\mathbb{H}^n$, written as
                $$g_{\mathbb{H}^n}=dr^2+\sinh^2{r}\; ds^2_{n-1},$$
            with $ds_{n-1}^2$ the round metric on $\mathbb{S}^{n-1}$. In these coordinates, a static potential for $g_{\mathbb{H}^n}$ is
                $$V_{\mathbb{H}^n}:= 2\cosh{r}.$$
                        Note that this means 
            \begin{equation*}
                \varphi_{\mathbb{H}^n} = -\ln\p{V_{\mathbb{H}^n}}=-\ln\p{e^r+e^{-r}}=-r+\ln\p{1+e^{-2r}}=-r+\mathcal{O}(e^{-2r})
            \end{equation*}
            as $r \rightarrow \infty$.             
            
            The corresponding static solution to Einstein's equations is the $(n+1)$-dimensional Anti-de Sitter metric. In Riemannian form, using the same coordinates as above, we have
                $$g_{AdS}= V_{\mathbb{H}^n}^2\; d\tau^2+g_{\mathbb{H}^n}=dr^2+\cosh^2{r}\; d\tau^2+\sinh^2{r}\; ds_{n-1}^2.$$
        \end{example}

            Motivated by the above example, from now on we will impose that \textbf{the potentials $\varphi$ we will consider are asymptotic to $-r$ and of regularity $C^{\ell,\alpha}$ (at least)}, regardless of whether they are associated to a static metric or not. 
        
            Moreover, we will often work with weighted H\"older and Sobolev spaces. In particular, proofs in future sections will involve working with $C^{\ell,\alpha}_\beta\p{M} = e^{-\beta r} C^{\ell,\alpha}\p{M}$ equipped with the norm      
            \begin{equation*}
                \norm{u}_{C^{\ell,\alpha}_\beta\p{M}} =    \norm{e^{\beta r} u}_{C^{\ell,\alpha}\p{M}},
            \end{equation*}
            where $\beta \in \mathbb{R}$ and $C^{\ell,\alpha}\p{M}$ is the usual H\"older space. As for weighted Sobolev spaces, the only weight we shall consider is $\frac{1}{2}$. However, as $\varphi$ is always asymptotic to $-r$, we shall use the weight factor $e^{\frac{\varphi}{2}}$ in place of $e^{-\frac{r}{2}}$. Thus we define $H^\ell_\varphi\p{M}$ to be the standard $L^2$-Sobolev space, defined with respect to the weighted volume measure $e^{-\varphi}dV_g$. Weighted H\"older and Sobolev spaces for sections of bundles are then defined as in, for instance, \cite{Lee}. We also have the following version of the Sobolev embedding theorem for our weighted spaces (see \cite[Lemma 3.6]{Lee}):

        \begin{lemma}\label{wgtd_sob_emb}
            If $\alpha \in \left(0,1\right)$ and $E$ is a vector bundle over $M$, then we have a continuous inclusion
            \begin{equation*}
                H^\ell_\varphi\p{M; E} \hookrightarrow C^{2,\alpha}_{\frac{1}{2}}\p{M;E}
            \end{equation*}
            provided $\ell > \frac{n}{2} + 2$.
        \end{lemma}
        
            We shall usually omit the manifold $M$, as exemplified immediately below, as long as we deem it clear from context.\\
        
            We will denote the space of metrics that are $C^{\ell,\alpha}_\beta$-close to a static background metric $\wh{g}$ by
            \begin{equation*}
                \mathcal{R}^{\ell,\alpha}_\beta\p{M,\wh{g}} := \left\{g : g - \wh{g} \in C^{\ell,\alpha}_\beta\p{S^2M}\right\},
            \end{equation*}
            and if $\wh\varphi$ is a static potential for $\wh{g}$, we shall denote the space of metrics that are $H^\ell_{\wh{\varphi}}$-close to $\wh{g}$ by
            \begin{equation*}
                \mathcal{R}^\ell_{\wh{\varphi}}\p{M,\wh{g}} := \left\{g : g - \wh{g} \in H^\ell_{\wh{\varphi}}\p{S^2M}\right\}.
            \end{equation*}

            Moreover, we will sometimes write the space of functions that are $H^\ell_{\wh{\varphi}}\p{M}$-close to $\wh{\varphi}$ as
            \begin{equation*}
                \mathcal{R}^\ell_{\wh{\varphi}}\p{M, \wh{\varphi}} := \left\{\varphi : \varphi - \wh{\varphi} \in H^\ell_{\wh{\varphi}}\p{M}\right\}.
            \end{equation*}

    \subsection[Static Metrics and the Corresponding Negative Einstein Metrics]{Static Metrics and the Corresponding $(n+1)$-Dimensional Negative Einstein Metrics}\label{sec:static_metrics_etc}

            Throughout this article, we will consider pairs $\p{g,\varphi}$, where $g$ is an AH Riemannian metric on an $n$-dimensional manifold $M$, and $\varphi$ is an \textbf{approximate potential} for $g$ in the sense of Lemma \ref{fefferman_graham} below. The lemma gives more precise asymptotic behaviour than simply being asymptotic to $-r$.\\

            Dual to this picture of a pair $\p{g,\varphi}$ is the $(n+1)$-dimensional \textbf{auxiliary manifold} 
                $$\p{\mathfrak{M}=\mathbb{S}^1\times M,\mathfrak{g}}\qquad \text{with}\qquad \mathfrak{g}=e^{-2\varphi}\; d\tau^2+g, $$
            where $\tau\in \mathbb{S}^1=\R/2\pi \mathbb{Z}$. The added dimension has to be compactified, so as to make $\mathfrak{g}$ asymptotically hyperbolic, but nowhere is it essential that the length of the circle is chosen to be $2\pi$.\\

            The non-vanishing Christoffel symbols of $\mathfrak{g}$ satisfy
                $$\overline{\Gamma}_{\tau \tau}^{i}=e^{-2\varphi}g^{ij}d\varphi_j,\qquad \overline{\Gamma}_{\tau i}^\tau = -d\varphi_i,\qquad \overline{\Gamma}_{ij}^k=\Gamma_{ij}^k\qquad\qquad \p{i,j,k\neq \tau},$$
            where $\overline{\Gamma}_{ij}^k$ and $\Gamma_{ij}^k$ denote the Christoffel symbols of $\mathfrak{g}$ and $g$, respectively. The Ricci curvature may now be calculated. For the $\tau\tau$ component, we have
            \begin{align*}
                \Ric^{\mathfrak{g}}_{\tau \tau} ={}& \partial_i \overline{\Gamma}_{\tau \tau}^{i}+\Gamma_{im}^{i}\overline{\Gamma}_{\tau \tau}^m-\overline{\Gamma}_{i\tau}^\tau \overline{\Gamma}_{\tau \tau}^{i}\\ 
                ={}&\nabla_i\p{e^{-2\varphi} g^{ij}d\varphi_j}+e^{-2\varphi}g^{ij}d\varphi_id\varphi_j\\={}&e^{-2\varphi}\p{g^{ij}\nabla^2_{ij}\varphi-g^{ij}d\varphi_id\varphi_j}\\ 
                ={}&-e^{-2\varphi}\p{\Delta_g\varphi+\abs{d\varphi}_g^2}.
            \end{align*}
            For $i,j\neq \tau$ we find that
            \begin{align*}
                \Ric^{\mathfrak{g}}_{ij}={}&\Ric_{ij}^g-\partial_i\overline{\Gamma}_{j\tau}^\tau+\Gamma_{ij}^k\overline{\Gamma}_{k\tau}^\tau-\overline{\Gamma}_{\tau i}^\tau\overline{\Gamma}_{j\tau}^\tau\\
                ={}&\Ric_{ij}^g-\nabla_i\overline{\Gamma}_{j\tau}^\tau-\overline{\Gamma}_{\tau i}^\tau\overline{\Gamma}_{j\tau}^\tau\\ ={}&\Ric_{ij}^g+\nabla^2_{ij}\varphi-d\varphi_id\varphi_j.
            \end{align*}
            Written invariantly, we have
            \begin{equation}\label{auxiliary_Ric}
                \Ric_{\mathfrak{g}}=-\p{\Delta_g\varphi+\abs{d\varphi}_g^2}e^{-2\varphi}\; d\tau^2+\Ric_g+\nabla^2_g\varphi-d\varphi\otimes d\varphi.
            \end{equation}
            Taking the trace with respect to $\mathfrak{g}$, we obtain 
            \begin{equation*}
                \scal_{\mathfrak{g}}=\scal_g-2\Delta_g\varphi-2\abs{d\varphi}_g^2.
            \end{equation*}

            One thus sees, as in Remark \ref{static_ricci} and Remark \ref{static_csc}, that if $\p{g, \varphi}$ is a static pair then $\Ric_{\mathfrak{g}} = -n\mathfrak{g}$ and $\scal_{\mathfrak{g}} = -n\p{n+1}$.   

        \subsubsection*{Identities}{\label{aux_identities}}
            We now record some important identities relating some standard differential operators on $M$, with the corresponding operators on $\mathfrak{M}$, acting on $\tau$-independent objects.\\
                
            For $\abs{s} \ll 1$, let $g_s$, $\varphi_s$ be families of, respectively, AH metrics and approximate potentials with associated auxiliary metrics $\mathfrak{g}_s$. Next, we define $h:=\left.\frac{d}{ds}\right|_{s=0}g_s$, $\psi:=\left.\frac{d}{ds}\right|_{s=0}\varphi_s$, and $\mathfrak{h}:=\left.\frac{d}{ds}\right|_{s=0}\mathfrak{g}_s$.\\
            
            We shall simultaneously view $\mfh$ as a section of $S^2M\times C^\infty(M)$ and a section of $S^2\mathfrak{M}$. Consequently, we write
                \begin{equation}\label{aux2tensor}\mathfrak{h} = \p{h,\psi} =-2\psi e^{-2\varphi}\; d\tau^2+h.\end{equation}

            Observe, in particular, the factor $-2$ on the first term, as it plays a role in how we interpret the norm and vector product of such pairs.  

            \begin{enumerate}
            \item Consider $\omega \in \Omega^1\p{M}$ and view it as a $\tau$-independent $1$-form on $\mathfrak{M}$. The divergences with respect to $g$ and $\mathfrak{g}$ are related by
                $$\delta_{\mathfrak{g}}\omega=\delta_g\omega+\g{d\varphi}{\omega}.$$
        
            The right hand side will be denoted by $\delta_{\varphi}$ when we want to emphasise the role of $\varphi$. We remark that $\p{\delta_\varphi}^*=d$, with respect to the auxiliary/weighted volume $dV_\mfg=e^{-\varphi}\; dV$.
                    
            \item Let $f:M\to \R$ and consider it as a $\tau$-independent function on $\mathfrak{M}$. The Laplacian of $\mathfrak{g}$ applied to $f$ is the drift Laplacian  
            \begin{equation}\label{aux_lap}
                \Delta_{\mathfrak{g}}f=\Delta_g f+\g{d\varphi}{df},
            \end{equation}
            which shall also be denoted by $\Delta_\varphi$ when the metric $g$ is clear from the context.
            
            \item Let $\mathfrak{h}$ be as above. Then     
            \begin{align}
                \tr_{\mathfrak{g}}\mathfrak{h}&=\tr_gh-2\psi, \label{aux_trace}\\
                \delta_{\mathfrak{g}}\mathfrak{h}&=\delta_gh+h\cdot d\varphi+2\psi\; d\varphi. \label{aux_divergence}
            \end{align}

            Here, and moving forward, $h \cdot d\varphi := h\p{d\varphi, \cdot}$. When $h$ is merely a tensor on $M$, e.g. a variation of the base metric $g$, we shall use the notation $\delta_{\varphi}h = \delta_gh + h \cdot d\varphi$ for the natural shift divergence.
                    
            \item Let $X$ be a vector field on $M$ and consider it as a $\tau$-independent vector field on $\mathfrak{M}$. We then have the following relation between codivergences:
                $$\delta^*_{\mathfrak{g}}X=-X(\varphi)e^{-2\varphi}\; d\tau^2+\delta^*_g X.$$
                    
            Specifically, if $\mathfrak{h}=\delta_{\mathfrak{g}}^*X=\frac{1}{2}\mathscr{L}_X\mathfrak{g}$, then it is  dual to the pair $\p{h,\psi}=\p{\delta^*_gX,\frac{1}{2}X(\varphi)}=\p{\frac{1}{2}\mathscr{L}_X g,\frac{1}{2}\mathscr{L}_X\varphi}$. The factor $\frac{1}{2}$ is the result of the $-2$ in the interpretation of the pair, given by \eqref{aux2tensor}.
            \end{enumerate}

        \subsubsection*{Relating Norms on $M$ and $\mathfrak{M}$}
            Later on, in the proofs of our stability and instability results, it will be useful to know H\"older and Sobolev norms on $M$ and $\mathfrak{M}$ are equivalent when dealing with $\tau$-independent objects.\\
                
            Let $f$ denote a real valued function on $M$, simultaneously viewed as a $\tau$-independent function $f:\mathfrak{M}\to \R$, where $\mathfrak{M}=\mathbb{S}^1\times M$. Let $\mathfrak{g}$ be the auxiliary metric associated to $(g,\varphi)$. Then 
            \begin{align*}
                \frac{1}{2\pi}\norm{f}_{L^2_{\mathfrak{g}}(\mathfrak{M})}^2&=\norm{f}_{L^2_\varphi\p{M}}^2\\
                \frac{1}{2\pi}\norm{f}_{H^1_{\mathfrak{g}}(\mathfrak{M})}^2&=\norm{f}_{H^1_\varphi\p{M}}^2\\
                \frac{1}{2\pi}\norm{f}_{H^2_{\mathfrak{g}}(\mathfrak{M})}^2&=\norm{f}_{H^2_\varphi\p{M}}^2+\p{df,d\varphi}^2_{L^2_\varphi\p{M}}.
            \end{align*}
            Let $\mathfrak{h}=(h,\psi)$ denote an auxiliary $2$-tensor, as in \eqref{aux2tensor}, then the pointwise norms satisfy
            \begin{align*}                    
                \abs{\mathfrak{h}}^2_{\mathfrak{g}}&=\abs{(h,\psi)}^2_{g},\\
                    \abs{\nabla_{\mathfrak{g}}\mathfrak{h}}^2_{\mathfrak{g}}&=\abs{\nabla_g \mathfrak{h}}_{g}^2 +2\abs{\mathfrak{h}\cdot d\varphi}_g^2,
            \end{align*} 
            where $\nabla_{g}\mathfrak{h}=\p{\nabla_g h, \nabla_g \psi}$,
                $$\mathfrak{h}\cdot d\varphi=h\cdot d\varphi+2\psi \;d\varphi,$$
            and $g$ should be understood as acting on a pair of tensors $(T_1,T_2)$ by
            $$\abs{\p{T_1,T_2}}^2_g=\abs{T_1}_g^2+4\abs{T_2}^2_g.$$
            
            More generally, we have 
            \begin{equation*}                \abs{\nabla^{\ell}_\mathfrak{g}\mathfrak{h}}^2_{\mathfrak{g}}=\abs{\nabla^{\ell}_g \mathfrak{h}}_{g}^2 +\sum_{i+j=0}^{\ell-1}\abs{\nabla^{i}_\mathfrak{g}\mathfrak{h}\ast \nabla^{j}_gd\varphi}_g^2,
            \end{equation*}
            where $\ast$ is allowed to express any linear combination of contractions of the tensors involved. As such contractions only influence normed inequalities by a multiplicative factor, we shall rely heavily on this notation in the latter parts of the article.
            From the pointwise equality, it follows that
            \begin{equation}\label{sobolev_norm_equiv}
                \norm{\mathfrak{h}}_{H^\ell_\varphi\p{M}}^2\leq \frac{1}{2\pi}\norm{\mathfrak{h}}_{H^\ell_{\mathfrak{g}}\p{\mathfrak{M}}}^2\leq\norm{\mathfrak{h}}_{H^\ell_\varphi\p{M}}^2+C_{\ell}\norm{\mathfrak{h}}^2_{H^{\ell-1}_{\mathfrak{g}}\p{\mathfrak{M}}} \leq D_{\ell}\norm{\mathfrak{h}}_{H^\ell_\varphi\p{M}}^2,
            \end{equation}
            where the third inequality follows from repeated application of the second, and $C_{\ell},D_{\ell}$ depend on the $C^m\p{M}$-norm of $d\varphi$ for $0\leq m\leq \ell-1$ (Recall that we implicitly assume that $\varphi\in C^{\ell,\alpha}(M)$). An analogous calculation gives
            \begin{equation*}
                \norm{\mathfrak{h}}_{C^\ell_g\p{M}}^2\leq \norm{\mathfrak{h}}_{C^\ell_\mathfrak{g}\p{\mathfrak{M}}}^2\leq E_\ell\norm{\mathfrak{h}}_{C^\ell_g\p{M}}^2,
            \end{equation*}
            where $E_\ell$ has the same dependencies as $C_\ell$ and $D_\ell$. To simplify notation moving forward, we will omit the manifolds $M$ and $\mathfrak{M}$ when writing norms, as long as it is clear from context which manifold we are working with.

        \subsubsection*{Decomposition of Symmetric $2$-Tensors}
            Let $(M,g,\varphi)$ be an AH triple, such that the auxiliary manifold $\p{\mathfrak{M},\mathfrak{g}}$ is AH as well. Then, for $0<\beta <n$, we have the common splitting
                \begin{equation}\label{standard_splitting}
                    C_\beta^{\ell,\alpha}\p{S^2\mathfrak{M}} = \ker\delta_{\mathfrak{g}}|_{C_\beta^{\ell,\alpha}\p{S^2\mathfrak{M}}}\oplus_{L^2} \mathrm{Im}\;\delta^*_{\mathfrak{g}}|_{C^{\ell+1,\alpha}_{\beta}(T\mathfrak{M})}.
                \end{equation}

            The range of allowable weights, $\beta$, is determined by the Fredholm range of $\delta_\mfg\delta_\mfg^*$, which is exactly $(0,n)$.\\
                
            Additionally, if $\mathfrak{g}$ is a negative Poincar\'e--Einstein metric (that is, if $(g,\varphi)$ is static), the linearised scalar curvature induces the splitting 
                $$C_\beta^{\ell,\alpha}\p{S^2\mathfrak{M}}=\ker D_{\mathfrak{g}}\scal|_{C_\beta^{\ell,\alpha}(S^2\mathfrak{M})}\oplus_{L^2} \mathrm{Im}\;\p{D_{\mathfrak{g}}\scal}^*|_{C^{\ell+2,\alpha}_{\beta}(\mathfrak{M})}. $$
                
            Note that the second factor of the first splitting is contained in the first factor of the second, and vice versa. Thus they combine to yield the orthogonal decomposition
                $$C_\beta^{\ell,\alpha}\p{S^2\mathfrak{M}}=TT_{\mathfrak{g}}\oplus_{L^2} \mathrm{Im}\;\p{D_{\mfg}\scal}^*|_{C^{\ell+2,\alpha}_{\beta}(\mathfrak{M})}\oplus_{L^2}\mathrm{Im}\;\delta^*_{\mathfrak{g}}|_{C^{\ell+1,\alpha}_{\beta}(T\mathfrak{M})},$$
            where 
                $$TT_{\mathfrak{g}}=\ker D_{\mfg}\scal|_{C_\beta^{\ell,\alpha}(S^2\mathfrak{M})}\cap \ker\delta_{\mathfrak{g}}|_{C_\beta^{\ell,\alpha}(S^2\mathfrak{M})}$$
            is the space of transverse--trace free tensors.
       
    \subsection{Existence of Approximate Potentials}
        Since we are interested in metric-potential pairs $(g,\varphi)$ that are not necessarily static, we shall need to weaken the concept of a static potential. The following lemma provides the definition, existence and asymptotic uniqueness of such potentials, for any AH metric whose scalar curvature decays fast enough.
        
        \begin{lemma}\label{fefferman_graham}
            Let $g$ be an AH metric with boundary defining function $\rho=e^{-r}$, satisfying $\scal_g+n(n-1)=\mathcal{O}\p{e^{-\beta r}}$, for some $\frac{n}{2}<\beta <n$. Then there exists a smooth function $\varphi$, asymptotic to $-r$, such that 
            \begin{equation}\label{phiequation}
                \Delta \varphi+\abs{d \varphi}^2=n+\mathcal{O}\p{e^{-\beta r}}.
            \end{equation}
            
            Moreover, this function is unique modulo $\mathcal{O}\p{e^{-\beta r}}$, and if $\wh{g}$ is another AH metric with the same conformal infinity and $\abs{g-\wh{g}},\abs{\scal_{\wh{g}}+n(n-1)}=\mathcal{O}\p{e^{-\beta r}}$, then $\varphi$ solves (\ref{phiequation}) with respect to $\wh{g}$.\\

            A function satisfying \eqref{phiequation} and $\varphi\to -r$, as $r\to \infty$, shall be called an \textbf{approximate potential} for $g$. We dedicate the notation $\varphi$ solely to such functions.
        \end{lemma}

        \begin{proof}
            Choosing a sufficiently large compact subset $K \subset M$, we can make the identification $M\setminus K\simeq (R,\infty)\times \partial N$, and write $g$ in the form 
                $$g=dr^2+e^{2r}h_r,$$
            where $h_r$ is a one-parameter family of metrics on $\partial N$. The Laplacian of $g$ can then be expressed as     
            \begin{equation*}
                \Delta \psi=-\partial^2_{rr}\psi-\p{n-1+G}\partial_r \psi +e^{-2r}\Delta_{h_r}\psi,
            \end{equation*}
            where $G=\frac{1}{2}\tr_{h_r}\partial_r h_r$. Suppose $\varphi$ is a function solving (\ref{phiequation}). This is equivalent to $\psi=\varphi+r$ satisfying            
            \begin{equation}\label{psiequation}
                \Delta \psi +\abs{d \psi}^2=2\partial_r \psi -G(r)+\mathcal{O}\p{e^{-\beta r}}.
            \end{equation}

            Note that $G(r)=G(-\ln{\rho})$ equals the left hand side of \eqref{rhodeltarho}. So, by the assumption on $\scal_g$, we have        
                $$G(r)=\frac{1}{2(n-1)}e^{-2r}\scal_{\overline{g}}+\mathcal{O}\p{e^{-\beta r}}. $$
            
            As $\beta >\frac{n}{2}\geq 1$, and we are only interested in lower order terms, we may choose an asymptotic expansion of $\scal_{\overline{g}}$ such that             
            \begin{equation}\label{Gexpansion}
                G(r,x)=\sum_{\ell=1}^{\lceil \beta\rceil}b_\ell(x)e^{-\ell r}+\mathcal{O}\p{e^{-\beta r}}, 
            \end{equation}
            where $x$ indicates tangential coordinates on $\partial N$. 
            Suppose $\psi=\psi (r,x)$ is a solution of \eqref{psiequation}. That is,           
                $$-\psi''-(n+1)\psi'+\p{\psi'}^2-G(\psi'-1)+e^{-2r}\p{\Delta_{h_r}\psi+\abs{d \psi}^2_{h_r}}=\mathcal{O}\p{e^{-\beta r}},$$
            where $\psi'=\partial_r\psi$. By inserting $\psi=\sum_{\ell=0}^\infty a_\ell(x)e^{-\ell r}$ and using \eqref{Gexpansion}, we obtain        
            \begin{align*}
                0 = \sum_{\ell = 0}^{\lceil \beta \rceil} \p{\ell(n + 1 - \ell)a_\ell + b_\ell}e^{-\ell r}+\sum_{\ell + m=0}^{\lceil \beta \rceil} \ell a_\ell(b_m+ma_m)e^{-(\ell+m)r}\\
                +\sum_{\ell=2}^{\lceil \beta \rceil}\p{\Delta_h a_{\ell-2}}e^{-\ell r}+\sum_{\ell+m=2}^{\lceil \beta  \rceil}\g{da_{\ell-1}}{da_{m-1}}_h e^{-(\ell+m)r}.
            \end{align*}
        
            Recall that $b_0=0$ by \eqref{Gexpansion}, so the equation above gives the following recurrence relations for the coefficients $a_\ell$:
                $$a_0=0,\quad a_1=-\frac{b_1}{n},\quad a_2=-\frac{b_2}{2(n-1)}+\frac{b_1^2}{2n^2}$$
            and for $\ell \geq 3$:
                $$ \ell(n+1-\ell)a_\ell+b_\ell+\sum_{m+p=\ell}ma_m\p{b_p+pa_p}+\Delta_ha_{\ell-2}+\sum_{m+p=\ell-2}\g{da_{m}}{da_{p}}_h=0.$$
          
            These recurrence relations are uniquely solvable for all $\ell \leq\lceil \beta \rceil< n+1$. Setting $\varphi\p{r}=-r+\sum_{\ell=1}^{\lceil \beta \rceil} a_\ell e^{-\ell r}$, we obtain a solution to \eqref{phiequation}.\\

            Suppose $\wh{g}$ is another AH metric with $\abs{g-\wh{g}},\abs{\scal_{\wh{g}}+n(n-1)}=\mathcal{O}\p{e^{-\beta r}}$ and the same conformal infinity. Then, by using \eqref{rhodeltarho}, we see that \eqref{psiequation} agrees with the corresponding equation with respect to $\wh{g}$, up to relevant order.
        \end{proof}

        \begin{remark}
            It is instructive to compare the functions in the proof of Lemma \ref{fefferman_graham} to those associated with hyperbolic space:
                $$\varphi_{\mathbb{H}^n}=-\ln\p{2\cosh{r}}=-r+\sum_{\ell=1}^\infty\frac{(-1)^\ell}{\ell}e^{-2\ell r}, $$
                $$G_{\mathbb{H}^n}=(n-1)\p{\coth{r}-1}=2\p{n-1}\sum_{\ell=1}^\infty e^{-2\ell r}. $$

            In this case the potential is not only approximate, it is static.
            From the proof of \ref{fefferman_graham}, it is clear that if $G=\mathcal{O}(e^{-\ell r})$ then $\varphi+r=\mathcal{O}(e^{-\ell r})$.
        \end{remark}

            \subsubsection{Isomorphism Properties of Laplace Type Operators}
            We will extensively use the following proposition throughout the paper.

            \begin{proposition}\label{isomorph_prop}
                Let $\p{M,g}$ be an AH manifold of dimension $n$ and class $C^{\ell,\alpha}$ with approximate potential function $\varphi$. Then
                \begin{align*}
                    \Lap_g + n : H^m_\varphi\p{M} &\rightarrow H^{m-2}_\varphi\p{M},\\
                    \Lap_\varphi + n : H^m_\varphi\p{M} &\rightarrow H^{m-2}_\varphi\p{M}
                \end{align*}
                are isomorphisms for $m = 2,\dots,\ell$. Moreover,
                \begin{align*}
                    \Lap_g + n : C^{m,\alpha}_\beta\p{M} &\rightarrow C^{m-2,\alpha}_\beta\p{M},\\
                    \Lap_\varphi + n : C^{m,\alpha}_\gamma\p{M} &\rightarrow C^{m-2,\alpha}_\gamma\p{M}
                \end{align*}
                are isomorphisms for $m = 2,\dots,\ell$, $\beta \in \p{-1,n}$, and $\gamma \in \p{\frac{n}{2} - \sqrt{\frac{n^2}{4}+n},\frac{n}{2} + \sqrt{\frac{n^2}{4}+n}}$.
            \end{proposition}
        
            \begin{proof}
                By \cite[Proposition F \& Corollary 7.4]{Lee}, the indicial radius of $\Lap_g + n$ is $R = \sqrt{\frac{\p{n-1}^2}{4} + n} = \frac{n+1}{2}$. Since $\abs{\frac{1}{2}} < \frac{n+1}{2}$, we can apply \cite[Theorem C]{Lee} to see that, for $m = 2,\dots,\ell$,
                \begin{equation*}
                    \Lap_g + n : H^m_\varphi\p{M} \rightarrow H^{m-2}_\varphi\p{M}
                \end{equation*}
                is a Fredholm operator of index $0$ and its kernel agrees with the $L^2$-kernel of $\Lap_g + n$. However, the $L^2$-kernel of $\Lap_g + n$ is trivial, hence $\Lap_g + n$ is an isomorphism between the Sobolev spaces detailed above.\\
        
                Similarly, since $\abs{\beta - \frac{n-1}{2}} < \frac{n+1}{2}$ if and only if $\beta \in \p{-1,n}$, we can use \cite[Theorem C]{Lee} and that $\Lap_g + n$ has trivial $L^2$-kernel to deduce that
                \begin{equation*}
                    \Lap_g + n : C^{m,\alpha}_\beta\p{M} \rightarrow C^{m-2,\alpha}_\beta\p{M}
                \end{equation*}
                is an isomorphism for $m = 2,\dots,\ell$ and $\beta $ in the given range.\\
        
                As for the statements involving $\Lap_\varphi$, we conjugate by the multiplication operator of $e^{\frac{\varphi}{2}}$
                $$e^{-\frac{\varphi}{2}}\circ \p{\Delta_\varphi+n}\circ e^{\frac{\varphi}{2}}=\Delta_g+F(\varphi)$$
                where $F\p{\varphi}=\frac{1}{2}\Delta_g\varphi+\frac{1}{4}\abs{d\varphi}^2+n=\frac{6n-1}{4}+\mathcal{O}\p{e^{-2r}}$. This means 
                $$\Delta_\varphi+n:C^{m,\alpha}_\gamma\p{M}\to C^{m-2,\alpha}_\gamma\p{M}$$
                is an isomorphism if and only if 
                $$\Delta_g+\frac{6n-1}{4}:C^{m,\alpha}_{\gamma-\frac{1}{2}}\p{M}\to C^{m-2,\alpha}_{\gamma-\frac{1}{2}}\p{M}$$
                is an isomorphism. By \cite[Proposition F, Corollary 7.4 \& Theorem C]{Lee}, this is the case whenever 
                $$\abs{\gamma-\frac{1}{2}-\frac{n-1}{2}}=\abs{\gamma-\frac{n}{2}}<\sqrt{\frac{(n-1)^2}{4}+\frac{6n-1}{4}}=\sqrt{\frac{n^2}{4}+n}.$$

                Similarly 
                $$\Delta_\varphi+n:H^m_\varphi\p{M}\to H^{m-2}_\varphi\p{M}$$
                is an isomorphism if and only if 
                $$\Delta_g+\frac{6n-1}{4}:H^m\p{M}\to H^{m-2}\p{M}$$
                is an isomorphism. The latter is the case, as the indicial radius $R=\sqrt{\frac{n^2}{4}+n}$ of $\Delta_g+\frac{6n-1}{4}$ is positive. Note that this is exactly the radius of the auxiliary operator $\Delta_\mfg+n$ (see also \eqref{aux_lap}).
            \end{proof}

\section{A Renormalised Expander Entropy}\label{sec_entropy}
        This section contains the definition of our entropy functional, as well as the proof that it is well-defined and finite. We also prove that a minimising function of the entropy exists and has asymptotic properties suitable for our future purposes.
    
    \subsection{Definition of the Entropy}\label{sec_entropy_defn}
        \begin{definition}
            Let $\wh{g}$ be a static background metric. Then define the following functional:
            $$S_{\wh{g}}\p{g,\varphi} := \lim \limits_{R \rightarrow \infty} \left[\int_{B_R} \scal_g+n(n-1)\; d\wh{V}_{\varphi} - m_{\wh{g},R}\p{g-\wh{g},\varphi}\right],$$
            where $m_{\wh{g},R}$ is the \textbf{weighted ADM mass} term
            \begin{equation}\label{def:mass}
                m_{\wh{g},R}\p{h,\varphi}:=-\int_{\partial B_R}\g{\delta_{\wh{g}}h+d\tr_{\wh{g}}h}{\wh{\nu}}_{\wh{g}}+\tr_{\wh{g}}h\; \wh{\nu}\p{\varphi}-h\p{d\varphi,\wh{\nu}}\; d\wh{A}_\varphi,
            \end{equation}
            where $d\wh{V}_\varphi = e^{-\varphi}\;dV_{\wh{g}}$, and $d\wh{A}_\varphi = e^{-\varphi}\;dA_{\wh{g}}$. We also define the following boundary term:
            \begin{equation*}
                \mathcal{B}_R\p{g,\varphi,f} := -\frac{1}{n} \int_{\partial B_R} \p{f - \varphi} \nu\p{\varphi} dA_\varphi.
            \end{equation*}
        \end{definition}

        We note that by setting $\wh{g}=g_{\mathbb{H}^n}$ in \eqref{def:mass}, we obtain the mass previously studied in \cite{CH} and \cite{HJM}.

        \begin{definition}\label{entropy_defn}
            Let $\p{M,g}$ be an AH manifold and let $f,\varphi \in C^\infty\p{M}$ be functions that are asymptotic to $-r$. Moreover, let $\wh{g}$ be a static background metric and define the functionals
            \begin{equation*}
                \mathcal{W}_R\p{g,\varphi,f} := \int_{B_R} \frac{1}{2n} \p{\abs{d f}^2 - \abs{d \varphi}^2 + \scal_g} + \varphi - f + \frac{n-1}{2}\; dV_f
            \end{equation*}
            and
            \begin{equation*}
                \mathcal{W}_{\mathrm{AH},\wh{g}}\p{g,\varphi,f} := \lim \limits_{R \rightarrow \infty}\p{\mathcal{W}_R\p{g,\varphi,f} - \frac{1}{2n} m_{\wh{g},R}\p{g-\wh{g},\varphi} + \mathcal{B}_R\p{g,\varphi, f}},
            \end{equation*}
            where $dV_f := e^{-f} dV_g$.
        \end{definition}

        \begin{proposition}\label{EHfinite} 
            Let $\p{M,g}$ be an AH manifold with approximate potential $\varphi$ (satisfying \eqref{phiequation}). Further let $\wh{g}$ be a static background metric with
                $$\scal_{\wh{g}}+n(n-1)=0.$$
            
            If $g\in \mathcal{R}^{2,\alpha}_\beta\p{M,\wh{g}}$, for some $\frac{n}{2}<\beta<n$, then $S_{\wh{g}}(g,\varphi)$ is finite. As a result, under these assumptions, $m_{\wh{g}}\p{g-\wh{g},\varphi} := \lim \limits_{R \rightarrow \infty} m_{\wh{g},R}\p{g-\wh{g},\varphi}$ is well-defined and finite.
        \end{proposition}
        
        \begin{proof}
            By Taylor expanding $\scal_g$ around $\wh{g}$ in the direction of $h=g-\wh{g}$, we have
            \begin{align*}
                \int_{B_R} \scal_g+n(n-1)\; d\wh{V}_\varphi={}& \int_{B_R} \scal_{\wh{g}}+n(n-1)+D_{\wh{g}}\scal_{\wh{g}}(h)\; d\wh{V}_\varphi+C_R\\ 
                ={}&m_{\wh{g},R}\p{h,\varphi}-\int_{B_R}\g{Q\varphi}{h}\; d\wh{V}_\varphi+C_R,
            \end{align*}
            where 
             $$Q\varphi=-e^\varphi\p{D_{\wh{g}}\scal}^*\p{e^{-\varphi}}=\p{\Delta_{\wh{g}}\varphi +\abs{d\varphi}^2_{\wh{g}}}\cdot \wh{g}+\nabla_{\wh{g}}^2\varphi-d\varphi \otimes d\varphi+\Ric_{\wh{g}},$$
            and $C_R $ denotes an integral of quadratic terms in $h$ and $\nabla h$. As such terms decay faster than the weighted volume grows, we have $\lim_{R\to \infty}\abs{C_R}<\infty$.\\

            Let $\wh\varphi$ be a static potential for $\wh{g}$ asymptotic to $-r$. Then it satisfies the defining equation for static potentials: $Q\wh{\varphi}=0$. Taking the trace, we also have 
                $$0=\tr Q\wh{\varphi}=(n-1)\p{\Delta_{\wh{g}} \wh{\varphi}+\abs{d\wh{\varphi}}_{\wh{g}}^2-n}.$$
       
            It then follows from Lemma \ref{fefferman_graham} that $\varphi-\wh{\varphi}=\mathcal{O}(e^{-\beta r})$. As a consequence, we have $\abs{Q\varphi}=\abs{Q(\varphi-\wh{\varphi})}=\mathcal{O}\p{e^{-\beta r}}$. It follows from the decay assumption on $h=g-\wh{g}$, that
                $$S_{\wh{g}}(g,\varphi)=\lim_{R\to \infty}\p{C_R -  \int_{B_R}\g{Q\varphi}{h}\; d\wh{V}_\varphi}$$          
            is finite.
        \end{proof}

        \begin{corollary}\label{finitenesscor} 
            Let $g\in\mathcal{R}^{2,\alpha}_\beta\p{M,\wh{g}}$ have approximate potential $\varphi$. If $f - \varphi \in C^\infty_c\p{M}$, then $\mathcal{W}_{\mathrm{AH},\wh{g}}\p{g,\varphi,f}$ is finite.
        \end{corollary}

        \begin{proof}
            From the definitions of $\mathcal{W}_{\mathrm{AH},\wh{g}}$ and $S_{\wh{g}}$, one can compute that
            \begin{align}\begin{split}\label{eq:decomp_W}
                \mathcal{W}_{\mathrm{AH},\wh{g}}(g,\varphi,f) ={}& \frac{1}{2n}\lim \limits_{R \rightarrow \infty} \bigg[\int_{B_R} \abs{d f}^2-\abs{d \varphi}^2-2n(f-\varphi)\; dV_f\\ 
                &+\int_{B_R} \p{\scal_g+n(n-1)}\p{e^{-f}-e^{-\varphi}}\; dV\\
                &+\int_{B_R} \p{\scal_g+n(n-1)}e^{-\varphi}\; \p{dV-d\wh{V}}\\ 
                &+ 2n\mathcal{B}_R\p{g,\varphi,f}\bigg] + \frac{1}{2n}S_{\wh{g}}\p{g,\varphi} 
            \end{split}\end{align}

            The limits of the first two lines and $\mathcal{B}_R$ are finite as $f-\varphi\in C_c^\infty\p{M}$. Note that $S_{\wh{g}}\p{g,\varphi}$ is finite by Proposition \ref{EHfinite}. As for the third line, by Taylor expanding the volume form of $g$ in the direction of $h=g-\wh{g}$, we have the estimate
            $$\p{\scal_g+n(n-1)}\p{dV-d\wh{V}}=\frac{1}{2}\g{\p{\scal_g+n(n-1)} g}{h}\; d\wh{V}+\mathcal{O}\p{\abs{h}^2}.$$
            
            This is integrable, since $g\in \mathcal{R}^{2,\alpha}_\beta\p{M,\wh{g}}$ implies $\abs{h},\scal_g+n(n-1)=\mathcal{O}\p{e^{-\beta r}}$. This completes the proof.
        \end{proof}

        Corollary \ref{finitenesscor} allows us to make the following definition:

        \begin{definition}\label{inf_expander_defn}
            The \textbf{renormalised expander entropy} of $\p{M,g,\varphi}$, relative to a static background metric $\wh{g}$, is defined as

            \begin{equation*}
                \mu_{\mathrm{AH},\wh{g}}\p{g,\varphi}: = \inf \limits_{\substack{f \in C^\infty\p{M}\: s.t. \\ f - \varphi \in C^\infty_c\p{M}}} \mathcal{W}_{\mathrm{AH},\wh{g}}\p{g,\varphi,f}.
            \end{equation*}
            
            In the following, we shall sometimes find it useful to take the infimum over compactly supported functions representing the difference $f-\varphi$. To that end, we define 
                $$\widetilde{\mathcal{W}}_{\mathrm{AH},\wh{g}}\p{g,\varphi,\omega}:=\mathcal{W}_{\mathrm{AH},\wh{g}}\p{g,\varphi,\varphi-2\ln(\omega+1)},$$
            from which it is immediate that 
                $$\mu_{\mathrm{AH},\wh{g}}\p{g,\varphi} = \inf \limits_{\omega\in C_{c}^\infty\p{M}} \widetilde{\mathcal{W}}_{\mathrm{AH},\wh{g}}\p{g,\varphi,\omega}.$$
            
            Should a minimiser $f$ of $\mathcal{W}_{\mathrm{AH},\wh{g}}$ exist, then $\omega = e^{-\frac{\p{f-\varphi}}{2}} - 1$ will be a minimiser of $\widetilde{\mathcal{W}}_{\mathrm{AH},\wh{g}}$. 
        \end{definition}

    \subsection{Existence of a Minimiser on Bounded Domains}\label{bnd_domain_min_sec}
        We now turn to showing local existence of a minimiser. In aid of that goal, we derive the Euler--Lagrange equation for the entropy. 

        \begin{proposition}\label{el_proposition}
            The \textbf{Euler--Lagrange equation} for the functional $\mathcal{W}_{\mathrm{AH},\wh{g}}\p{g,\varphi,\cdot}$ is
            \begin{align}\label{el_eqn}
                \EL f :=\frac{1}{2n}\p{2\Delta f+\abs{d f}^2+\abs{d \varphi}^2+2nf-2n\varphi-\scal_g-n(n-1)-2n}=0.
            \end{align}
        \end{proposition}

        \begin{proof} 
            Whether or not \eqref{el_eqn} is the Euler--Lagrange equation for $\mathcal{W}_{\mathrm{AH},\wh{g}}$, we may write 
            \begin{equation}\label{entropy_formula_with_el}
                \mathcal{W}_R(g,\varphi,f)=\frac{1}{n}\int_{B_R}\Delta f+\abs{d f}^2-n-n\EL f\; dV_f.
            \end{equation}
        
            For a perturbation $f+sv$ with $\abs{s} \ll 1$ and $v \in C^\infty_c\p{M}$, the first variation of \eqref{el_eqn} is
            $$\left.\frac{d}{ds}\right|_{s=0}\EL \p{f+sv}=\frac{1}{n}\p{\Delta v+\g{d f}{d v}+nv}.$$
            
            Thus, the integrand of \eqref{entropy_formula_with_el} has first variation 
            \begin{align*}
                \left.\frac{d}{ds}\right|_{s=0}\p{\Delta (f+sv)+\abs{d(f+sv)}^2-n-n\EL \p{f+sv}}=\g{d f}{d v}-nv.
            \end{align*}
            
            As the weighted volume satisfies $\left.\frac{d}{ds}\right|_{s=0}dV_{f+sv}=-v\; dV_f$, we obtain        
            \begin{align*}
                \left.\frac{d}{ds}\right|_{s=0}\mathcal{W}_{\mathrm{AH},\wh{g}}\p{g,\varphi,f+sv}={}& \frac{1}{n} \lim \limits_{R \rightarrow \infty}\int_{B_R} \g{d f}{d v}-nv-v\p{\Delta f+\abs{d f}^2-n-n\EL f}\; dV_f\\
                ={}& \lim \limits_{R \rightarrow \infty}\left[\int_{B_R} v\EL f\; dV_f-\frac{1}{n}\int_{B_R} v\Delta f+v\abs{d f}^2-\g{d f}{dv}\; dV_f\right]\\
                ={}&\lim \limits_{R \rightarrow \infty}\left[\int_{B_R} v\EL f\; dV_f + \frac{1}{n}\int_{\partial B_R}v\:\nu\p{f}\; dA_f\right].
            \end{align*}
            
            The limit of the boundary integral above vanishes, as $v$ has compact support. The variation of $\mathcal{B}_R\p{g,\varphi,f}$ vanishes for the same reason, while the variation of the weighted ADM-mass vanishes entirely, as it does not depend on $f$. All together, this yields the claimed Euler--Lagrange equation.
        \end{proof}

        \begin{remark}\label{static_el_remark}
            Suppose $\p{M,\wh{g},\wh{\varphi}}$ is a static triple with scalar curvature $\scal_{\wh{g}}=-n(n-1)$. Then the Euler-Lagrange equation simplifies to
             $$\EL f=\frac{1}{n}\p{\Delta f-\Delta \wh{\varphi}}+\frac{1}{2n}\p{\abs{d f}^2-\abs{d \wh{\varphi}}^2}+f-\wh{\varphi}=0.$$
             It is plain to see that $f= \wh{\varphi}$ is a solution of this equation.
        \end{remark}

        We also have the following uniqueness result when $f$ and $\varphi$ agree at infinity:

        \begin{proposition}\label{el_uniqueness}
            Let $\p{M,g,\varphi}$ be an AH manifold with approximate potential $\varphi$. Then there exists at most one solution $f$ of \eqref{el_eqn} such that $f - \varphi \rightarrow 0$ as $r \rightarrow \infty$. Moreover, for every bounded domain $\Omega \subset M$, there is at most one function $f$ solving $\eqref{el_eqn}$ on $\Omega$, such that $\left.\p{f-\varphi}\right|_{\partial \Omega} = 0$.
        \end{proposition}

        \begin{proof}
            Assume $f_1$ and $f_2$ are two solutions to \eqref{el_eqn} and $f_i - \varphi \rightarrow 0$ as $r \rightarrow \infty$ for $i = 1,2$. Then, setting $\wt{f} := f_1 - f_2$, we have
            \begin{align}
                0 &= \frac{1}{n} \Lap \wt{f} + \frac{1}{2n}\p{\abs{d f_1}^2 - \abs{d f_2}^2} + \wt{f} \nonumber\\
                &= \frac{1}{n} \Lap \wt{f} + \frac{1}{2n}\p{|d \wt{f}|^2 + 2\g{d \p{f_1 + f_2}}{d\wt{f}}} + \wt{f}. \label{diff_el_eqn}
            \end{align}
            Since $f_i - \varphi \rightarrow 0$ as $r \rightarrow \infty$ for $i=1,2$ we have
            \begin{equation*}
                \wt{f} = \p{f_1 - \varphi} + \p{\varphi - f_2} \rightarrow 0
            \end{equation*}
            as $r \rightarrow \infty$. Therefore, either $\wt{f} \equiv 0$ or $\wt{f}$ achieves an extremum on $M$. In the former case we are done.
            In the latter case, $\wt{f}$ has either a positive maximum or a negative minimum.
            If $\wt{f}$ has a positive maximum, let $p := \mathrm{argmax}_M \wt{f}$. Then $\wt{f}\p{p} > 0$, $d \wt{f}\p{p} = 0$, and $\Lap \wt{f}\p{p} \geq 0$, which contradicts \eqref{diff_el_eqn}. The case of a negative minimum is treated analogously, which completes the proof
 %On the other hand, if $p := \mathrm{argmin}_M \wt{f}$ then $\wt{f}\p{p} < 0$, $d\wt{f}\p{p} = 0$, and $\Lap \wt{f}\p{p} \leq 0$, which also contradicts \eqref{diff_el_eqn} and completes the proof. 
The proof for bounded domains is the same after replacing instances of $M$ by $\Omega$.
        \end{proof}

        Next, we have a result which generalises Remark \ref{static_el_remark}:
        
        \begin{proposition}\label{klaus_konj} 
            Let $\p{\mathbb{S}^1\times M,\mathfrak{g}}$ be a $(n+1)$-dimensional AH manifold, such that
                $$\mathfrak{g}=e^{-2\varphi}\; d\tau^2+g$$
            and
                $$\scal_{\mathfrak{g}}=-n(n+1).$$
            
            If $f$ solves \eqref{el_eqn} with respect to $g$ and $\varphi$, and if $f - \varphi \rightarrow 0$ as $r \rightarrow \infty$, then $f=\varphi$. 
        \end{proposition}
        
        \begin{proof} 
            By direct computation (see Section \ref{sec:static_metrics_etc}), one finds that
                $$\scal_{\mathfrak{g}}=\scal_g-2\Delta_g \varphi-2\abs{d\varphi}_g^2.$$
                
            Inserting $\scal_g=-n(n+1)+2\Delta_g \varphi+2\abs{d\varphi}_g^2$ into the Euler--Lagrange equation \eqref{el_eqn}, we obtain
                $$\EL f=\frac{1}{2n}\p{2\Delta_gf -2\Delta_g \varphi+\abs{df}^2_g-\abs{d\varphi}_g^2}+f-\varphi=0.$$
            
            Clearly $\EL\varphi=0$, and by Proposition \ref{el_uniqueness} we have $f=\varphi$.
        \end{proof}       

        Before proving local existence of a minimiser for the entropy, we will prove two-sided $H^1_\varphi$-bounds for the entropy. The following Poincar\'e type inequality will be of aid in this.
        
        \begin{lemma}\label{alttospecgap} 
            Suppose $\varphi$ satisfies \eqref{phiequation} and let $\omega =\mathcal{O}\p{e^{-\beta r}}$, with $\frac{n}{2}<\beta<n$. For every $0 < \Lambda < \frac{2n-1}{4}$, we can find a compact subset $K \subset M$ such that 
            $$\int_M \abs{\nabla \omega}^2\; dV_\varphi \geq \int_K \p{\frac{1}{2}\Lap \varphi + \frac{1}{4}\abs{d \varphi}^2} \omega^2\; dV_\varphi +  \int_{M\setminus K}\Lambda \omega^2\; dV_\varphi$$
        \end{lemma}
        
        \begin{proof}
            Set $V:=e^{\frac{\varphi}{2}}$. Then
            \begin{equation}\label{alttospecgap1}
                \begin{split}
                    0\leq{} & \int_M\abs{\nabla\p{V^{-1}\omega}}^2\; dV=\int_M\abs{V^{-1}\nabla \omega-V^{-2}\omega\nabla V}^2\; dV\\ ={}&\int_MV^{-2}\abs{\nabla \omega}^2-V^{-3}\g{\nabla \omega^2}{\nabla V}+V^{-4}\omega^2\abs{\nabla V}^2\; dV\\ ={}& \int_MV^{-2}\abs{\nabla \omega}^2-\omega^2\p{V^{-3}\Delta V + 2V^{-4}\abs{\nabla V}^2}+\delta\p{\omega^2V^{-3}\nabla V}\; dV
                \end{split}
            \end{equation}
            
            Since $\omega = \mathcal{O}\p{e^{-\beta r}}$ by assumption, we must have $\omega^2=\mathcal{O}\p{e^{-(n+\varepsilon)r}}$ for some $\eps > 0$. In particular,
                $$\int_M\delta\p{\omega^2V^{-3}\nabla V}\; dV=-\lim_{R\to \infty} \frac{1}{2}\int_{\partial B_R}\omega^2\nu(\varphi)\; dA_\varphi=0.$$
            
            Recall the asymptotics of $\varphi$, provided by the construction in Lemma \ref{fefferman_graham}:
            \begin{align*}\Delta \varphi+\abs{d \varphi}^2 &= n+\mathcal{O}\p{e^{-\beta r}},\\
                \abs{d \varphi}^2&=1+\mathcal{O}\p{e^{-r}}
            \end{align*}
            
            Thus for every $0 < \Lambda <\frac{2n-1}{4}$, we can choose a compact set $K \subset M$ such that on the complement $M\setminus K$ we have
                $$V^{-1}\Delta V+2V^{-2}\abs{\nabla V}^2=\frac{1}{2}\p{\Delta\varphi+\abs{d \varphi}^2}-\frac{1}{4}\abs{d \varphi}^2\geq \Lambda.$$
            
            Using this and rearranging \eqref{alttospecgap1}, we obtain 
            \begin{align*}
                \int_M V^{-2}\abs{\nabla \omega}^2\; dV &\geq \int_M V^{-2}\omega^2\p{V^{-1}\Delta V+2V^{-2}\abs{\nabla V}^2}\; dV\\
                &\geq \int_K V^{-2}\omega^2\p{V^{-1}\Delta V+2V^{-2}\abs{\nabla V}^2}\; dV +  \int_{M\setminus K}\Lambda V^{-2}\omega^2\; dV,    
            \end{align*}          
            as desired.
        \end{proof}

        Now for the two-sided $H^1_\varphi$-bounds.
        
        \begin{proposition}\label{2sided_H1_bounds}
            Suppose $g\in \mathcal{R}^{2,\alpha}_\beta\p{M,\wh{g}}$ for some $\frac{n}2<\beta<n$, and that $\varphi$ satisfies \eqref{phiequation}. Then there are constants $C_1, C_2, C_3 > 0$ depending only on $g, \wh{g}, \varphi$, and a constant $C_\eps > 0$ depending on the previous parameters as well as on $0 < \eps \ll 1$, such that    
                $$C_1\norm{ \omega}^2_{H^1_\varphi} -C_2 \leq \wt{\mathcal{W}}_{\mathrm{AH},\wh{g}}\p{g,\varphi,\omega} \leq C_\varepsilon\p{1+\norm{\omega}_{H^1_{\varphi}}^\varepsilon}\norm{\omega}^2_{H^1_\varphi} + C_3,$$
            for every $\omega \in H^1_\varphi\p{M}$.
        \end{proposition}

        \begin{proof}
            Setting $f=\varphi-2\ln\p{1+\omega}$ in the decomposition \eqref{eq:decomp_W}, we obtain 
            \begin{align}
                \begin{split}\label{eq:2sided1}
                \wt{\mathcal{W}}_{\mathrm{AH},\wh{g}}(g,\varphi,\omega)
                    ={}&\frac{1}{2n} \lim_{R\to \infty} \left[\int_{B_R} 4\abs{d\omega}^2-4(\omega+1)\g{d\omega}{d\varphi} +4n\p{1+\omega}^2\ln\p{1+\omega}\right.\\&\qquad \qquad +\p{\omega^2+2\omega}\p{\scal_g+n(n-1)}\; dV_\varphi+2n\mathcal{B}_R\p{g,\varphi,f}\biggr]+C(g,\varphi),
                \end{split}
            \end{align}
            where 
                $$C(g,\varphi):=\frac{1}{2n}\lim_{R\to \infty}\int_{B_R}\p{\scal_g+n(n-1)}e^{-\varphi}\p{dV-d\wh{V}} + \frac{1}{2n} S_{\wh{g}}\p{g,\varphi}.$$
            
            This was shown to be finite during the proof of Corollary \ref{finitenesscor}. Since $\omega \in H^1_\varphi\p{M}$, we must have 
                $$\lim_{R\to \infty }\mathcal{B}_R(g,\varphi,f)=\frac{1}{n}\lim_{R\to \infty }\int_{\partial B_R}2\ln\p{1+\omega}\nu\p{\varphi}\; dA_\varphi=\frac{1}{n}\lim_{R\to \infty }\int_{\partial B_R}2\omega\:\nu\p{\varphi}\; dA_\varphi,$$
            as the higher order terms of $\ln(1+\omega)=\omega+\mathcal{O}\p{\omega^2}$ do not contribute to the limit. By the same logic, we may exchange $2n\mathcal{B}_R\p{g,\varphi,f}$ for 
                $$\int_{\partial B_R}2\p{\omega^2+2\omega}\nu\p{\varphi}\; dA_\varphi=\int_{B_R}4\p{\omega+1}\g{d\omega}{d\varphi}-2\p{\omega^2+2\omega}\p{\Delta\varphi+\abs{d\varphi}^2}\; dV_\varphi$$
            in \eqref{eq:2sided1}. After the immediate cancellations, we are left with 
            \begin{align*}
                \wt{\mathcal{W}}_{\mathrm{AH},\wh{g}}(g,\varphi,\omega)
                    ={}&\frac{1}{2n} \lim_{R\to \infty} \left[\int_{B_R} 4\abs{d\omega}^2+4n\p{1+\omega}^2\ln\p{1+\omega}\right.\\&\qquad \qquad +\p{\omega^2+2\omega}\p{\scal_g+n(n-1)-2\Delta\varphi-2\abs{d\varphi}^2}\; dV_\varphi\biggr]+C(g,\varphi).
            \end{align*}
            
            By defining the quantities
              \begin{align*}
                H\p{\omega} &:= 2n \p{\omega + 1}^2 \ln\p{\p{\omega + 1}^2} - 2n\p{\omega^2 + 2\omega},
            \end{align*}
            and 
                $$R\p{g,\varphi} := \scal_g + n\p{n-1} -2\p{\Lap \varphi + \abs{d \varphi}^2 - n},$$
            we may simply write 
            \begin{align*}
                \wt{\mathcal{W}}_{\mathrm{AH},\wh{g}}(g,\varphi,\omega)
                ={}&\frac{1}{2n} \lim_{R\to \infty} \left[\int_{B_R} 4\abs{d\omega}^2+H(\omega)+\p{\omega^2+2\omega}R(g,\varphi)\; dV_\varphi\right]+C(g,\varphi).
            \end{align*}
            
            Note that $R(g,\varphi)=\mathcal{O}\p{e^{-\beta r}}$ by assumption, while $H$ satisfies 
            \begin{equation*}
                0 \leq H\p{\omega} \leq C_\eps \p{\omega^2+\omega^{2+\eps}},    
            \end{equation*}
            for arbitrary $\varepsilon>0$. The upper bound for the entropy is now derived using Young's Inequality:
            \begin{align*}
                \wt{\mathcal{W}}_{\mathrm{AH},\wh{g}}\p{g,\varphi, \omega}  &= \frac{1}{2n}\lim \limits_{R \rightarrow \infty}\bigg[\int_{B_R} 4\abs{d \omega}^2 + H\p{\omega} + \p{\omega^2 + 2\omega}R\p{g,\varphi}\;  dV_\varphi\bigg]+C(g,\varphi)\\
                &\leq \frac{1}{2n}\lim \limits_{R \rightarrow \infty}\bigg[\int_{B_R} 4\abs{d \omega}^2 +C_n \omega^2 + C_\eps \omega^{2 + \eps}  + 4\abs{R\p{g,\varphi}}^2 dV_\varphi\bigg]+C(g,\varphi)\\
                &\leq \wh{C}_\eps\p{1 + \norm{\omega}^\eps_{H^1_\varphi\p{M}}}\norm{\omega}^2_{H^1_\varphi\p{M}} + \norm{R\p{g,\varphi}}^2_{L^2_\varphi\p{M}}+C(g,\varphi),
            \end{align*}
            where $\wh{C}_\eps > 0$ is a constant depending on $\eps, g, \wh{g}, \varphi$. Due to its assumed decay, the $L^2_\varphi$-norm of $R\p{g,\varphi}$ is finite. Note we have also used the Sobolev embedding $H^1_\varphi\subset L^{2+\eps}_\varphi$ for $\eps$ small enough (see \cite[Lemma 3.6(c)]{Lee}). \\

            To find the claimed lower bound, we choose a compact subset $K \subset M$ such that Lemma \ref{alttospecgap} holds for some fixed $0 < \Lambda < \frac{2n-1}{4}$. By possibly enlarging $K$, we may also assume that, on $M\setminus K$, we have 
                $$\Lambda+R(g,\varphi) - \eps>0$$

            Applying Young's inequality and Lemma \ref{alttospecgap}, we obtain
            \begin{align*}
                \lim \limits_{R \rightarrow \infty} \int_{B_R} &4\abs{d \omega}^2 + H\p{\omega} + R\p{g,\varphi}\p{\omega^2 + 2\omega} dV_\varphi\\
                \geq{} & \lim \limits_{R \rightarrow \infty} \int_{B_R} 4\abs{d \omega}^2+H(\omega)+\p{R(g,\varphi)-\eps}\omega^2-\frac{1}{\eps}\abs{R(g,\varphi)}^2\; dV_\varphi\\
                \geq{} & \int_K 2\abs{d \omega}^2 + H\p{\omega} + \p{R\p{g,\varphi} - \eps + \Lap \varphi + \frac{1}{2}\abs{d \varphi}^2} \omega^2 \;dV_\varphi\\
                &+ \lim \limits_{R \rightarrow \infty} \int_{B_R \backslash K} 2\abs{d \omega}^2 + H\p{\omega} + \p{R\p{g,\varphi} - \eps + 2\Lambda} \omega^2 \; dV_\varphi - \frac{1}{\eps}\norm{R\p{g,\varphi}}^2_{L^2_\varphi\p{M}}\\
                \geq{} &\lim \limits_{R \rightarrow \infty} \int_{B_R} 2\abs{d \omega}^2+\Lambda \omega^2\; dV_\varphi- \frac{1}{\eps}\norm{R\p{g,\varphi}}^2_{L^2_\varphi\p{M}}\\ &+\int_K H(\omega)+\p{R(g,\varphi)-\eps-\Lambda+\Delta \varphi+\frac{1}{2}\abs{d \varphi}^2}\omega^2\; dV_\varphi.
            \end{align*}

            Now, since $g$ and $\varphi$ are $C^{\ell,\alpha}$ (and $\ell\geq 2$), we can use the definition of $K$, as well as $H\p{\omega} \geq 0$, and that it has faster than quadratic growth, to deduce that           
            \begin{equation*}
                H(\omega)+\p{R(g,\varphi) - \eps -\Lambda+ \Lap \varphi + \frac{1}{2}\abs{d \varphi}^2}\omega ^2\geq -C_2    
            \end{equation*}   
            on $K$. Here $C_2 > 0$ is a constant depending only on $g,\wh{g}, \varphi$. Putting everything together yields the claimed lower bound:
            \begin{align*}
                &\lim \limits_{R \rightarrow \infty} \int_{B_R} 4\abs{d \omega}^2 + H\p{\omega} + R\p{g,\varphi}\p{\omega^2 + 2\omega} dV_\varphi
                \\ &\qquad \qquad \geq C_1\norm{\omega}_{H^1_\varphi}^2-C_2\Vol_\varphi\!\p{K}-\frac{1}{\eps}\norm{R(g,\varphi)}_{L^2_\varphi}^2.
            \end{align*}
        \end{proof}

        Note that, by a density argument, Proposition \ref{2sided_H1_bounds} will allow us to use the alternative definition of the entropy:
        \begin{equation*}
            \mu_{\mathrm{AH},\wh{g}}\p{g,\varphi}: = \inf \limits_{\substack{f \in C^\infty\p{M} \; s.t.\\ f - \varphi \in H^1_\varphi\p{M}}} \mathcal{W}_{\mathrm{AH},\wh{g}}\p{g,\varphi,f} = \inf_{\omega \in H^1_\varphi\p{M}} \wt{\mathcal{W}}_{\mathrm{AH},\wh{g}}\p{g,\varphi, \omega}.
        \end{equation*}
        
        We can now prove existence of a minimiser on bounded domains by a standard variational argument. For this, it will be convenient to consider a local version of the entropy. For a bounded domain $\Omega \subset M$, we define $\mathcal{W}_{\mathrm{AH},\wh{g},\Omega}$ by replacing integration over $M$ with integration over $\Omega$ in Definition \ref{entropy_defn}. We define $\wt{\mathcal{W}}_{\mathrm{AH},\wh{g},\Omega}$ analogously. Finally, we define the local entropy:
        \begin{equation*}
            \mu_{\mathrm{AH}, \wh{g}, \Omega}\p{g,\varphi} := \inf \limits_{\substack{f \in C^\infty\p{\Omega}\; s.t. \\ f - \varphi \in H^1_{\varphi,0}\p{\Omega}}} \mathcal{W}_{\mathrm{AH},\wh{g},\Omega}\p{g,\varphi,f} = \inf_{\omega \in H^1_{\varphi,0}\p{\Omega}} \wt{\mathcal{W}}_{\mathrm{AH},\wh{g},\Omega}\p{g,\varphi, \omega}
        \end{equation*}
        where $H^1_{\varphi,0}\p{\Omega}:=\left\{\psi\in H^1_\varphi(\Omega)\; :\; \psi|_{\partial \Omega}=0\right\}$.

        \begin{proposition}\label{minimiseronboundeddomains}
            Let $\Omega \subset M$ be a bounded domain with smooth boundary. Then there is a function $\omega\in H^1_{\varphi,0}\p{\Omega}$ which realises the infimum of $\wt{\mathcal{W}}_{\mathrm{AH},\wh{g},\Omega}\p{g, \varphi, \cdot}$.
            
            This function $\omega$ is the unique such minimiser vanishing on the boundary. Equivalently, there is a unique $f \in C^{\ell, \alpha}\p{\Omega}$ solving \eqref{el_eqn} on $\Omega$, such that $\left.\p{f-\varphi}\right|_{\partial \Omega} = 0$, and it achieves the infimum of $\mathcal{W}_{\mathrm{AH},\wh{g},\Omega}$.
        \end{proposition}

        \begin{proof}
            By Lemma \ref{2sided_H1_bounds}, we know that
            \begin{equation*}
                \mu_{\mathrm{AH},\wh{g},\Omega}\p{g,\varphi} = \inf_{\omega \in H^1_{\varphi,0}\p{M}} \wt{\mathcal{W}}_{\mathrm{AH},\wh{g}, \Omega}\p{g,\varphi,\omega} > -\infty.
            \end{equation*}

            Let $\omega_i \in H^1_{\varphi,0}\p{\Omega} $ be a minimising sequence for $\mu_{\mathrm{AH},\wh{g},\Omega}\p{g,\varphi}$.\\
            
            As the weighted volume measure $e^{-\varphi} dV_g$ is equivalent to the unweighted $dV_g$ on bounded domains, we can apply the Rellich--Kondrachov theorem to deduce the existence of a subsequence, still denoted $\omega_i$, such that $\omega_i \rightharpoonup \omega$ in $H^1_\varphi\p{\Omega}$ and $\omega_i \rightarrow \omega$ in $L_\varphi^p\p{\Omega}$ for some fixed $p < \frac{2n}{n-2}$. \\

            As the functional $\omega \mapsto \wt{\mathcal{W}}_{\mathrm{AH},\wh{g}, \Omega}\p{g,\varphi,\omega}$ is coercive and convex with respect to $d\omega$, it is lower-semicontinuous on $H^1_{\varphi,0}\p{\Omega}$. Hence
            \begin{equation*}
                \mu_{\mathrm{AH},\wh{g},\Omega}\p{g,\varphi} \leq \wt{\mathcal{W}}_{\mathrm{AH},\wh{g}, \Omega}\p{g,\varphi,\omega} \leq \liminf \limits_{i \rightarrow \infty} \wt{\mathcal{W}}_{\mathrm{AH},\wh{g}, \Omega}\p{g,\varphi,\omega_i} \leq \mu_{\mathrm{AH},\wh{g},\Omega}\p{g,\varphi}.
            \end{equation*}

            Therefore, $\omega \in H^1_{\varphi,0}\p{\Omega}$ is the desired minimiser. It follows that $f:=\varphi-2\ln\p{\omega+1}$ is a weak solution to the Euler--Lagrange equation \eqref{el_eqn}. We may thus use standard arguments, involving elliptic regularity theory and the Sobolev embedding theorem, to deduce that $f \in C^{\ell,\alpha}\p{\Omega}$. Moreover, $f$ is the unique such solution due to Proposition \ref{el_uniqueness}.
         \end{proof}

    \subsection{Existence of a Global Minimiser}\label{global_min_sec}
        We now turn to proving the existence of a global minimiser for the renormalised expander entropy. This roughly follows the procedure in Section $5$ of \cite{DKM}.\\

        First of all, we find appropriate upper and lower barriers. For this, it will be convenient to rewrite the Euler--Lagrange equation \eqref{el_eqn} as follows:
        \begin{align*}
            0 = \EL f &= \frac{1}{2n}\p{2\Delta f+\abs{d f}^2+\abs{d \varphi}^2+2nf-2n\varphi-\scal_g-n(n-1)-2n}\\
            &= \frac{1}{n}\p{\p{\Lap +n}\p{f-\varphi} + \frac{1}{2}\p{\abs{d f}^2 - \abs{d \varphi}^2} + \p{\Lap \varphi + \abs{d \varphi}^2-n}  - \frac{1}{2} \p{\scal_g+n(n-1)}}.
        \end{align*}

        The second line follows from adding and subtracting $\frac{1}{n} \Lap \varphi + \frac{1}{n} \abs{d \varphi}^2$. Now, if $f$ is a minimiser we have $\EL f = 0$, hence, setting $u:= f - \varphi$,
        \begin{equation}\label{modified_el_eqn}
            \Lap u + \frac{1}{2}\abs{d u}^2 + \left<d \varphi, d u\right> + nu = -\p{\Lap \varphi + \abs{d \varphi}^2-n} + \frac{1}{2}\p{\scal_g+n(n-1)}.
        \end{equation}
        
        \begin{proposition} \label{uasymptotic}
            Let $\p{M,g,\varphi}$ be an asymptotically hyperbolic manifold with approximate potential $\varphi$, and consider $\Omega \subset M$ a bounded domain with smooth boundary. Let $u = f-\varphi$ be a solution of \eqref{modified_el_eqn} with $\left.u\right|_{\partial \Omega} = 0$. Then we have
            \begin{equation*}
               \inf_\Omega \left(\frac{1}{2} \scal_g + \frac{n\left(n-1\right)}{2} - \Lap \varphi - \abs{d \varphi} + n\right) \leq n u \leq \sup_\Omega \left(\frac{1}{2} \scal_g + \frac{n\left(n-1\right)}{2} - \Lap \varphi - \abs{d \varphi} + n\right).
            \end{equation*}
        \end{proposition}

        \begin{proof}
            If we evaluate \eqref{modified_el_eqn} at a maximum point $p := \mathrm{arg\, max}_\Omega u$ then, since $\nabla u\left(p\right) = 0$ and $\Lap u\left(p\right) \geq 0$, we see that
            \begin{equation*}
                n \cdot \sup_\Omega u \leq \sup_\Omega \left(\frac{1}{2} \scal_g + \frac{n\left(n-1\right)}{2} - \Lap \varphi - \abs{d \varphi} + n\right).
            \end{equation*}

            If instead $p := \argmin_\Omega u$, then $\nabla u\left(p\right) = 0$ and $\Lap u\left(p\right) \leq 0$, which yields
            \begin{equation*}
                n \cdot \inf_\Omega u \geq \inf_\Omega \left(\frac{1}{2} \scal_g + \frac{n\left(n-1\right)}{2} - \Lap \varphi - \abs{d \varphi} + n\right).
            \end{equation*}
        \end{proof}

        \begin{lemma}\label{barrier_lemma}
            Suppose $g$ is AH with approximate potential $\varphi$ and $\abs{\scal_g+n(n-1)}=\mathcal{O}(e^{-\beta r})$. Then there exists an $L^2_\varphi$-integrable function $b$, with the property that the Dirichlet solution of
            \begin{equation}\label{barrier_dirichlet_eq}
               \Delta u +nu+\frac{1}{2}\abs{d u}^2+\g{d \varphi}{d u}=\frac{1}{2}R(g,\varphi) 
            \end{equation}
            on any bounded set $\Omega \subset M$, satisfies $\abs{u}\leq b$. 
            As in the preceding section, $R\p{g,\varphi}=\scal_g+n(n-1)-2\p{\Delta\varphi+\abs{d \varphi}^2-n}$. 
        \end{lemma}

        \begin{proof}
            Let $K\subset M$ be a compact set, on whose complement we can write $g$ as in the proof of Lemma \ref{fefferman_graham}. Consider the radial function $b(r)=\lambda e^{-\beta r}$, for some $\lambda>0$ still to be determined. Then 
            \begin{align*}
                \Delta b +nb+\frac{1}{2}\abs{d b}^2+\g{d \varphi}{d b}={}&-b'' -(n+G-\psi')b'+nb+\frac{1}{2}(b')^2\\={}&\p{n+n\beta-\beta^2}\lambda e^{-\beta r}+\p{G-\psi'+\frac{1}{2}\beta\lambda e^{-\beta r}}\beta \lambda e^{-\beta r}
            \end{align*}
            where $G$ and $\psi'=\partial_r \psi$ are as in the aforementioned lemma. Note that the factor $n+n\beta -\beta^2>0$ for $\beta$ in the given range. As $\abs{G-\psi'}=\mathcal{O}\p{e^{-r}}$ and $\abs{R(g,\varphi)}=\mathcal{O}\p{e^{-\beta r}}$, it is possible to find constants $C_1,C_2>0$, and a radius $\rho \gg1 $ such that 
                $$\abs{G-\psi'}\leq C_1e^{-r}\qquad \&\qquad \abs{R(g,\varphi)}\leq C_2 e^{-\beta r}$$
            for $r\geq \rho$. We can then choose $\lambda$ such that, when $\rho$ is sufficiently large, we have 
                $$\p{n+n\beta -\beta^2}\lambda\geq \frac{1}{2}\p{\beta\lambda}^2e^{-\beta \rho}+ C_1 \beta \lambda e^{-\rho}+\frac{1}{2}C_2.$$
            
            Consequently, $b(r)=\lambda e^{-\beta r}$ is a supersolution 
            \begin{equation}\label{eq:supersol}
                \Delta b +nb+\frac{1}{2}\abs{d b}^2+\g{d \varphi}{d b}\geq \frac{1}{2}R(g,\varphi)\end{equation}
            on $M\setminus B_{\rho}$. For simplicity of further arguments, we shall additionally assume that $\lambda $ and $\rho$ are chosen such that $\lambda e^{-\beta\rho}\geq \sup_{M}\abs{R(g,\varphi)}$, and extend $b$ to all of $B_\rho$ by the constant value it has on $\partial B_\rho$. Let $\Omega \subset M$ be an arbitrary bounded domain with smooth boundary, and let $u$ be a solution of \eqref{barrier_dirichlet_eq}, with $u|_{\partial \Omega}=0$. By design, $u\leq b$ everywhere, except possibly in the interior of $\Omega\setminus B_\rho$. Assume this set has an interior, as we are otherwise done, and that the function $u-b$ has a positive local maximum at $p\in \Omega\setminus B_\rho$:
                $$u(p)-b(p)>0,\qquad du(p)-db(p)=0,\qquad \Delta u(p)-\Delta b(p)\geq 0.$$
            
            Then, combining \eqref{barrier_dirichlet_eq} and \eqref{eq:supersol}, we must have 
                $$\Delta u(p)-\Delta b\p{p} +nu(p)-nb(p)\leq 0,$$
            which is a clear contradiction. The conclusion is that $u\leq b$ on all of $M$, and by finding $b$ before the introduction of either $\Omega$ or $u$, it should be clear that it is independent hereof.\\
            
            The same argument, with inverted inequalities, proves that $\lambda $ may be chosen such that $-\lambda e^{-\beta r}\leq u$ on all of $M$, thus completing the proof of the lemma.
        \end{proof}

        We can now prove the existence of a global minimiser of the entropy.

        \begin{theorem}\label{global_minimiser}
            Let $\p{\wh{g},\wh{\varphi}}$ be a static pair, and let $g\in \mathcal{R}^{2,\alpha}_\beta\p{M,\wh{g}}$ be an AH metric with approximate potential $\varphi$, with $\frac{n}{2}<\beta<n$. Then there is a unique $f$ such that $f - \varphi \in C^{2,\alpha}_\beta\p{M}\cap C^{\ell,\alpha}\p{M}$ and
            \begin{equation*}
                \mu_{\mathrm{AH},\wh{g}}\p{g,\varphi} = \mathcal{W}_{\mathrm{AH},\wh{g}}\p{g,\varphi,f}.
            \end{equation*}

            Furthermore,  if $g\in \mathcal{R}_\beta^{\ell,\alpha}\p{M,\wh{g}}$ (or $g\in \mathcal{R}^\ell_\varphi\p{M,\wh{g}}$) for some $\ell\geq2$, then $f-\varphi\in C^{\ell,\alpha}_\beta\p{M}$ (respectively $H^{\ell}_\varphi\p{M}$).
        \end{theorem}

        \begin{proof}
            Let $\left\{\Omega_i\right\}^\infty_{i=1}$ be a sequence of bounded domains with smooth boundary, such that $\Omega_i \subset \Omega_{i+1}$ and $\bigcup^\infty_{i=1} \Omega_i = M$. Denote by $u_i = f_i - \varphi$ a sequence of solutions to \eqref{barrier_dirichlet_eq} such that $u_i \in C^{\ell,\alpha}\p{\Omega_i} \cap C^0\p{\overline{\Omega_i}}$ and $\left. u_i\right|_{\partial \Omega_i} = 0$. According to Proposition \ref{uasymptotic},
            \begin{equation*}
               \frac{1}{2n}\inf_M R\p{g,\varphi} \leq u_i \leq \frac{1}{2n}\sup_M R\p{g,\varphi}.
            \end{equation*}

            Therefore, there is a subsequence, also denoted $u_i$, which converges locally uniformly in $C^{\ell,\alpha}\p{M}$ to a function $u$, which is defined on all of $M$ and is still a solution to \eqref{modified_el_eqn}. Moreover, by Lemma \ref{barrier_lemma}, $\abs{u}$ is bounded by an $L^2_\varphi$-integrable barrier, which decays like $e^{-\beta r}$ as $r \rightarrow \infty$.\\

            Set $\omega_i:=e^{-\frac{u_i}{2}}-1$ and $\omega:=e^{-\frac{u}{2}}-1$. By Proposition \ref{minimiseronboundeddomains}, $\omega_i$ is the unique minimiser of the entropy over $\Omega_i$:
                $$\mu_{\mathrm{AH}, \wh{g}, \Omega_i}\p{g,\varphi} = \wt{\mathcal{W}}_{\mathrm{AH},\wh{g},\Omega_i}\p{g,\varphi,\omega_i}.$$

            By domain monotonicity of the entropy, we have
            \begin{equation*}
                \mu_{\mathrm{AH}, \wh{g}, \Omega_i}\p{g,\varphi} \geq \mu_{\mathrm{AH}, \wh{g}, \Omega_{i+1}}\p{g,\varphi}.
            \end{equation*} 

            If we extend each $\omega_i$ trivially outside $\Omega_i$, we get from the proof of Lemma \ref{2sided_H1_bounds}, that $\wt{\mathcal{W}}_{\mathrm{AH},\wh{g},\Omega_i}\p{g,\varphi,\omega_i}=\wt{\mathcal{W}}_{\mathrm{AH},\wh{g}}\p{g,\varphi,\omega_i}$ is given by
            \begin{equation*}
               \wt{\mathcal{W}}_{\mathrm{AH},\wh{g}}\p{g,\varphi, \omega_i} = \frac{1}{2n}\int_{M} 4\abs{d \omega_i}^2 + H\left(\omega_i\right) + R\p{g,\varphi}\left(\omega_i^2 + 2\omega_i\right)\;  dV_\varphi+ C(g,\varphi),
            \end{equation*}
            with $H,R,$ and $C$ as in said proof. As $H$ has subextremal growth, in the sense that $H\p{\omega}=\mathcal{O}\p{\omega^{2+\varepsilon}}$ where $\varepsilon$ may be chosen so small that $H^1_\varphi\subset L^{2+\varepsilon}_\varphi$, and $\bigcup^\infty_{i = 1} \Omega_i = M$, we see that 
            \begin{equation*}
                \lim \limits_{i \rightarrow \infty} \mu_{\mathrm{AH}, \wh{g}, \Omega_i}\p{g,\varphi}= \mu_{\mathrm{AH},\wh{g}}\p{g,\varphi}. 
            \end{equation*}

            From the coercivity estimate of Proposition \ref{2sided_H1_bounds}, we obtain a uniform bound on the $H^1_\varphi$-norms of the sequence $\omega_i$. By passing to a subsequence, if necessary, we may assume $\omega_i\rightharpoonup \omega$ in $H^1_\varphi$. Using the weak lower semi-continuity of $\wt{\mathcal{W}}_{\mathrm{AH},\wh{g}}$, ensured by the upper bound in Proposition \ref{2sided_H1_bounds}, we have 
            $$\mu_{\mathrm{AH}, \wh{g}}\p{g,\varphi}\leq \wt{\mathcal{W}}_{\mathrm{AH},\wh{g}}\p{g,\varphi, \omega}\leq \liminf_{i\to \infty}\wt{\mathcal{W}}_{\mathrm{AH},\wh{g}}\p{g,\varphi, \omega_i}=\mu_{\mathrm{AH}, \wh{g}}\p{g,\varphi}.$$

            The barrier constructed in Lemma \ref{barrier_lemma} can be used in conjunction with the dominated convergence theorem, to show that 
            $$\int_M H(\omega_i)+R(g,\varphi)\p{\omega_i^2+2\omega_i}\; dV_\varphi\longrightarrow \int_M H(\omega)+R(g,\varphi)\p{\omega^2+2\omega}\; dV_\varphi.$$

            Combined with the convergence of $\wt{\mathcal{W}}_{AH,\wh{g}}\p{g,\varphi, \omega_i}$, we actually have norm convergence $\norm{\omega_i}_{H^1_\varphi}\to \norm{\omega}_{H^1_\varphi}$, from which follows strong convergence. Concludingly, we have found a global $H^1_\varphi$-minimiser $\omega$ with
            \begin{equation*}
                \mu_{\mathrm{AH},\wh{g}}\p{g,\varphi} =\wt{\mathcal{W}}_{\mathrm{AH},\wh{g}}\p{g,\varphi,\omega}=\mathcal{W}_{\mathrm{AH},\wh{g}}\p{g,\varphi,f},
            \end{equation*}
            where $f=\varphi-2\ln\p{1+\omega}$. The Euler--Lagrange equation \eqref{modified_el_eqn} for $\omega$ reads
                \begin{equation}\label{EL_for_omega}\Delta_{\varphi} \omega+n\omega=n\wt{H}\p{\omega}-\frac{1}4 \p{\omega+1}R\p{g,\varphi},\end{equation}
            where 
            $$\wt{H}(\omega)=\omega-\p{\omega+1}\ln\p{\omega+1}=\mathcal{O}\p{\omega^2}.$$

            Since the right hand side of \eqref{EL_for_omega} lies in $H^1_\varphi \cap C^{0,\alpha}_\beta$, we can apply Proposition \ref{isomorph_prop} to see that $\omega\in H^3_\varphi\p{M}\cap C^{2,\alpha}_\beta\p{M}$. \\

            We now turn to proving the regularity statements for $f$ and $f-\varphi$. This will be achieved by proving the same regularity statements for $\omega=e^{-\frac{1}{2}\p{f-\varphi}}-1$. Assuming  $g$ is of class $C^{\ell,\alpha}$, we have $R\p{g, \varphi} \in C^{\ell-2,\alpha}\p{M}$ and $\Lap_\varphi + n: C^{m,\alpha}\p{M} \rightarrow C^{m-2,\alpha}\p{M}$ is an isomorphism by Proposition \ref{isomorph_prop}.\\

            Using \eqref{EL_for_omega}, we may compute 
            \begin{align*}
                \norm{\omega}_{C^{m,\alpha}} &\leq C\norm{\p{\Lap_\mfg + n} \omega}_{C^{m-2,\alpha}}\\
                &= C\Big|\Big| n \wt{H}\p{\omega} - \frac{1}{4}R\p{g,\varphi}\p{\omega + 1}\Big|\Big|_{C^{m-2,\alpha}}\\
                &\leq C'\p{\big|\big|\wt{H}\p{\omega}\big|\big|_{C^{m-2,\alpha}} + \norm{R\p{g,\varphi}}_{C^{m-2,\alpha}}\p{\norm{\omega}_{C^{m-2,\alpha}} + 1}}
            \end{align*}

            The right hand side is certainly finite for $m=2$, indeed it is even finite with respect to $C^{0,\alpha}_\beta$, as previously explained. Using that $C^{m,\alpha}\p{M}$ is an algebra for $1\leq m\leq \ell$ (see e.g. \cite[Lemma 3.6(a)]{Lee}), we find that $\wt{H}(\omega)\sim\omega^2$ inherits the H\"older regularity of $\omega$. A bootstrapping argument then shows that $\omega\in C^{\ell,\alpha}$.\\

            If we additionally assume that $g\in\mathcal{R}^{\ell,\alpha}_\beta\p{M,\wh{g}}$ (or $\mathcal{R}^{\ell}_\varphi\p{M,\wh{g}}$), we gain $R(g,\varphi)\in C^{\ell-2,\alpha}_\beta\p{M}$ (respectively $H^{\ell-2}_\varphi\p{M}$). Now the same bootstrapping argument goes through, using that $\Lap_\varphi+ n:C^{m,\alpha}_\beta\p{M}\to C^{m-2,\alpha}_\beta\p{M}$ (respectively $H^m_\varphi\p{M} \to H^{m-2}_\varphi\p{M}$) is an isomorphism for $m=2,\ldots, \ell$. 
            \end{proof}

        We conclude this section by proving that the minimiser depends analytically on the pair $\p{g,\varphi}$.

        \begin{proposition}\label{entropy_analytic_metric_dep}
            Let $\p{M,\wh{g},\wh{\varphi}}$ be a static triple, $\beta \in \left(\frac{n}{2}, n\right)$ and $\ell\geq 2$. Then the map
            \begin{equation*}
                \mathcal{R}^{\ell,\alpha}_\beta \p{M,\wh{g}} \times \mathcal{R}^{\ell,\alpha}_\beta \p{M, \wh{\varphi}} \ni \p{g, \varphi} \mapsto f_g - \varphi \in C^{\ell,\alpha}_\beta \p{M}
            \end{equation*}
            is analytic, hence so is the map
            \begin{equation*}
                \mathcal{R}^{\ell,\alpha}_\beta \p{M,\wh{g}} \times \mathcal{R}^{\ell,\alpha}_\beta \p{M, \wh{\varphi}} \ni \p{g, \varphi} \mapsto \mu_{\mathrm{AH},\wh{g}}\p{g,\varphi}.
            \end{equation*}
            Moreover, the map
            \begin{equation*}
                \mathcal{R}^\ell_{\wh{\varphi}}\p{M,\wh{g}} \times \mathcal{R}^\ell_{\wh{\varphi}}\p{M,\wh\varphi} \ni \p{g, \varphi} \mapsto f_g - \varphi \in H^{\ell}_{\wh{\varphi}}\p{M}
            \end{equation*}
            is analytic for $\ell > \frac{n}{2} + 2$. This map should be understood simply as sending the pair $\p{g,\varphi}$ to a solution of \eqref{modified_el_eqn}, disregarding any minimising properties.
        \end{proposition}

        \begin{proof}
            First, consider the mapping
            \begin{equation*}
                \Phi\p{g,\varphi,f} :=2\Delta f+\abs{d f}^2+\abs{d \varphi}^2+2nf-2n\varphi-\scal-n(n-1)-2n.
            \end{equation*}
            The differential of $\Phi$ in the third argument is
            \begin{equation*}
                D_{g,\varphi,f} \Phi\p{0,0,v} = 2\Lap v + 2\g{df}{dv} + 2nv.
            \end{equation*}
            If we can show that
            \begin{equation*}
                P_{g,\varphi,f} := D_{g,\varphi,f} \Phi\p{0,0,\cdot} : C^{\ell,\alpha}_\beta\p{M} \rightarrow C^{\ell-2,\alpha}_\beta \p{M}
            \end{equation*}
            is an isomorphism, then we can deduce the claimed result by applying the implicit function theorem.\\

            To this end, one can use an integration by parts argument with respect to the weighted volume measure $e^{-f}dV_g$ to deduce that $P_{g,\varphi,f}$ has trivial kernel. Indeed, if $v\in \ker P_{g,\varphi,f}$ then
                $$0=\int_MvP_{g,\varphi,f}v\; dV_f=2n\norm{v}_{L^2_f\p{M}}^2-\lim_{R\to\infty}\int_{\partial B_R}\g{\nabla v^2}{\nu}\; dA_f,$$
            with the boundary integral vanishing, as $v\in C_\beta^{\ell,\alpha}\p{M}$.\\

            Next, setting $u:=f-\varphi$, we note that $P_{g,\varphi,f} = e^{\frac{u}{2}} \circ Q_{g,\varphi,f} \circ e^{-\frac{u}{2}}$, where
            \begin{equation*}
                Q_{g,\varphi,f}\p{v} = 2\Lap v + 2nv + \p{\Lap u + \frac{1}{2}\abs{d u}^2 }v.
            \end{equation*}

            By Proposition \ref{isomorph_prop}, we know that
            \begin{equation*}
                2\Lap + 2n: C^{\ell,\alpha}_\beta\p{M} \rightarrow C^{\ell-2,\alpha}_\beta \p{M}
            \end{equation*}
            is an isomorphism for $\beta \in \p{\frac{n}{2},n}$. Moreover, since $u \in C^{2,\alpha}_\beta(M)$ by Theorem \ref{global_minimiser}, the operator $Q_{g,\varphi,f}$ has the same indicial radius as $2\Lap + 2n$, hence it is also Fredholm with zero index. Finally, since multiplication by $e^{\pm \frac{u}{2}}$ is an isomorphism, $P_{g,\varphi,f}$ is similarly Fredholm of index zero. We already know $P_{g,\varphi,f}$ has trivial kernel, therefore it must be an isomorphism and we can apply the implicit function theorem, as desired. That this implies the entropy $\mu_{\mathrm{AH},\wh{g}}\p{g,\varphi}$ depends analytically on $\p{g,\varphi}$ is due to all the operations involved in the definition of the entropy being analytic. \\

            The proof of the analogous result for Sobolev spaces is the same after replacing occurrences of $C^{\ell,\alpha}_\beta\p{M}$ and $C^{\ell-2,\alpha}_\beta \p{M}$ with, respectively, $H^\ell_{\wh{\varphi}}\p{M}$ and $H^{\ell-2}_{\wh{\varphi}}\p{M}$ and using Lemma \ref{wgtd_sob_emb} to ensure $u$ and its first two derivatives decay appropriately.
        \end{proof}

        In light of Proposition \ref{entropy_analytic_metric_dep}, we should really write $f_{g,\varphi}$ (or $f_{\mathfrak{g}}$) to reflect the dependence of $f$ on $g$ and $\varphi$. However, to simplify notation (and since $\varphi$ itself is somewhat dependent on $g$), we shall usually write $f_g$, unless we are working with auxiliary metrics or we really want to emphasise the dependence on $\varphi$.

\section{The First and Second Variations of the Entropy}\label{sec_variations}

        Having proved the entropy is well-defined and finite in the previous section, we now compute its first and second variations. This will yield the diffeomorphism invariance of the entropy as well as our stability operator.\\

        To make notation less cumbersome moving forward, we will fix a background static pair $\left(\wh{g}, \wh{\varphi}\right)$ and omit them when we write the entropy, instead writing just $\mu_{\mathrm{AH}}$. Additionally, we will at times use $D_{g,\varphi} \mu_{\mathrm{AH}}$ to denote the first variation of the entropy, while $\nabla \mu_{\mathrm{AH}}$ will denote the corresponding (weighted) $L^2$-gradient (that is, the integrand of $D_{g,\varphi} \mu_{\mathrm{AH}}$). We will also analogously use $D_{\mathfrak{g}} \mu_{\mathrm{AH}}$ to denote the variation of the entropy with the auxiliary metric $\mathfrak{g}=\p{g,\varphi}$ as the base point, while $D^2 \mu_{\mathrm{AH}}$ and $\nabla^2 \mu_{\mathrm{AH}}$ (with appropriate subscripts) will be used to, respectively, denote the second variation of the entropy and the weighted $L^2$-Hessian of the entropy (that is, the integrand of $D^2 \mu_{\mathrm{AH}}$).
    
    \subsection{The First Variation}\label{subsec_first_var}
        \begin{proposition}[First variation]\label{first_var} 
            Let $ \frac{n}{2} < \beta < n$, $g \in \mathcal{R}^{2,\alpha}_\beta\p{M,\wh{g}}$, $\varphi$ be an approximate potential, and $\abs{h},\psi,v,f-\varphi \in C_\beta^{2,\alpha}\p{M}$. Then, for $\abs{s} \ll 1$,  
            \begin{align*}
                \left.\frac{d}{ds}\right|_{s=0}\mathcal{W}_{\mathrm{AH},\wh{g}}\p{g+sh,\varphi+s\psi,f+sv} ={}& -\frac{1}{2n}\int_M \g{\Ric_g + ng + \nabla^2 f - d\varphi\otimes d\varphi}{h}\; dV_{f}\\
                &- \frac{1}{n}\int_M \p{\Delta \varphi + \g{d \varphi}{d f}-n}\psi\; dV_{f}\\ 
                &-\int_M\p{\frac{1}{2}\tr_g f-v}\EL f\; dV_f,           
            \end{align*}
            where $\EL$ is the Euler-Lagrange operator from \eqref{el_eqn}.
        \end{proposition}
        
        \begin{proof} 
            Recall from Definition \ref{entropy_defn}, that 
                $$\mathcal{W}_{\mathrm{AH},\wh{g}}(g,\varphi,f) = \lim \limits_{R \rightarrow \infty} \left[\mathcal{W}_R(g,\varphi,f)-\frac{1}{2n}m_{\wh{g},R}\p{g-\wh{g},\varphi}+\mathcal{B}_R(g,\varphi,f)\right],$$     
            where 
            \begin{align*}
                \mathcal{W}_R(g,\varphi,f)={}&\frac{1}{2n}\int_{B_R}\abs{df}^2-\abs{d\varphi}^2-2n(f-\varphi)+\scal_g+n(n-1)\; dV_f\\
                ={}&\frac{1}{n}\int_{B_R}\Delta f+\abs{d f}^2-n-n\EL f\; dV_f,\\
                m_{\wh{g},R}\p{h,\varphi}={}& - \int_{\partial B_R}\g{\delta_{\wh{g}}h+d\tr_{\wh{g}}h}{\wh{\nu}}_{\wh{g}}+\tr_{\wh{g}}h\; \wh{\nu}\p{\varphi}-h\p{d\varphi,\wh{\nu}}\; d\wh{A}_\varphi,\\
                \mathcal{B}_R(g,\varphi,f) ={}& -\frac{1}{n}\int_{\partial B_R} \p{f - \varphi} \nu\p{\varphi}\; dA_\varphi.
            \end{align*}
            
            Using \eqref{scal_var} and \eqref{f_wgtd_vol_var}, we obtain the formula for the change in $\mathcal{W}_R$, with respect to an infinitesimal perturbation of the metric:   
            \begin{align*}
                \left.\frac{d}{ds}\right|_{s=0}\mathcal{W}_R\p{g+sh,\varphi,f}={}& -\frac{1}{2n}\int_{B_R} \g{\Ric_g}{h} + h\p{d f,d f}-h\p{d \varphi,d \varphi}-\Delta \tr h-\delta \p{\delta h}\; dV_{f} \\ 
                &+\frac{1}{2n}\int_{B_R}  \tr h\p{\Delta f + \abs{d f}^2-n-n\EL f} \; dV_{f}\\
                ={}&- \frac{1}{2n}\int_{B_R} \g{\Ric_g+ ng + \nabla^2 f - d\varphi\otimes d\varphi}{h}\; dV_f\\ 
                &-\frac{1}{2}\int_{B_R}\tr h\:\EL f\; dV_{f}+\frac{1}{2n}m_{g,R}\p{h,f}.
            \end{align*}              
            Note that the second equality follows from integrating by parts. Similarly, for a perturbation of the potential $\varphi$, we have
            \begin{align*}
                \left.\frac{d}{ds}\right|_{s=0}\mathcal{W}_R\p{g,\varphi+s\psi,f}={}& -\frac{1}{n}\int_{B_R}\g{d\varphi}{d\psi} -n\psi\; dV_{f}\\ 
                ={}& -\frac{1}{n}\int_{B_R} \p{\Delta \varphi+\g{d \varphi}{d f}-n}\psi \; dV_{f}\\
                &-\frac{1}{n}\int_{\partial B_R}\psi\:\nu\p{\varphi} \; dA_{f}.
            \end{align*}
            where the second equality follows from integrating by parts. Finally, for a perturbation of $f$, Proposition \ref{el_proposition} tells us that
            \begin{align*}
                \left.\frac{d}{ds}\right|_{s=0}\mathcal{W}_R\p{g,\varphi,f+sv}
                ={}&\int_{B_R} v\EL f\; dV_{f}+\frac{1}{n}\int_{\partial B_R}v\: \nu(f)\; dA_{f}.
            \end{align*}          
             
            Now we consider the variation of the mass and boundary terms. The variation of the mass is the following:
            \begin{align*}
                \left.\frac{d}{ds}\right|_{s=0} m_{\wh{g},R}\p{g-\wh{g}+sh,\varphi+s\psi}&\\
                =m_{\wh{g},R}\p{h,\varphi}- \int_{\partial B_R} &\biggl[ \tr_{\wh{g}}\p{g - \wh{g}} \wh{\nu}\p{\psi} - \p{g - \wh{g}}\p{d\psi, \wh{\nu}}\\
                &+ \psi \left<\delta_{\wh{g}}\p{g - \wh{g}} + d\tr_{\wh{g}}\p{g - \wh{g}}, \wh{\nu}\right>_{\wh{g}} \\ &+ \psi\tr_{\wh{g}}\p{g - \wh{g}}\wh{\nu}\p{\varphi} - \psi\p{g - \wh{g}}\p{d \varphi, \wh{\nu}}\biggr]\; d\wh{A}_\varphi
            \end{align*}

            Note that only the mass term survives as $R \rightarrow \infty$, since $g - \wh{g} \in C^{2,\alpha}_\beta\p{S^2M}$ and $\psi \in C^{2,\alpha}_\beta\p{M}$ for $\beta > \frac{n}{2}$. As for the variation of the boundary term $\mathcal{B}_R$, we have
            \begin{align*}
                \left. \frac{d}{ds}\right|_{s=0}& \mathcal{B}_R\p{g+sh,\varphi+s\psi,f+sv}\\
                ={}& -\frac{1}{n} \int_{\partial B_R} \p{v - \psi}\nu\p{\varphi} + \p{f - \varphi} \nu\p{\psi} +\p{\frac{1}{2}\tr_{g^T}h-\psi} \p{f - \varphi} \nu\p{\varphi}\; dA_\varphi,
            \end{align*}
            where $g^T=g|_{\partial B_R}$ is the induced metric.
            Since $f - \varphi, \psi \in C^{2,\alpha}_\beta\p{M}$, only the first term survives as $R \rightarrow \infty$.\\

            Putting what we have so far together yields
            \begin{align*}
                \left.\frac{d}{ds}\right|_{s=0}&\mathcal{W}_{\mathrm{AH},\wh{g}}\p{g+sh,\varphi+s\psi,f+sv}\\
                ={}& -\frac{1}{2n}\int_M \g{\Ric_g +ng+\nabla^2 f-d\varphi\otimes d\varphi}{h}\; dV_{f} - \frac{1}{n}\int_M \p{\Delta \varphi+\g{d \varphi}{d f}-n}\psi\; dV_{f}\\
                &+\frac{1}{n}\lim_{R\to \infty}\p{\int_{\partial B_R}v\:\nu\p{ f} - \psi\:\nu\p{\varphi}\; dA_{f} - \int_{\partial B_R} \p{v - \psi}\nu\p{\varphi} \; dA_\varphi}\\&+\frac{1}{2n}m_{g}\p{h,f}-\frac{1}{2n}m_{\wh{g}}(h,\varphi) -\int_M\p{\frac{1}{2}\tr_g f-v}\EL f\; dV_f  .
            \end{align*}
            Since $dA_f=\p{1+\mathcal{O}(e^{-\beta r})} dA_\varphi$ and $f-\varphi, v,\psi=\mathcal{O}(e^{-\beta r})$, we may combine the boundary integrals:
            \begin{align*}
                &\lim_{R\to \infty}\p{\int_{\partial B_R}v\:\nu\p{ f} - \psi\:\nu\p{\varphi}\; dA_{f} -\int_{\partial B_R} \p{v - \psi}\nu\p{\varphi} \; dA_\varphi}\\
                &\qquad =\lim_{R\to \infty}\int_{\partial B_R}v\: \nu\p{f-\varphi}\; dA_\varphi\\
                &\qquad =0.
            \end{align*}            
            Applying the same principles, and $\abs{h},\abs{g-\wh{g}},f-\varphi=\mathcal{O}\p{e^{-\beta r}}$, it is not hard to see that 
                $$m_{g}\p{h,f}-m_{\wh{g}}(h,\varphi)=0.$$
                
            The first variation formula now follows.
        \end{proof}

        For $\abs{t} \ll 1$, let $g(t) := g + th$ (with $h \in C^{2,\alpha}_\beta\p{S^2M}$) be a family of AH metrics with approximate potentials $\varphi(t) := \varphi + t\psi+\mathcal{O}\p{t^2}$ and associated minimisers $f_{g(t)} := f_g + tv+\mathcal{O}\p{t^2}$ (with $v,\psi \in C^\infty\p{M}$). Note that, by Proposition \ref{entropy_analytic_metric_dep} and $h \in C^{2,\alpha}_\beta\p{S^2M}$, we also have $v \in C^{2,\alpha}_\beta\p{M}$. This implies $\psi \in C^{2,\alpha}_\beta\p{M}$ as $v - \psi \in C^{2,\alpha}_\beta\p{M}$. We may therefore use Proposition \ref{first_var} to conclude the following:

        \begin{corollary}\label{first_var_crit}
            Let $\mu_{\mathrm{AH}}\p{g,\varphi}$ be the renormalised expander entropy of Definition \ref{inf_expander_defn}, and let $f_g$ be the unique minimiser from Theorem \ref{global_minimiser}. The $L^2_{f_g}$-gradient of $\mu_{\mathrm{AH}}$ is given by
            \begin{align}\label{entropy-gradient}
                \nabla\mu_{\mathrm{AH}}\p{g,\varphi}=\p{-\frac{1}{4n}\p{2\Ric_g+2n g-2\: d\varphi\otimes d\varphi+\mathscr{L}_{\nabla f_g}g},-\frac{1}{n}\p{\Delta \varphi + \mathscr{L}_{\nabla f_g}\varphi - n}}.
            \end{align}

            The gradient flow is thus
            \begin{equation}\label{entropy_grad_flow}
                \begin{cases}
                    \partial_t g(t) &= -2\left(\Ric_{g(t)} + ng(t) - d\varphi(t) \otimes d\varphi(t) + \nabla^2 f_{g(t)}\right)\\
                    \partial_t \varphi(t) &= -\left(\Lap \varphi(t) + \left<d \varphi(t), d f_{g(t)}\right> - n\right).
                \end{cases}
            \end{equation}
            
            With respect to the auxiliary metric, $\mathfrak{g}=e^{-2\varphi}\; d\tau^2+g$, the first variation in the direction $\mathfrak{h}=-2\psi e^{-2\varphi}\; d\tau^2+h$ can be expressed as
            \begin{equation}\label{aux_first_var}
                D_{\mathfrak{g}}\mu_{\mathrm{AH}}\p{\mathfrak{h}}=-\frac{1}{2 n}\int_{M}\g{\Ric_{\mathfrak{g}}+n\mathfrak{g}+\nabla^2_{\mathfrak{g}}\p{f_g-\varphi}}{\mathfrak{h}}_{\mathfrak{g}}\;dV_{f_g}
            \end{equation}
        \end{corollary}
        
        \begin{proof}
            The gradient expression for $\mu_{\mathrm{AH}}$ follows directly from Proposition \ref{first_var}, with the only difference being $\EL f_g=0$. The gradient flow is then read off using $\mathscr{L}_{\nabla f_g}g=2\delta^*\nabla f_g=2\nabla^2f_g$ and $\mathscr{L}_{\nabla f_g}\varphi=\g{d \varphi}{d f_g}$.\\

            The formula for the first variation with respect to the auxiliary metric is a reformulation of the gradient above, using only \eqref{auxiliary_Ric} and $\nabla^2_{\mathfrak{g}}f=-\g{d\varphi}{df}_ge^{-2\varphi}\; d\tau^2+\nabla^2_{g}f$, for any $\tau$-independent function $f$ on $M$.
        \end{proof}

        \begin{remark}
            Modulo the diffeomorphisms generated by $\nabla\p{f_g-\varphi}$, the gradient flow is given by 
                $$\begin{cases}
                    \partial_t g(t) = -2\p{\Ric_{g(t)} + n g(t) - d\varphi(t) \otimes d\varphi(t) + \nabla^2 \varphi(t)}\\
                    \partial_t \varphi(t) = -\p{\Delta \varphi(t) + \abs{d \varphi(t)}^2-n}.
                \end{cases}$$
                
            Note that these are precisely the equations of staticity \eqref{static_equations}.
        \end{remark}

        Though the entropy is ostensibly only defined for pairs $\p{g,\varphi}\in \mathcal{R}^{2,\alpha}_\beta\p{M,\wh{g}}\times \mathcal{R}^{2,\alpha}_\beta\p{M,\wh\varphi}$, it is actually possible to extend the concept to an $H^\ell_{\wh\varphi}$ neighbourhood of $\p{\wh{g},\wh\varphi}$. As stated back in Proposition \ref{entropy_analytic_metric_dep}, the map sending $\p{g,\varphi}\in \mathcal{R}^\ell_{\wh\varphi}\p{M,\wh{g}}\times \mathcal{R}^\ell_{\wh\varphi}\p{M,\wh{\varphi}}$ to $u_g=f_g-\varphi\in H^\ell_{\wh{\varphi}}$ is analytic. Here $u_g$ should a priori not be thought of as a minimiser, but simply a solution to the Euler--Lagrange equation \eqref{modified_el_eqn}. It was not possible back then to conclude that the entropy itself was well-defined and analytic as a consequence, but it is now.

        \begin{corollary} The map 
                $$\mathcal{R}^{\ell}_{\wh\varphi}\p{M,\wh{g}}\times \mathcal{R}^{\ell}_{\wh\varphi}\p{M,\wh{g}}\ni \p{g,\varphi}\longmapsto \mu_{\mathrm{AH}}\p{g,\varphi}$$
            is well-defined and analytic.
        \end{corollary}
        \begin{proof}    
             By \cite[Lemma 4.9]{Lee}, we can approximate $\mfg=\p{g,\varphi}$ by a sequence $\mfg_i=\p{g_i,\varphi_i}$ such that $\mfg_i-\wh\mfg\in C_c^\infty$. Clearly, every $\mfg_i$ satisfies the decay conditions necessary for $\mu_{\mathrm{AH}}\p{\mfg_i}$ to be well-defined and finite, and for the first variation formula \eqref{aux_first_var} to hold. Thus, through the fundamental theorem of calculus, we get 
            $$\abs{\mu_{\mathrm{AH}}\p{\mfg_i}-\mu_{\mathrm{AH}}\p{\mfg_j}}\leq \sup_{t\in [0,1]}\abs{\p{\nabla\mu_{\mathrm{AH}}\p{\mfg_{ij}(t)},\mfg_i-\mfg_j}_{L^2_{\mfg_{ij}(t)}}},$$
        where $\mfg_{ij}(t)=\mfg_i+t\p{\mfg_j-\mfg_i}$, and 
            $$\nabla\mu_{\mathrm{AH}}\p{\mfg}=-\frac{1}{2n}\p{\Ric_{\mfg}+n\mfg+\nabla^2_{\mfg}u_\mfg}e^{-u_\mfg}.$$
        
        For $i,j$ sufficiently large, we may assume that $\norm{\mfg_{ij}(t)-\wh\mfg}_{H^\ell_{\wh{\varphi}}}\leq C $ (since $\mfg_{ij}(t)\overset{H^{\ell}_{\wh{\varphi}}}{\longrightarrow}\mfg$ for $i,j\to \infty$ and all $t\in[0,1]$). Since $\ell\geq \frac{n}{2}+2\geq 2$, we also obtain a uniform bound on the gradient 
        \begin{align*}
            \norm{\nabla\mu_{\mathrm{AH}}\p{\mfg_{ij}(t)}}_{L^2_{\wh\mfg}} &\leq \norm{\nabla\mu_{\mathrm{AH}}\p{\mfg_{ij}(t)}}_{H^{\ell-2}_{\wh\mfg}}\leq C'\norm{\mfg_{ij}(t)-\wh\mfg}_{H^\ell_{\wh\mfg}} \leq C''.
        \end{align*}        
        Thus, Cauchy-Schwarz gives 
        $$\abs{\mu_{\mathrm{AH}}\p{\mfg_i}-\mu_{\mathrm{AH}}\p{\mfg_j}}\leq C''\norm{\mfg_i-\mfg_{j}}_{L^2_{\wh\mfg}}$$
        where we also used that the $L^2$-norms are comparable for all metrics involved. This proves that $\mu_{\mathrm{AH}}\p{\mfg_i}$ is a Cauchy sequence, converging to some finite value, which we denote by $\mu_{\mathrm{AH}}\p{\mfg}$. Additionally, it is clear from the explicit expression of the gradient that it carries over to the $H^{\ell}_{\wh\varphi}$ setting. Now that we know that the entropy is well-defined, we can use the fact that 
        $$\mathcal{R}^\ell_{\wh\varphi}\p{M,\wh{g}}\times \mathcal{R}^\ell_{\wh\varphi}\p{M,\wh{\varphi}}\ni \mfg \longmapsto \nabla\mu_{\mathrm{AH}}\p{\mfg}\in H^{\ell-2}_{\wh\varphi}(M)$$
        is analytic to conclude that 
        $$\mu_{\mathrm{AH}}\p{\mfg}=\int_0^1\p{\nabla \mu_{\mathrm{AH}}\p{\wh\mfg+t\p{\mfg-\wh\mfg}},\mfg-\wh\mfg}_{L^2_{\mfg(t)}}\; dt$$
        is also analytic (as it is composed of analytic maps).
        \end{proof}

        In the remainder of this section, $\p{g,\varphi}$ is assumed to be in either $\mathcal{R}^{2,\alpha}_{\beta}\p{M,\wh{g}}\times \mathcal{R}^{2,\alpha}_{\beta}\p{M,\wh{g}}$ (with $\frac{n}{2}<\beta<n$) or $\mathcal{R}^{\ell}_{\wh\varphi}\p{M,\wh{g}}\times \mathcal{R}^{\ell}_{\wh\varphi}\p{M,\wh{g}}$ (with $\ell>\frac{n}{2}+2$). The same goes for the flow $\p{g(t),\varphi(t)}$ in Lemma \ref{entropy_monot}.

        \begin{proposition}\label{static_crit_points}
            Suppose $(g,\varphi)$ is a critical point for the renormalised expander entropy, then $(g,\varphi)$ is a static pair.
        \end{proposition}

        \begin{proof}
            Criticality implies           
            \begin{equation}\label{critical_equations}
                \begin{cases}
                    \Ric_g+ng-d\varphi\otimes d\varphi+\nabla^2f&=0\\ \Delta\varphi+\g{d \varphi}{d f}-n&=0,
                \end{cases}
            \end{equation} 
            where $f=f_g$ is the unique minimiser and solution to the Euler--Lagrange equation:  
                $$2\Delta f+\abs{d f}^2+\abs{d \varphi}^2+2n(f-\varphi)-\p{\scal_g+n(n-1)}-2n=0.$$    Taking the trace of the first equation of \eqref{critical_equations}, we obtain
                $$\Delta f+\abs{d \varphi}^2-n=\scal_g+n(n-1).$$
                By inserting this in the Euler--Lagrange equation, it simplifies to             
                $$\Delta f+\abs{d f}^2-n+2n(f-\varphi)=0.$$
            Subtracting the second equation of \eqref{critical_equations} from this, leaves us with    $$\Delta(f-\varphi)+\g{d f}{d (f-\varphi)}+2n(f-\varphi)=0.$$
            Multiplying the above by $(f-\varphi)$ and integrating over the ball $B_R$ with respect to the weighted volume $e^{-f} dV$ yields    
                $$0=\int_{B_R}\abs{d(f-\varphi)}^2+2n\abs{f-\varphi}^2\; dV_f-\int_{\partial B_R}(f-\varphi)\nu\p{f-\varphi}\; dA_f.$$
            Taking the limit as $R\to \infty$, and using that $\p{f-\varphi}=\mathcal{O}_1\p{e^{-\beta r}}$, we obtain           
                $$\norm{d\p{f-\varphi}}_{L^2_f\p{M}}^2+2n\norm{f-\varphi}_{L^2_f\p{M}}^2=0.$$
            This proves that $f=\varphi$, which turns \eqref{critical_equations} into the defining equations for a static pair \eqref{static_equations}.
        \end{proof}
        Note that Proposition \ref{static_crit_points} tells us that no non-trivial solitons are possible.

        \begin{remark}\label{DCSC_crit} 
            While static metrics are the only true critical points for the entropy, another class of semi-critical metrics will play an important role in the proof of the local positive mass result (Theorem \ref{thm:local_PMT}). Let $\mathfrak{g}=\p{g,\varphi}$ be a pair satisfying 
                $$\scal_g=-n(n-1)\qquad \qquad \&\qquad \qquad \Delta\varphi+\abs{d\varphi}^2=n.$$
        
            This means both $g$ and $\mathfrak{g}$ are constant scalar curvature (CSC) metrics. If $\mathfrak{h}=-2\psi e^{-2\varphi}\; d\tau^2+\eta g$ preserves the conformal class of the base metric $g$, then
            \begin{align*}
                D_{\mathfrak{g}}\mu_{\mathrm{AH}}(\mathfrak{h})={}&-\frac{1}{4\pi n}\int_{\mathfrak{M}}\g{\Ric_\mathfrak{g}+n\mathfrak{g}}{\mathfrak{h}}\; dV_{\mathfrak{g}}\\ ={}&-\frac{1}{4\pi n}\int_{\mathfrak{M}} \eta\p{\scal_g-\Delta \varphi-\abs{d\varphi}^2+n^2}+2\psi\p{\Delta \varphi+\abs{d\varphi}^2-n}\; dV_{\mathfrak{g}}\\ ={}&0,\end{align*}
            where the only ingredients were the formula for the auxiliary Ricci curvature \eqref{auxiliary_Ric}, and $f_g=\varphi$ from Proposition \ref{klaus_konj}.
        \end{remark}

        \begin{proposition}\label{entropy_diffeo_invar} 
            Let $X$ be a vector field on $M$ such that $\delta^*X\in C_\beta^{2,\alpha}\p{S^2M}$ (or $H^\ell_{\wh{\varphi}}\p{S^2M}$). Then 
                $$\p{\nabla\mu_{\mathrm{AH}}(g,\varphi),\p{\mathscr{L}_Xg,\mathscr{L}_X\varphi}}_{L^2_f\p{M}}=0.$$
        
            In particular, the entropy $\mu_{\mathrm{AH}}(g,\varphi)$ is invariant under the action of diffeomorphisms that preserve the structure of the auxiliary metric $\mathfrak{g}=e^{-2\varphi}\; d\tau^2+g$.    
        \end{proposition}

        \begin{proof}
            To shorten notation, we set             
                \begin{align*}\nabla\mu_1&:=-\frac{1}{2n}\p{\Ric+ng+\nabla^2f-d\varphi\otimes d\varphi},\\
                \nabla\mu_2&:=-\frac{1}{n}\p{\Delta \varphi+\g{df}{d\varphi}-n},\end{align*}
            where $f=f_g$ is the unique solution to the Euler--Lagrange equation for the pair $\p{g,\varphi}$. Let $\theta_t$ be the flow of $X$. Then Proposition \ref{first_var} tells us the first variation is
            \begin{equation}\label{diffeo_invar_variation}
                \left.\frac{d}{ds}\right|_{s=0}\mathcal{W}_{\mathrm{AH}}\left(\theta_s^*g,\theta_s^*\varphi,\theta_s^*f\right) = \int_M \left<\nabla\mu_1, 2\delta^\ast X \right>\; dV_{f} + \int_M \nabla\mu_2 X\left(\varphi\right)\; dV_{f}.
            \end{equation}            
            Integrating by parts, we have 
            \begin{align*}
                \int_M \g{\nabla\mu_1}{\delta^*X}\; dV_f=\int_M\g{\delta \nabla\mu_1+\nabla\mu_1 \cdot df}{X}\; dV_f.
            \end{align*}

            There is no limiting boundary term as both $\nabla\mu_1$ and $\abs{X}$ decay as $e^{-\beta r}$. By using the contracted second Bianchi identity, $\delta\Ric=-\frac{1}{2}d\scal$, and the Ricci identity, $\Delta \nabla f=\nabla \Delta f-\Ric\cdot df$, we set about simplifying the expression above:
            \begin{align*}
                \delta \nabla\mu_1+\nabla\mu_1\cdot df={}&-\frac{1}{2n}\left(\delta\Ric+\Delta\nabla f-\Delta \varphi\; d\varphi+\nabla^2\varphi\cdot d\varphi\right.\\&\left.+\Ric\cdot df+n\;df+\nabla^2f\cdot df-\g{d\varphi}{df}\; d\varphi\right)\\
                ={}&-\frac{1}{2n}d\p{\Delta f+\frac{1}{2}\abs{\nabla f}^2+\frac{1}{2}\abs{\nabla \varphi}^2+nf-n\varphi-\frac{1}{2}\scal}\\&+\frac{1}{2n}\p{\Delta \varphi+\g{d\varphi}{df}-n}\; d\varphi\\={}&-\frac{1}{2}d\p{\EL f}-\frac{1}{2}\nabla\mu_2\; d\varphi.
            \end{align*}
            Since $\EL f=0$, we have
            $$\int_M \g{\nabla\mu_1}{\delta^*X}\; dV_f=-\frac{1}{2}\int_M \nabla\mu_2\g{\nabla\varphi}{X}\; dV_f.$$ 
            Substituting this into \eqref{diffeo_invar_variation} yields the desired result and completes the proof.
        \end{proof}

        The diffeomorphism invariance of the entropy now allows us to prove that the entropy is non-decreasing along the modified Ricci-harmonic flow \eqref{entropy_grad_flow}.

        \begin{lemma}\label{entropy_monot}
            Let $\left(g(t), \varphi(t)\right)$ be a solution of \eqref{entropy_grad_flow}. Then the function $t\mapsto \mu_{\mathrm{AH}}\p{g(t),\varphi(t)}$ is monotonically increasing. Moreover, it is strictly increasing unless $\p{g(t),\varphi(t)}$ is a family of static pairs, in which case the entropy is constant in time.
        \end{lemma}

        \begin{proof}
            Proposition \ref{entropy_diffeo_invar} tells us the following for the first variation of the entropy:
                $$D_{g,\varphi} \mu_{\mathrm{AH}}\left(\mathscr{L}_{\nabla f} g, \mathscr{L}_{\nabla f} \varphi\right) = 0.$$
            
            Then, since $\left(g(t), \varphi(t)\right)$ solves \eqref{entropy_grad_flow}, the first variation at any given $t$ is 
            \begin{align*}
                \left.\frac{d}{ds}\right|_{s=t} \mu_{\mathrm{AH}}\p{g(s), \varphi(s)} ={}& \frac{1}{n} \int_M \abs{\Ric\left(g(t)\right) + ng(t) + \nabla^2 f_{g(t)} - d\varphi(t) \otimes  d\varphi(t)}^2\; dV_{f_{g(t)}}\\
                &+ \frac{2}{n}\int_M \abs{\Delta \varphi(t) +\g{df_{g(t)}}{d\varphi(t)}-n}^2 \;dV_{f_{g(t)}}\\
                \geq{} & 0.
            \end{align*}

            The equality case follows from Proposition \ref{static_crit_points}.
        \end{proof}

    \subsection{The Second Variation}\label{sec_second_var}
        Now we turn to deriving our stability operator through the computation of the second variation of the entropy. Before doing this, we make the following observation:\\        
    
        If we let $\p{g,\varphi}$ be a critical point of the renormalised expander entropy -- that is to say, a static pair -- then the associated Poincar\'e--Einstein metric $\mfg=e^{-2\varphi}\; d\tau^2 + g$ provides a natural splitting of the space of variations of $\p{g,\varphi}$, as given by \eqref{standard_splitting}. \\

        For a general variation $\mfh=\p{h,\psi}$ of $\p{g,\varphi}$ (see \eqref{aux2tensor} for the notation), it is sometimes prudent to assume it is \textbf{transverse} (divergence free). The consequence for the pair $\p{h,\psi}$ is provided by \eqref{aux_divergence}:
        \begin{equation}\label{eq:divfree}
            \delta_{\mfg}\mfh=\delta_{g} h+h\cdot d\varphi +2\psi \; d\varphi=0.
        \end{equation}

        If we instead assume $\mfh$ is trace free, we obtain another condition on the pair $(h,\psi)$ using \eqref{aux_trace}:
            $$\tr_{\mfg}\mfh=-2\p{\psi-\frac{1}{2}\tr_{g} h}=0.$$

        Now consider a family of triplets $\p{g_s,\varphi_s,f_s}$, where $f_s$ is the unique solution to the Euler--Lagrange equation with respect to $\mfg_s=\p{g_s,\varphi_s}$ and $f_0=\varphi_0$. Then we may use \eqref{el_eqn}, \eqref{scal_var}, \eqref{lap_var}, and \eqref{inner_prod_var}, to linearise $\EL$         
        \begin{equation}\label{auxiliary_EL_linearised}
            0=\left.\frac{d}{ds}\right|_{s=0}2n\EL_{\, g_s,\varphi_s}\! f_s=2\p{\Delta_{\mfg_0}+n}\p{v-\frac{1}{2}\tr_{g_0} h} -\delta_{\mfg_0}\delta_{\mfg_0}\mathfrak{h}+\g{\Ric_{\mfg_0}+n\mfg_0}{\mfh}_{\mfg_0},
        \end{equation}
        where, as usual, $v=\left.\frac{d}{ds}\right|_{s=0}f_s$. In the particular case that $\mfg_0$ is static and $\mfh$ is transverse, is simplifies to
            \begin{equation}\label{linearised_minimiser} 0=2\p{\Delta_{\mfg_0}+n}\p{v-\frac{1}{2}\tr_{g_0}h},\end{equation}
        which implies $v=\frac{1}{2}\tr_{g_0}h$.\\

        It is now time to compute the second variation of $\mu_{\mathrm{AH}}$, here in the form of the Hessian. In the following $\wh\mfg=\p{\wh{g},\wh\varphi}$ will be used to denote an arbitrary static pair, not necessarily the background pair of sections \ref{sec_entropy_defn}-\ref{subsec_first_var}.
    
        \begin{proposition}[Second Variation]\label{second_var} 
            Let $(M,\wh{g},\wh{\varphi})$ be a static triple, then the $L^2_{\wh{\varphi}}$--Hessian of the entropy is
            \begin{align}\begin{split}\label{entropy-hessian}
                \nabla^2\mu_{\mathrm{AH}}\p{h,\psi}&\\=-\frac{1}{2n}&\biggl( \frac{1}{2}\Delta_{E,\wh{\varphi}}h+\p{h\cdot d\wh{\varphi}}\otimes d\wh{\varphi}+2\psi\nabla^2\wh{\varphi}+\nabla^2\p{v-\frac{1}{2}\tr h}-\delta^*\p{\delta_{\wh{\mathfrak{g}}} \mathfrak{h}},\\
                &\quad\quad\: 2\Delta_{\wh{\varphi}} \psi+2\g{h}{\nabla^2\wh{\varphi}}-2\g{\delta_{\wh{\varphi}} h}{d\wh{\varphi}}+2\g{d\p{v-\frac{1}{2}\tr h}}{d\wh{\varphi}}\biggr),\end{split}
            \end{align}
            where $\Delta_{E,\wh{\varphi}}h=\Delta_Eh+\nabla_{\nabla \wh{\varphi}}h=\nabla^*\nabla h-2\overset{\;\circ}{R}_{\wh{g}}h+\nabla_{\nabla \wh{\varphi}}h$ and $\overset{\;\circ}{R}_{\wh{g}} h := \wh{g}^{km}\wh{g}^{\ell p} \wh{R}_{ikj\ell}h_{mp}$.
            Note that $\delta_{\wh{\mathfrak{g}}} \mathfrak{h}$ can be interpreted as a 1-form on $M$ due to \eqref{aux_divergence}.   
        \end{proposition}   

        \begin{proof}
            Observe that the terms in the following computation involving the variation of the weighted volume vanish after evaluating at a critical point, so we will omit them. Specifically, the Hessian is obtained as the entry--wise variation of the gradient \eqref{entropy-gradient}. \\
        
            Combining the variational formulas \eqref{ric_var}, \eqref{hess_var}, \eqref{lap_var}, \eqref{inner_prod_var} and evaluating them at the critical point $\left(\wh{g}, \wh{\varphi}\right)$ (so $\left.f_{g_s}\right|_{s=0} = \wh{\varphi}$ by Proposition \ref{static_crit_points}), we obtain    
            \begin{align}\begin{split} \label{second_var_first_argument}
                \left. \frac{d}{ds}\right|_{s = 0}&\p{\Ric_{g_s} + \nabla^2_{g_s} f_{g_s} + ng_s - d\varphi_s \otimes d\varphi_s}\\
                ={}&\frac{1}{2}\Delta_Lh+\frac{1}{2}\nabla_{\nabla \wh{\varphi}}h+nh+\frac{1}{2}h\times \nabla^2\wh{\varphi}+\nabla^2\p{v-\frac{1}{2}\tr h}\\
                & -\delta^*\p{\delta_{\wh{\varphi}} h}-d\psi\otimes d\wh{\varphi}-d\wh{\varphi}\otimes d\psi,
            \end{split}\end{align}
            where $\left(h \times k\right)_{ij} = \wh{g}^{\ell m}\left(h_{i\ell}k_{mj} + h_{j\ell}k_{mi}\right)$. Recalling the relationship between the auxiliary and shift divergence, $\delta_{\wh{\mathfrak{g}}}\mathfrak{h}=\delta_{\wh{\varphi}} h+2\psi \; d\wh{\varphi}$, the final line may be reduced to        
                $$-\delta^*\p{\delta_{\wh{\varphi}} h}-d\psi\otimes d\wh{\varphi}-d\wh{\varphi}\otimes d\psi=-\delta^*\p{\delta_{\wh{\mathfrak{g}}}\mathfrak{h}}+2\psi\nabla^2\wh{\varphi}.$$

            A further reduction is possible by using the relationship between the Lichnerowicz Laplacian and the Einstein operator: $\Delta_Lh=\Delta_E h+h\times \Ric$. Since the static equations hold, we have 
            \begin{align}
            \begin{split}\label{where_we_used_static}
                \frac{1}{2}\Delta_Lh+nh+\frac{1}{2}h\times \nabla^2\wh{\varphi} ={}& \frac{1}{2}\Delta_E h+\frac{1}{2}h\times \p{\Ric+n\wh{g}+\nabla^2\wh{\varphi}}\\ ={}&\frac{1}{2}\Delta_E h+\p{h\cdot d\wh{\varphi}}\otimes d\wh{\varphi}.
                \end{split}
            \end{align}     
        
            The variation of the second term of the gradient of $\mu_{\mathrm{AH}}$, scaled by the appropriate factor of two, is  
            \begin{align}\begin{split} \label{second_var_second_argument}
                \left. \frac{d}{ds}\right|_{s = 0}&\p{2\Delta_{g_s}\varphi_s+2\g{df_{g_s}}{d\varphi_s}_{g_s}-2n}\\
                ={}&2\Delta \psi+2\g{h}{\nabla^2\wh{\varphi}-d\wh{\varphi}\otimes d\wh{\varphi}}\\ &-2\g{\delta h-d\psi}{d\wh{\varphi}}+2\g{d\p{v-\frac{1}{2}\tr h}}{d\wh{\varphi}}.
            \end{split}\end{align}
        
            The Hessian \eqref{entropy-hessian} now follows from the use of \eqref{second_var_first_argument} and \eqref{second_var_second_argument}, when differentiating \eqref{entropy-gradient}.
        \end{proof}

        \begin{corollary}\label{stability_op_cor} 
            Let $\p{M, \wh{g},\wh{\varphi}}$ be a static triple. If the variation $\p{h,\psi}$ is transverse (in the sense of \eqref{eq:divfree}), then 
            \begin{align*}
                \nabla^2\mu_{\mathrm{AH}}&\p{h,\psi}\\={}&-\frac{1}{2n}\left( \frac{1}{2}\Delta_{E,\wh{\varphi}}h+\p{h\cdot d\wh{\varphi}}\otimes d\wh{\varphi}+2\psi\nabla^2\wh{\varphi},\;2\Delta_{\wh{\varphi}} \psi+2\g{h}{\nabla^2\wh{\varphi}}+4\psi \abs{d\wh{\varphi}}^2\right).
            \end{align*}
        \end{corollary}
        
        \begin{proof} 
            Follows immediately from Proposition \ref{second_var} and the discussion at the start of this section.
        \end{proof}

        \begin{proposition}\label{auxiliary_einstein_op} 
            Let $\p{\wh{g},\wh{\varphi}}$ be a static pair with auxiliary (PE) metric $\wh{\mathfrak{g}}=e^{-2\wh{\varphi}}\; d\tau^2+\wh{g}$. If $\mathfrak{h}=-2\psi e^{-2\wh{\varphi}}\; d\tau^2+h$, then        
            \begin{align*}
                \frac{1}{2}\Delta_{E,\wh{\mathfrak{g}}}\mathfrak{h}={}&\p{\frac{1}{2}\Delta_{E,\wh{\varphi}}h+(h\cdot d\wh{\varphi})\otimes d\wh{\varphi}+2\psi \nabla ^2\wh{\varphi},\; 2\Delta_{\wh{\varphi}}\psi+2\g{h}{\nabla^2\wh{\varphi}}+4\psi \abs{d\wh{\varphi}}^2}.
            \end{align*}
            
            In particular:         
            \begin{equation}\label{aux_mu_Hessian}
                D^2_{\wh{\mathfrak{g}}}\mu_{\mathrm{AH}}\p{\mathfrak{h}}=-\frac{1}{2n}\p{\frac{1}{2}\Delta_{E,\wh{\mathfrak{g}}}\mathfrak{h}+\nabla^2_{\wh{\mathfrak{g}}}\p{v-\frac{1}{2}\tr_{\wh{g}} h}-\delta^*_{\wh{\mathfrak{g}}}\delta_{\wh{\mathfrak{g}}}\mathfrak{h},\mathfrak{h}}_{L^2_{\wh\mfg}},
            \end{equation}
            where $v$ is the variation of the minimiser $f_{\wh{g}}$ and solves \eqref{auxiliary_EL_linearised}. In \eqref{aux_mu_Hessian} integration is over $M$, but with respect to the volume element of $\wh\mfg$ -- that is to say, $L^2_{\wh\mfg}(M):=L^2_{\wh\varphi}(M)$.
        \end{proposition}

        \begin{proof}
            To simplify notation, we will use $g$ for both $\wh{\mathfrak{g}}$ and $\wh{g}$, distinguishing them simply by their index. A similar approach is applied to $\mathfrak{h}$ and $h$. Integers $i,j,m,l$ will run from $1$ to $n$, thus indexing only the coordinates on $M$. We commence by calculating the $ij$-components of $\Delta_{E,\wh{\mathfrak{g}}}\mathfrak{h}$:
            \begin{align*}
                \p{\nabla^*\nabla }_{\wh{\mathfrak{g}}}\mathfrak{h}_{ij}={}&-g^{lm}\nabla_l\p{\nabla h}_{mij}-g^{\tau \tau}\nabla_\tau \p{\nabla h}_{\tau ij}\\
                ={}&\nabla^*\nabla h_{ij}+g^{\tau \tau}\p{\Gamma_{\tau \tau}^m\nabla_mh_{ij}+\Gamma_{\tau i}^\tau \nabla_\tau h_{\tau j}+\Gamma_{\tau j}^\tau \nabla_\tau h_{i\tau}}\\
                ={}& \nabla^*\nabla h_{ij}+\nabla^m\wh{\varphi}\nabla_mh_{ij}-e^{2\wh{\varphi}}\p{\Gamma_{\tau i}^\tau\Gamma_{\tau \tau}^mh_{mj}+2\Gamma_{\tau i}^\tau \Gamma_{\tau j}^\tau h_{\tau \tau}+\Gamma_{\tau j}^\tau\Gamma_{\tau \tau}^mh_{im}}\\
                ={}& \nabla^*\nabla h_{ij}+\nabla_{\nabla \wh{\varphi}}h_{ij}+d\wh{\varphi}_ig^{ml}d\wh{\varphi}_l h_{mj}-2e^{2\wh{\varphi}}d\wh{\varphi}_id\wh{\varphi}_jh_{\tau \tau}+g^{ml}d\wh{\varphi}_lh_{mi}d\wh{\varphi}_j\\
                ={}&  \nabla^*\nabla h_{ij}+\nabla_{\nabla \wh{\varphi}}h_{ij}+d\wh{\varphi}_i\p{h\cdot d\wh{\varphi}}_j+4\psi d\wh{\varphi}_id\wh{\varphi}_j+\p{h\cdot d\wh{\varphi}}_id\wh{\varphi}_j,\\
                (\overset{\;\circ}{R}_{\wh{\mathfrak{g}}}\mathfrak{h})_{ij}={}&R_{i\;\; \;\;\;j}^{\;\; lm}h_{lm}+R_{i\;\; \;\;j}^{\;\ \tau \tau}h_{\tau \tau}\\
                ={}&(\overset{\;\circ}{R}h)_{ij}-e^{2\wh{\varphi}}\p{\partial_i\Gamma_{\tau j}^\tau +\Gamma_{\tau i}^\tau\Gamma_{\tau j}^\tau-\Gamma_{ij}^m\Gamma_{\tau m}^\tau}h_{\tau \tau}\\ 
                ={}&(\overset{\;\circ}{R}h)_{ij}+e^{2\wh{\varphi}}\p{\nabla^2_{ij}\wh{\varphi}-d\wh{\varphi}_id\wh{\varphi}_j}h_{\tau \tau}\\ 
                ={}& (\overset{\;\circ}{R}h)_{ij}-2\psi\nabla^2_{ij}\wh{\varphi}+2\psi d\wh{\varphi}_id\wh{\varphi}_j.
            \end{align*}

            Which means we can express the $ij$-components of the Einstein operator as
                $$\Delta_{E,\wh{\mathfrak{g}}}\mathfrak{h}_{ij}=\Delta_E h_{ij}+\nabla_{\nabla \wh{\varphi}}h_{ij}+d\wh{\varphi}_i\p{h\cdot d\wh{\varphi}}_j+\p{h\cdot d\wh{\varphi}}_id\wh{\varphi}_j+4\psi \nabla^2_{ij}\wh{\varphi}.$$
        
            Next we calculate the $\tau \tau$-component of the Einstein operator
            \begin{align*}
                \p{\nabla^*\nabla }_{\wh{\mathfrak{g}}}\mathfrak{h}_{\tau \tau}={}&-g^{\tau \tau}\nabla_\tau\p{\nabla h}_{\tau \tau \tau}-g^{ij}\nabla_i\p{\nabla h}_{j\tau \tau}\\ 
                ={}&g^{\tau \tau}\Gamma_{\tau \tau}^{i}\p{\partial_i h_{\tau \tau}-2\Gamma_{\tau \tau}^kh_{ik}-4\Gamma_{i\tau}^\tau h_{\tau \tau}}\\&-g^{ij}\p{\nabla^2_{ij}(h_{\tau \tau})+2h_{\tau \tau}\nabla^2_{ij}\wh{\varphi}+2\Gamma_{i\tau}^\tau\Gamma_{j\tau}^\tau h_{\tau \tau}-4\Gamma_{i\tau}^\tau\partial_jh_{\tau \tau}}\\
                ={}&-e^{-2\wh{\varphi}}g^{ij}\partial_j\wh{\varphi}\p{2\partial_j\psi+2(h\cdot d\wh{\varphi})_i+8\psi\partial_i\wh{\varphi}}\\
                &-g^{ij}\p{\nabla^2_{ij}\p{h_{\tau \tau}}+2h_{\tau \tau}\nabla_{ij}^2\wh{\varphi}+4\partial_i\wh{\varphi}\partial_j(h_{\tau \tau})+4h_{\tau \tau}\partial_i\wh{\varphi}\partial_j\wh{\varphi}}\\
                ={}&-e^{-2\wh{\varphi}}g^{ij}\p{2d\wh{\varphi}_id\psi_j+4\psi d\wh{\varphi}_id\wh{\varphi}_j+2d\wh{\varphi}_i h_{jm}g^{ml}d\wh{\varphi}_l-2\nabla^2_{ij}\psi}\\
                ={}&-e^{-2\wh{\varphi}}\p{2\g{d\wh{\varphi}}{d\psi}+4\psi \abs{d\wh{\varphi}}^2+2h\p{d\wh{\varphi},d\wh{\varphi}}+2\Delta \psi},\\
                (\overset{\;\circ}{R}_{\wh{\mathfrak{g}}}\mathfrak{h})_{\tau \tau}={}& g^{ij}g^{lm}R_{\tau il\tau}h_{jm}\\={}&g^{ij}g^{lm}g_{\tau \tau}g_{\tau \tau}R_{i\;\;\;\;l}^{\;\; \tau \tau}h_{jm}\\ 
                ={}&e^{-2\wh{\varphi}}g^{ij}g^{lm}\p{\nabla^2_{il}\wh{\varphi}-d\wh{\varphi}_i d\wh{\varphi}_l}h_{jm}\\
                ={}&e^{-2\wh{\varphi}}\g{\nabla^2\wh{\varphi}-d\wh{\varphi}\otimes d\wh{\varphi}}{h}.
            \end{align*}
        
            These combine to produce
                $$\Delta_{E,\wh{\mathfrak{g}}} \mathfrak{h}_{\tau \tau}=-2e^{-2\wh{\varphi}}\p{\Delta_{\wh{\varphi}} \psi+2\psi\abs{d\wh{\varphi}}^2+\psi \g{\nabla^2\wh{\varphi}}{h}},$$
            which, together with the $ij$-components, gives the claimed formula for $\Delta_{E,\wh{\mathfrak{g}}} \mathfrak{h}$. The expression \eqref{aux_mu_Hessian} is the resulting simplification of \eqref{entropy-hessian}, with the added use of 
                $$\nabla^2_{\wh{\mathfrak{g}}}f=-\g{df}{d\wh{\varphi}}_{\wh{g}}e^{-2\wh{\varphi}}\; d\tau^2+\nabla^2_{\wh{g}}f$$
            applied to $f=v-\frac{1}{2}\tr_{\wh{g}}h$. 
        \end{proof}

        \begin{remark}\label{DCSC_second_var} 
            As was remarked upon in the previous subsection (see \ref{DCSC_crit}), metric pairs $\mathfrak{g}=(g,\varphi)$ satisfying 
                $$\scal_g=-n(n-1)\qquad \qquad \& \qquad \qquad \Delta\varphi+\abs{d\varphi}^2=n,$$
            are critical points for the entropy restricted to directions preserving the conformal class of $g$. \\
    
            An additional property of such a pair is the fact that the second variation in the direction of $\mathfrak{h}=-2\psi e^{-2\varphi}\; d\tau^2+\eta g$, is the same as that of static metrics: \eqref{aux_mu_Hessian}. The argument is simple, as the only time during the proof of the second variation we used the static equations was in \eqref{where_we_used_static}. While this simplification does not hold outright for the non-static pair $(g,\varphi)$, it does hold ``weakly'' when $h=\eta g$:
            \begin{align*}
                \frac{1}{2}\g{\eta g\times \p{\Ric_g+ng+\nabla^2\varphi}}{\eta g}={}& \eta^2 \tr\p{\Ric_g+ng+\nabla^2\varphi}\\={}& \eta^2\p{\scal_g+n^2-\Delta \varphi}\\={}&\eta^2\abs{d\varphi}^2\\
                ={}&\g{\p{\eta g\cdot d\varphi}\otimes d\varphi}{\eta g}.\end{align*}
        \end{remark}

        We end this section by estimating the second variation of the entropy along specific curves of auxiliary metrics and pairs $\p{g,\varphi}$. The following proposition even adds a third type of semi-critical point for the entropy.
    
        \begin{proposition}\label{stability_op_prop} 
            \begin{enumerate}
            \item Let $\mathfrak{g}=\p{g,\varphi}$ be a constant scalar curvature metric, $\scal_{\mathfrak{g}}=-n(n+1)$, then it is a critical point for the entropy restricted to its conformal class, and 
            \begin{align*}
                \left.\frac{d^2}{ds^2}\right|_{s=0}\mu_{\mathrm{AH}}\p{\p{1+s\chi}\mathfrak{g}}\leq -\frac{1}{4 n}\norm{\chi \mathfrak{g}}^2_{H^1_{\mathfrak{g}}},
            \end{align*}
            for every $\chi\in H^\ell_\varphi\p{M}= H^\ell_{\mathfrak{g}}\p{M}$.
            
            \item Let $(g,\varphi)$ be a pair, such that $\scal_g=-n(n-1)$ and $\Delta \varphi+\abs{d\varphi}^2=n$. Then $\p{g,\varphi}$ is a critical point for the entropy restricted to variations preserving the conformal class of $g$. Furthermore, for every $\eta,\psi\in H^\ell_\varphi\p{M}$,
            \begin{align}\label{sec_var_nonpos}
                \left.\frac{d^2}{ds^2}\right|_{s=0}\mu_{\mathrm{AH}}\p{\p{1+s\eta}g,\varphi+s\psi} \leq -C_{g,\varphi}\norm{\p{\eta g,\psi}}^2_{H^1_{\mathfrak{g }}},
            \end{align}
            for some $C_{g,\varphi}>0$.
            \end{enumerate}
        \end{proposition}
            
        \begin{proof}
            \textit{(1.)} Consider the curve of auxiliary metrics $\mathfrak{g}_s=(1+s\chi)\mathfrak{g}$. Recall that Proposition \ref{klaus_konj} says that the minimiser $f_s$ agrees with $\varphi_s$ at $s=0$. Thus, from Corollary \ref{first_var_crit}, we see that $\mathfrak{g}$ is conformally critical:   
            \begin{align*}
                \left.\frac{d}{ds}\right|_{s=0}\mu_{\mathrm{AH}}\p{(1+s\chi)\mathfrak{g}}={}&-\frac{1}{2 n}\int_M\g{\Ric_{\mathfrak{g}}+n\mathfrak{g}+\nabla^2_{\mathfrak{g}}\p{f-\varphi}}{\chi\mathfrak{g}}_{\mathfrak{g}}\; dV_{\mathfrak{g}}\\ ={}& -\frac{1}{2 n}\int_M\p{\scal_{\mathfrak{g}}+n(n+1)}\chi\; dV_{\mathfrak{g}}\\ ={}&0,
            \end{align*}
            where, as we usually do, view $dV_\mfg=dV_\varphi$ as a volume form on $M$.\\
                
            Using that $\mathfrak{g}$ is CSC and that $f_{0}=f_{\mfg}=\varphi$, it is a simple task to check that the second variation in a conformal direction springs from the first variation of the $L^2$-gradient:   
            \begin{align*}
                \left.\frac{d^2}{ds^2}\right|_{s=0}&\mu_{\mathrm{AH}}\p{(1+s\chi)\mathfrak{g}}\\={}&-\frac{1}{2 n}\int_M\g{\left.\frac{d}{ds}\right|_{s=0}\p{\Ric_{\mathfrak{g}_s}+n\mathfrak{g}_s+\nabla^2_{\mathfrak{g}_s}\p{f_s-\varphi_s}}}{\chi\mathfrak{g}}_{\mathfrak{g}}\; dV_{\mathfrak{g}}\\
                ={}&-\frac{1}{2 n}\int_M\g{\frac{1}{2}\p{\Delta_{\mathfrak{g}}\chi} \mathfrak{g} - \frac{n-1}{2}\nabla^2_{\mathfrak{g}}\chi+n\chi\mathfrak{g}+\nabla^2_{\mathfrak{g}}\p{f'-\frac{\chi}{2}}}{\chi\mathfrak{g}}_{\mathfrak{g}}\; dV_{\mathfrak{g}}\\
                ={}&-\frac{1}{2 n}\int_M\frac{n+1}{2}\chi\Delta_{\mathfrak{g}}\chi+n(n+1)\chi^2-\chi\Delta_{\mathfrak{g}}\p{f'-\frac{n\chi}{2}}\; dV_{\mathfrak{g}}.
            \end{align*}
                    
            The form of the final term of the integrand is designed to allow for the use of the linearised Euler--Lagrange equation \eqref{auxiliary_EL_linearised}:
                $$f'-\frac{n\chi}{2}=\frac{1}{2}\p{\Delta_{\mathfrak{g}}+n}^{-1}\delta_{\mathfrak{g}}\delta_{\mathfrak{g}}\p{\chi\mathfrak{g}}=-\frac{1}{2}\p{\Delta_{\mathfrak{g}}+n}^{-1}\Delta_{\mathfrak{g}}\chi.$$
                
            Since $\p{\Delta_{\mathfrak{g}}+n}^{-1}$ is a positive operator, we have:
            \begin{align*}
                \left.\frac{d^2}{ds^2}\right|_{s=0}&\mu_{\mathrm{AH}}\p{(1+s\chi)\mathfrak{g}}\\={}&-\frac{1}{2n}\int_M\frac{n+1}{2}\abs{\nabla\chi}^2_{\mathfrak{g}}+n(n+1)\chi^2+\frac{1}{2}\p{\Delta_{\mathfrak{g}}\chi}\p{\Delta_{\mathfrak{g}}+n}^{-1}\p{\Delta_{\mathfrak{g}}\chi}\; dV_{\mathfrak{g}}\\\leq{} &-\frac{1}{2 n}\int_M\frac{n+1}{2}\abs{\nabla\chi}^2_{\mathfrak{g}}+n(n+1)\chi^2\; dV_{\mathfrak{g}}\\={}&-\frac{1}{2 n}\int_M\frac{1}{2}\abs{\nabla_\mathfrak{g}\p{\chi\mathfrak{g}}}^2_{\mathfrak{g}}+n\abs{\chi\mathfrak{g}}^2_\mathfrak{g}\; dV_{\mathfrak{g}}.
            \end{align*}
                
            The inequality in first part of proposition is then obtained by using $n \geq 3 > \frac{1}{2}$.\\
                
            \textit{(2.)} Let $(g,\varphi)$ be a pair as in the statement of the proposition and consider a perturbation of the auxiliary metric  $\mathfrak{g}=e^{-2\varphi}\; d\tau^2+g$, of the form $\mathfrak{h}=-2\psi e^{-2\varphi}\; d\tau^2+\eta g$. It has already been proven that $(g,\varphi)$ is critical in this direction (see Remark \ref{DCSC_crit}). Before we can show \eqref{sec_var_nonpos}, we prove two claims.
                
            \begin{claim}
                The variation $\mathfrak{h}=\p{h,\psi}$ splits uniquely as $\mathfrak{h}=\mathfrak{y}+\chi \mathfrak{g}$, with $D_{\mathfrak{g}}\scal\p{\mathfrak{y}}=0$ and $\chi\in H^\ell_{\mfg}(M)$. Furthermore, $\mathfrak{y}$ integrates to a curve of CSC metrics. That is, we find a curve of metrics $\mathfrak{g}_s=e^{-2\varphi_s}\; d\tau^2+g_s$ with $\scal_{\mathfrak{g}_s}=-n(n+1)$ and $\left.\frac{d}{ds}\right|_{s=0}\mathfrak{g}_s=\mathfrak{y}$.
            \end{claim}

            \begin{proof}
                To prove the claim, note first that 
                \begin{equation*}
                    D_{\mathfrak{g}}\scal\p{\mathfrak{h}}=\p{\Delta_{\mathfrak{g}}+n}\tr_{\mathfrak{g}}\mathfrak{h}+\delta_{\mathfrak{g}}\delta_{\mathfrak{g}}\mathfrak{h},
                \end{equation*}
                where we used \eqref{auxiliary_Ric} to compute that 
                    $$\g{\Ric_\mathfrak{g}}{\mathfrak{h}}_{\mathfrak{g}}=2n\psi+\p{\scal_g-n}\eta=-n\p{n\eta-2\psi}=-n\tr_{\mathfrak{g}}\mathfrak{h}.$$
                    
                Now the claimed decomposition of $\mfh$ is equivalent to the equation             
                \begin{align*}
                    \p{\Delta_{\mathfrak{g}}+n}\tr_{\mathfrak{g}}\mathfrak{h}+\delta_{\mathfrak{g}}\delta_{\mathfrak{g}}\mathfrak{h}=n(\Delta_{\mathfrak{g}}+n+1)\chi,
                \end{align*}
                which is uniquely solvable for a function $\chi\in H^{\ell}_\varphi\p{M}\subset H^\ell_{\mathfrak{g}}\p{\mathfrak{M}}$, provided that $\mathfrak{h}\in H^\ell_{\mathfrak{g}}\p{S^2\mathfrak{M}}$. This yields the claimed decomposition.\\
                    
                Next, we observe that $D_{\mathfrak{g}}\scal$ is surjective on this class of perturbations. Therefore, the set
                    $$\mathcal{A}:=\left\{\mathfrak{g}=e^{-2\varphi}\; d\tau^2+g\, :\; \scal_{\mathfrak{g}}=-n(n+1)\right\}$$
                is a smooth manifold satisfying $T_{\mathfrak{g}}\mathcal{A}=\mathrm{ker}(D_{\mathfrak{g}}\scal)$. This establishes the second part of the claim.
                \end{proof}  
        
                Next, we claim the following:
            \begin{claim}
                \begin{align*}
                   \left.\frac{d^2}{drds}\right|_{r,s=0}\mu_{\mathrm{AH}}\p{\mathfrak{g}+s\chi\mathfrak{g}+r\mathfrak{y}}=0
                \end{align*}                            
                for every  $\mathfrak{y}\in T_{\mathfrak{g}}\mathcal{A}$ and $\chi\in H^{\ell}_\varphi\p{M}\subset H^\ell_{\mathfrak{g}}(\mathfrak{M})$.
            \end{claim}

            \begin{proof}
                By the previous claim, and because $\mathfrak{g}$ is a semi-critical metric, we can replace $\mathfrak{g}+s\chi\mathfrak{g}+r\mathfrak{y}$ by the $2$-parameter family $(1+s\chi)\mathfrak{g}_r$, where $\mathfrak{g}_r$ is a curve in $\mathcal{A}$ with $\left.\frac{d}{dr}\right|_{r=0}\mathfrak{g}_r=\mathfrak{y}$. By the first part of the proposition,               
                \begin{align*}
                    \left.\frac{d}{ds}\right|_{s=0}\mu_{\mathrm{AH}}\p{\mathfrak{g}_r+s\chi\mathfrak{g}_r}=0
                \end{align*}                    
                for all $r$, and differentiating this identity with respect to $r$ yields
                \begin{align*}
                    0=\left.\frac{d^2}{drds}\right|_{r,s=0}\mu_{\mathrm{AH}}\p{\mathfrak{g}_r+s\chi\mathfrak{g}_r}=    \left.\frac{d^2}{drds}\right|_{r,s=0}\mu_{\mathrm{AH}}\p{\mathfrak{g}+s\chi\mathfrak{g}+r\mathfrak{y}},
                \end{align*}                    
                as desired. 
            \end{proof}
                 
            Thus, the second variation is orthogonal with respect to the splitting $ H^{\ell}_{\varphi}\p{M}\cdot\mathfrak{g}\oplus T_{\mathfrak{g}}\mathcal{A}$. For $\mathfrak{h}=\chi\mathfrak{g}\in H^{\ell}_{\varphi}\p{M}\cdot\mathfrak{g}$, the desired estimate follows already from the first part of the proposition. It therefore suffices to establish the claimed estimate for tensors $\mathfrak{h}\in T_{\mathfrak{g}}\mathcal{A}$. That is, for those tensors $\mathfrak{h}$ which satisfy 
            \begin{equation}\label{aux_linearised_scalar=0}
                0=D_{\mathfrak{g}}\scal\p{\mathfrak{h}}=\p{\Delta_{\mathfrak{g}}+n}\tr_{\mathfrak{g}}\mathfrak{h}+\delta_{\mathfrak{g}}\delta_{\mathfrak{g}}\mathfrak{h}.
            \end{equation}
                
            When combined with the linearised Euler--Lagrange equation \eqref{auxiliary_EL_linearised}, scalar curvature preservation is revealed to be equivalent to $\psi=v$, as                
            \begin{align*}
                0 ={}& 2\p{\Delta_{\mathfrak{g}}+n}\p{v-\frac{1}{2}\tr_gh} - \delta_{\mathfrak{g}}\delta_{\mathfrak{g}}\mathfrak{h}+\g{\Ric_{\mfg}+n\mfg}{\mfh}\\
                ={}&2\p{\Delta_{\mathfrak{g}}+n}\p{v-\psi-\frac{1}{2}\tr_\mathfrak{g}\mathfrak{h}} - \delta_{\mathfrak{g}}\delta_{\mathfrak{g}}\mathfrak{h}\\
                ={}& 2\p{\Delta_{\mathfrak{g}}+n}\p{v-\psi},
            \end{align*}                
            and $\Lap_{\mathfrak{g}} + n$ has trivial kernel. Using this and \eqref{aux_trace}, we find that \eqref{aux_mu_Hessian} becomes
            \begin{align*}
                D^2_\mathfrak{g}\mu_{\mathrm{AH}}\p{\mathfrak{h}}=-\frac{1}{4 n}\p{\Delta_{E,\mathfrak{g}}\mathfrak{h}-\nabla^2_{\mathfrak{g}}\tr_{\mathfrak{g}}\mathfrak{h}-2\delta_{\mathfrak{g}}^*\delta_{\mathfrak{g}}\mathfrak{h},\mathfrak{h}}_{L^2_\mathfrak{g}}.
            \end{align*}
                    
            Recall that \eqref{aux_mu_Hessian} holds for the non-static pair $(g,\varphi)$ by Remark \ref{DCSC_second_var}. The explicit form of $\mathfrak{h}$ may be used to compute the curvature term in the definition of the Einstein operator:
            \begin{align*}
                \overset{\;\circ}{R}_{\mathfrak{g}}\mathfrak{h} &= \overset{\;\circ}{R}_g \left(\eta g\right) - 2\psi \left(\nabla^2 \varphi - d\varphi \otimes d\varphi\right) + e^{-2\varphi} \left<\nabla^2 \varphi - d\varphi \otimes d\varphi, \eta g\right> d\tau^2\\
                &= \eta \Ric_g - 2\psi \left(\nabla^2 \varphi - d\varphi \otimes d\varphi\right) - \eta e^{-2\varphi} \left(\Lap \varphi - \abs{d \varphi}^2\right) d\tau^2 \\
                &= \eta\Ric_{g}-2\psi\p{\nabla^2_g\varphi+d\varphi\otimes d\varphi} -n\eta e^{-2\varphi}d\tau^2,
            \end{align*}
            which means 
            \begin{align*}
                \g{\overset{\;\circ}{R}_{\mathfrak{g}}\mathfrak{h}}{\mathfrak{h}}_\mathfrak{g}=\eta^2\scal_g+2\psi \eta \p{\Delta_g\varphi+\abs{d\varphi}^2}+2n\psi\eta= \abs{\mathfrak{h}}^2_{\mathfrak{g}}-\abs{\tr_{\mathfrak{g}}\mathfrak{h}}^2_{\mathfrak{g}},
            \end{align*}
            where we used       
            \begin{align*}
                \abs{\mathfrak{h}}^2_{\mathfrak{g}} &= 4\psi^2 + n \eta^2,\\
                \abs{\tr_{\mathfrak{g}}\mathfrak{h}}^2_{\mathfrak{g}} &= 4\psi^2 - 4n\psi \eta + n^2 \eta^2.
            \end{align*}                
            Using Integration by parts and \eqref{aux_linearised_scalar=0} produces 
            \begin{align*}
                \p{\nabla^2_{\mathfrak{g}}\tr_{\mathfrak{g}}\mathfrak{h},\mathfrak{h}}_{L^2_{\mathfrak{g}}}=\p{\tr_\mathfrak{g}\mathfrak{h},\delta_{\mathfrak{g}}\delta_{\mathfrak{g}}\mathfrak{h}}_{L^2_\mathfrak{g}}=-\norm{\nabla_{\mathfrak{g}}\tr_\mathfrak{g}\mathfrak{h}}^2_{L^2_\mathfrak{g}}-n\norm{\tr_\mathfrak{g}\mathfrak{h}}^2_{L^2_\mathfrak{g}}.
            \end{align*}                
            In the end, the second variation can be written as
                $$D^2_{\mathfrak{g}}\mu_{\mathrm{AH}}\p{\mathfrak{h}}=-\frac{1}{4 n}\p{\norm{\nabla_{\mathfrak{g}}\mathfrak{h}}^2_{L^2_{\mathfrak{g}}}+\norm{\nabla_{\mathfrak{g}}\tr_\mathfrak{g}\mathfrak{h}}^2_{L^2_\mathfrak{g}}-2\norm{\delta_\mathfrak{g}\mathfrak{h}}^2_{L^2_\mathfrak{g}}+(n+2)\norm{\tr_\mathfrak{g}\mathfrak{h}}^2_{L^2_\mathfrak{g}}-2\norm{\mathfrak{h}}^2_{L^2_\mathfrak{g}}}.$$
                
            The final part of the proof is most easily achieved by again decomposing $\mfh$, this time by trace, $\mathfrak{h}=\theta\mathfrak{g}+\mathfrak{K}$, where
                $$\theta=\frac{1}{n+1}\tr_{\mathfrak{g}}\mathfrak{h}=\frac{n\eta-2\psi}{n+1}$$
            and
                $$\mathfrak{K}=\mathfrak{h}-\theta\mathfrak{g}=-n\omega e^{-2\psi}\; d\tau^2+\omega g, \qquad \qquad \omega=\frac{\eta-2\psi}{n+1}.$$
            It will be useful to keep the following identities in mind:   
            \begin{align*}
                \abs{\mathfrak{K}}^2 &= n\left(n+1\right) \omega^2\\
                \left(\nabla_{\mathfrak{g}} \mathfrak{K}\right)_{\tau i \tau} &= \left(\nabla_{\mathfrak{g}} \mathfrak{K}\right)_{\tau i \tau} = -\left(n+1\right)\omega e^{-2\varphi} d\varphi_i\\
                \left(\nabla_{\mathfrak{g}} \mathfrak{K}\right)_{i \tau \tau} &= -n e^{2\varphi} d\omega_i\\
                \left(\nabla_{\mathfrak{g}} \mathfrak{K}\right)_{ijk} &= -d\omega_i g_{jk}\\
                \abs{\nabla_{\mathfrak{g}} \mathfrak{K}}^2 &= 2\left(n+1\right)^2 \omega^2 \abs{d\varphi}^2 + n \left(n+1\right) \abs{d \omega}^2\\
                \left(\delta_{\mathfrak{g}} \mathfrak{K}\right)_i &= \left(n+1\right)\omega d\varphi_i - d\omega_i\\
                \abs{\nabla_{\mathfrak{g}} \mathfrak{K}}^2 - 2\abs{\delta_{\mathfrak{g}} \mathfrak{K}}^2 &= \left(n^2 + 2n - 2\right)\abs{d \omega}^2 + 4\left(n+1\right) \omega\left<d\omega,d\varphi\right>\\
                2\left(d\theta, \delta_{\mathfrak{g}} \mathfrak{K}\right)_{L^2_{\mathfrak{g}}} &= -2n\norm{d\theta}_{L^2_\mfg}^2 - 2n\left(n+1\right)\norm{\theta}_{L^2_\mfg}^2. 
            \end{align*}    
            We also note that, due to \eqref{aux_linearised_scalar=0}, we have
            \begin{align*}
                0 &= \Lap_{\mathfrak{g}} \tr_{\mathfrak{g}} \mathfrak{h} + n\tr_{\mathfrak{g}} \mathfrak{h} + \delta_{\mathfrak{g}} \delta_{\mathfrak{g}} \mathfrak{h}\\
                &= \left(n+1\right) \Lap_{\mathfrak{g}} \theta + n \left(n+1\right) \theta - \Lap_{\mathfrak{g}} \theta + \delta_{\mathfrak{g}} \delta_{\mathfrak{g}} \mathfrak{K}\\
                &= n\Lap_{\mathfrak{g}} \theta + n\left(n+1\right) \theta + \delta_{\mathfrak{g}} \delta_{\mathfrak{g}} \mathfrak{K}.
            \end{align*}
                
            Using the above identities and the decomposition of $\mathfrak{h}$ into full trace and trace--free parts, one may show that            
            \begin{align*}
                &\norm{\nabla_{\mathfrak{g}}\mathfrak{h}}^2_{L^2_{\mathfrak{g}}}-2\norm{\delta_{\mathfrak{g}}\mathfrak{h}}^2_{L^2_{\mathfrak{g}}}\\
                &= \norm{\nabla_{\mathfrak{g}} \left(\theta \mathfrak{g}\right) + \nabla_{\mathfrak{g}} \mathfrak{K}}^2_{L^2_{\mathfrak{g}}} - 2\norm{\delta_{\mathfrak{g}} \left(\theta \mathfrak{g}\right) + \delta_{\mathfrak{g}} \mathfrak{K}}^2_{L^2_{\mathfrak{g}}}\\
                &= \norm{\nabla_\mathfrak{g}\p{\theta\mathfrak{g}}}_{L^2_\mathfrak{g}}^2+\norm{\nabla_\mathfrak{g}\mathfrak{K}}_{L^2_\mathfrak{g}}^2 -2\norm{\delta_{\mathfrak{g}}\p{\theta\mathfrak{g}}}^2_{L^2_{\mathfrak{g}}}-2\norm{\delta_{\mathfrak{g}}\mathfrak{K}}^2_{L^2_{\mathfrak{g}}} + 2\left(\nabla_{\mathfrak{g}} \left(\theta \mathfrak{g}\right), \nabla_{\mathfrak{g}} \mathfrak{K}\right)_{L^2_{\mathfrak{g}}}\\
                &~~~~~~ - 4\left(\delta_{\mathfrak{g}}\left(\theta \mathfrak{g}\right), \delta_{\mathfrak{g}} \mathfrak{K}\right)_{L^2_{\mathfrak{g}}}\\
                &= \left(n+1\right)\norm{d\theta}_{L^2_\mathfrak{g}}^2+\norm{\nabla_\mathfrak{g}\mathfrak{K}}_{L^2_\mathfrak{g}}^2-2\norm{d\theta}^2_{L^2_{\mathfrak{g}}}-2\norm{\delta_{\mathfrak{g}}\mathfrak{K}}^2_{L^2_{\mathfrak{g}}}+2\p{d\theta,\delta_{\mathfrak{g}}\mathfrak{K}}_{L^2_{\mathfrak{g}}}\\
                &= \left(n-1\right)\norm{d\theta}^2_{L^2_{\mathfrak{g}}} + 2\left(d\theta, \delta_{\mathfrak{g}} \mathfrak{K}\right)_{L^2_{\mathfrak{g}}} + \left(n^2 + n -2\right)\norm{d \omega}^2_{L^2_{\mathfrak{g}}} + 4\left(n+1\right)\int_M \omega\left<d\omega, d\varphi\right> dV_{\mathfrak{g}}\\
                &= -(n+1)\norm{d\theta}^2_{L^2_{\mathfrak{g}}}-2n(n+1)\norm{\theta}_{L^2_{\mathfrak{g}}}^2+(n-1)(n+2)\norm{d\omega}^2_{L^2_{\mathfrak{g}}}+2n(n+1)\norm{\omega}^2_{L^2_{\mathfrak{g}}}.
            \end{align*}                
            Note that the final line used the identity for $\left(d\theta, \delta_{\mathfrak{g}} \mathfrak{K}\right)_{L^2_{\mathfrak{g}}}$ which we stated earlier, as well as   
            \begin{equation*}
                \int_M \omega \left<d\omega, d\varphi\right> dV_{\mathfrak{g}} = \frac{1}{2}\int_M \left<d\omega^2, d\varphi\right> dV_{\mathfrak{g}} = \frac{1}{2}\int_M \omega^2 \Lap_\varphi \varphi \; dV_{\mathfrak{g}} = \frac{n}{2} \norm{\omega}^2_{L^2_{\mathfrak{g}}}.
            \end{equation*}                
            The final term above will be cancelled by a contribution from                 $$-2\norm{\mathfrak{h}}^2_{L^2_{\mathfrak{g}}}=-2(n+1)\norm{\theta}^2_{L^2_{\mathfrak{g}}}-2n(n+1)\norm{\omega}^2_{L^2_{\mathfrak{g}}}.$$    
           Finally, using 
            \begin{align*}
                \norm{\nabla_{\mathfrak{g}} \tr_{\mathfrak{g}}\mathfrak{h}}^2_{L^2_{\mathfrak{g}}} &= \left(n+1\right)^2 \norm{d\theta}^2_{L^2_{\mathfrak{g}}},\\
                \norm{\tr_{\mathfrak{g}} \mathfrak{h}}^2_{L^2_{\mathfrak{g}}} &= \left(n+1\right)^2 \norm{\theta}^2_{L^2_{\mathfrak{g}}},
            \end{align*}
            the full second variation can be written as
            \begin{equation}\label{second_var_trace_tracefree}
                D^2_{\mathfrak{g}}\mu_{\mathrm{AH}}\p{\mathfrak{h}}=-\frac{1}{4 n}\p{n(n+1)\norm{d\theta}^2_{L^2_{\mathfrak{g}}}+n(n+1)^2\norm{\theta}^2_{L^2_{\mathfrak{g}}}+(n-1)(n+2)\norm{d\omega}^2_{L^2_{\mathfrak{g}}}}.%\\ \leq{} & -\frac{1}{4(n+1)}\norm{\tr_\mathfrak{g}\mathfrak{h}}^2_{H^1_{\mathfrak{g}}}
            \end{equation}             
            Using the expression for $\abs{\nabla_{\mathfrak{g}} \mathfrak{K}}$ from above, we have   
            \begin{align*}
                \norm{\nabla_{\mathfrak{g}} \mathfrak{K}}^2_{L^2_{\mathfrak{g}}} \leq 2\left(n+1\right)^2 \sup_M \abs{d \varphi}\norm{\omega}^2_{L^2_\mathfrak{g}} + n\left(n+1\right)\norm{d\omega}^2_{L^2_{\mathfrak{g}}}.          
            \end{align*}    
             \cite[Proposition F]{Lee} provides a spectral gap $\Lap_{\mathfrak{g}} \geq \Lambda > 0$, where $\Lambda=\frac{\p{n-1}^2}{4}$. This, in turn, gives the Poincar\'e inequality $\norm{\omega}^2_{L^2_\mfg} \leq \frac{1}{\Lambda}\norm{d\omega}^2_{L^2_\mfg}$. Combining this with the bound on the $L^2$-norm of $\nabla\mathfrak{K}$ from above, we have 
            \begin{align*}
                \norm{\nabla_{\mathfrak{g}} \mathfrak{K}}^2_{L^2_{\mathfrak{g}}} \leq \left(\frac{2\left(n+1\right)^2 \sup_M \abs{d \varphi} + \Lambda n\left(n+1\right)}{\Lambda}\right) \norm{d\omega}^2_{L^2_{\mathfrak{g}}}.
            \end{align*}    
            This, when combined with \eqref{second_var_trace_tracefree}, yields   
            \begin{align*}
                &D^2_\mathfrak{g}\mu_{\mathrm{AH}}\p{\mathfrak{h}}\\
                &\leq -\frac{1}{4 n}\left(n\left(n+1\right)\norm{\theta}^2_{H^1_{\mathfrak{g}}} + \frac{\left(n-1\right)\left(n+2\right)\Lambda}{2\left(n+1\right)^2 \sup_M \abs{d \varphi} + \Lambda n\left(n+1\right)}\norm{\nabla_{\mathfrak{g}}\mathfrak{K}}^2_{L^2_{\mathfrak{g}}}\right)\\
                &\leq -\frac{1}{4 n}\left(n\left(n+1\right)\norm{\theta}^2_{H^1_{\mathfrak{g}}} + \frac{\left(n-1\right)\left(n+2\right)\Lambda}{4\left(n+1\right)^2 \sup_M \abs{d \varphi} + 2\Lambda n\left(n+1\right)}\left(\norm{\nabla_{\mathfrak{g}}\mathfrak{K}}^2_{L^2_{\mathfrak{g}}} + \wt{\Lambda} \norm{\mathfrak{K}}^2_{L^2_{\mathfrak{g}}}\right)\right)\\
                &\leq -C_{g,\varphi} \norm{\mathfrak{h}}^2_{H^1_{\mathfrak{g}}},
            \end{align*}
            where the penultimate line follows from applying a spectral gap argument to the trace--free $\mathfrak{K}$. Specifically, using \cite[Proposition E]{Lee} to guarantee that $\wt{\Lambda} \norm{\mathfrak{K}}^2_{L^2_{\mathfrak{g}}} \leq \norm{\nabla_{\mathfrak{g}} \mathfrak{K}}^2_{L^2_{\mathfrak{g}}}$, for $\wt\Lambda=\frac{\p{n-1}^2}{4}+2$.
        \end{proof}

\section{A Local Positive Mass Theorem}\label{sec_local_pos_mass}
    In this section, we prove a local positive mass theorem which, in essence, shows that positivity of the weighted mass
    \begin{equation*}
        m_{\wh{g}}\p{h,\varphi}=-\lim_{R\to \infty}\int_{\partial B_R}\g{\delta_{\wh{g}}h+d\tr_{\wh{g}}h}{\wh{\nu}}_{\wh{g}}+\tr_{\wh{g}}h\; \wh{\nu}\p{\varphi}-h\p{d\varphi,\wh{\nu}}\; d\wh{A}_\varphi
    \end{equation*}
    for certain $h\in C^{2,\alpha}_\beta(M)\cap H^\ell_{\wh\varphi}(M)$ and $\varphi\in \mathcal{R}^\ell_{\wh{\varphi}}\p{M,\wh\varphi}$, is tantamount to the static pair $\p{\wh{g},\wh{\varphi}}$ being a local maximum of the entropy $\mu_{\mathrm{AH},\wh{g}}$. This, in turn, will be proven to be equivalent to dynamical stability under the Ricci-harmonic flow in Section \ref{stability_sec}.\\
    
    A subtle difference between the current section and the next, is that we require all metrics in this section to have a certain pointwise, in addition to being $H^\ell_{\wh{\varphi}}$-close to a static metric $\wh{g}$. To be specific, for the rest of Section \ref{sec_local_pos_mass} and unless stated otherwise, we assume $g\in \mathcal{R}^{\ell}_{\wh\varphi}\p{M,\wh{g}}\cap \mathcal{R}^{2,\alpha}_\beta\p{M,\wh{g}}$, where $\alpha\in [0,1)$ and $\beta\in \p{\frac{n}{2},n}$. In Section \ref{stability_sec}, we drop the requirement that $g\in \mathcal{R}^{2,\alpha}_{\beta}\p{M,\wh{g}}$. 
    
        \begin{lemma}\label{mass_independence}
            Consider a static background pair $\p{\wh{g},\wh{\varphi}}$ and fix a symmetric $2$-tensor $h \in H^1_{\wh{\varphi}}\p{S^2M}$. For any other pair $\p{g,\varphi}$ we have 
            \begin{equation*}
                m_{\wh{g}}\p{h,\wh{\varphi}} = m_{g}\p{h,\varphi},
            \end{equation*}     
            provided $\wh{\varphi} - \varphi \in H^1_{\wh{\varphi}}\p{M}$, $\abs{\wh{g}-g} + \abs{\scal_g + n\p{n-1}} = \mathcal{O}\p{e^{-\beta r}}$ for some $\beta > \frac{n}{2}$. In particular, it is meaningful to discuss the mass relative to non--static metric pairs $\p{g,\varphi}\in \mathcal{R}^{2,\alpha}_\beta\p{M,\wh{g}}\times \mathcal{R}^1_{\wh\varphi}\p{M,\wh\varphi}$.
        \end{lemma}

        \begin{proof}
            The result follows from analysing the error terms that result from changing from quantities defined with respect to $\wh{g}$ and/or $\wh{\varphi}$ to those defined with respect to $g$ and/or $\varphi$. For instance,
            \begin{equation*}
                \tr_gh = \tr_{\wh{g}}h + \tr_g h - \tr_{\wh{g}}h = \tr_{\wh{g}}h + \left(g^{-1} - \wh{g}^{-1}\right) \ast h. 
            \end{equation*}

            Using our decay assumption on $\wh{g} - g$, and that $h \in H^1_{\wh\varphi}\p{S^2M}$, implies the boundary integrals involving the error term $\left(g^{-1} - \wh{g}^{-1}\right) \ast h$ vanish as $R \rightarrow \infty$. The rest of the proof follows along similar lines, after using the formula
            \begin{equation}\label{grad_swap}
                    \nabla_g T = \nabla_{\wh{g}} T + \wh{g}^{-1} \ast \nabla_g h \ast T
            \end{equation}
            to deal with the divergence term.
        \end{proof}

    \subsection{A Conformal Positive Mass Result}
        An integral part of the local positive mass theorem, Theorem \ref{thm:local_PMT}, is the following \textit{conformal} positive mass proposition.
        
        \begin{proposition}\label{conformal_PMT}
            Let $g\in \mathcal{R}^{2,\alpha}_\beta\p{M,\wh{g}}$ be a metric of constant scalar curvature, $\scal_{g}=-n(n-1)$, and let $\varphi$ be an \textbf{exact potential} -- that is, an approximate potential satisfying the defining equation exactly;
            \begin{equation}\label{exact_equation}
                \Delta_{g}\varphi+\abs{d\varphi}_g^2=n.
            \end{equation}
    
            Furthermore, let $\wt{g}\in \mathcal{R}^{2,\alpha}_\beta\p{M,\wh{g}}$ be a metric conformal to $g$ such that $\scal_{\wt{g}}\geq -n(n-1)$. Then,
            \begin{equation*}
                m_{g}\p{\wt{g}-g,\varphi}\geq 0,    
            \end{equation*}
            and equality holds if and only if $\wt{g}=g$.
        \end{proposition}

        \begin{remark} \label{perturbation_theory}
            Note that the assumptions on $g$ and $\varphi$ imply that the auxiliary metric $e^{-2\varphi}\;d\tau^2+g$ has constant scalar curvature $-n(n+1)$. Note also that for a metric $g\in \mathcal{R}^{2,\alpha}_\beta\p{M, \wh{g}}$, we can always find an exact potential due to perturbation theory. That is, one linearises \eqref{exact_equation} at $\p{g,\varphi}$ and applies the implicit function theorem.
        \end{remark}

        \begin{proof}[Proof of Proposition \ref{conformal_PMT}]
            We parametrise the conformal factor such that $\wt{g}=\phi^{\frac{4}{n-2}}g$, where $\phi=1+\mathcal{O}(e^{-\beta r})$. The scalar curvatures of the two metrics are related by
            \begin{equation}\label{eq_conformal_law}
                \frac{4(n-1)}{n-2}\Delta_{g}\phi=\scal_{\wt{g}}\phi^{\frac{n+2}{n-2}}-\scal_{g}\phi.
            \end{equation}
    
            Since $\scal_{\wt{g}}\geq \scal_{g}=-n(n-1)$, a maximum principle argument shows that $\phi\geq 1$ (see the proof of \cite[Proposition 4.5]{DKM}). This is crucial later in the proof.\\
    
            Inserting $h=\wt{g}-g=(\phi^{\frac{4}{n-2}}-1)g$ in the definition of the weighted ADM mass, and substituting $V:=e^{-\varphi}$, yields
            \begin{equation*}
                m_{g}\p{h,\varphi}=(n-1)\lim_{R\to \infty}\int_{\partial B_R}\left((\phi^{\frac{4}{n-2}}-1)\wh{\nu}(V)-V\wh{\nu}(\phi^{\frac{4}{n-2}})\right)dA_{g}.
            \end{equation*}

            Now consider the equation            
            \begin{equation}\label{eq_potential_eq_2}
                \Delta_{g}W=-n\phi^pW
            \end{equation}
            where $p$ is a large exponent which will be chosen later. We can find a solution by considering the equivalent equation
            \begin{equation}\label{p_static_eqn}
                (\Delta_{g}+n\phi^p)(W-V)=n(1-\phi^p)V.
            \end{equation}

            The operator on the left hand side is strictly positive and therefore has trivial $L^2$-kernel. Because $\phi\to 1$ at infinity, $\Delta_{g}+n\phi^p: C^{2,\alpha}_\beta\p{M} \to C^{0,\alpha}_\beta \p{M}$ has the same indicial radius as $\Delta_{g}+n$, and is therefore an isomorphism for all $\beta \in(-1,n)$ by Proposition \ref{isomorph_prop}. Since the right hand side of \eqref{p_static_eqn} is contained in $C^{0,\alpha}_{\beta-1}\p{M}$, we can solve this equation and get  $W-V\in C^{2,\alpha}_{\beta-1}\p{M}$. In particular, $W$ is positive in a neighbourhood of infinity because $V$ is. Moreover, a maximum principle argument applied to the equation shows that $W\geq 0$ everywhere.\\
                
            The decay is strong enough, that the mass remains unchanged under the substitution of $V$ by $W$;
            \begin{equation*}
                m_{g}\p{h,\varphi}=(n-1)\lim_{R\to \infty}\int_{\partial B_R}\left((\phi^{\frac{4}{n-2}}-1)\wh{\nu}(W)-W\wh{\nu}(\phi^{\frac{4}{n-2}})\right)dA_{g},
            \end{equation*}
            
            By the decay assumptions of $\phi$ at infinity, we have
            \begin{align*}
                -(n-1)\lim_{R\to \infty}\int_{\partial B_R}W\wh{\nu}(\phi^{\frac{4}{n-2}})\; dA_{g}
                &=-4\frac{n-1}{n-2}\lim_{R\to \infty}\int_{\partial B_R}W\phi^{\frac{6-n}{n-2}}\wh{\nu}(\phi)\; dA_{g}\\
                &=-4\frac{n-1}{n-2}\lim_{R\to \infty}\int_{\partial B_R}W\wh{\nu}(\phi)\; dA_{g}.
            \end{align*}
            By Green's formula, this can be further modified to       
            \begin{align*}
                -4\frac{n-1}{n-2}\lim_{R\to \infty}\int_{\partial B_R}W\wh{\nu}(\phi)\; dA_{g}&=
                -4\frac{n-1}{n-2}\lim_{R\to \infty}\int_{\partial B_R}(\phi-1)\wh{\nu}(W)\; dA_{g}\\
                &+4\frac{n-1}{n-2}\lim_{R\to \infty}\int_{ B_R}(W\Delta_{g}\phi -(\phi-1)\Delta_{g}W)\; dV_{g}.
            \end{align*}
            Now we want to substitute this into the above formula for $m_{g}\p{h,\varphi}$. For this purpose, we note that Taylor expansion implies
            \begin{equation*}
                (n-1)(\phi^{\frac{4}{n-2}}-1)-4\frac{n-1}{n-2}(\phi-1)=\mathcal{O}\p{(\phi-1)^2}=\mathcal{O}\p{e^{-2\beta r}}.
            \end{equation*}
            Therefore        
            \begin{equation*}
                (n-1)\lim_{R\to \infty}\int_{\partial B_R}(\phi^{\frac{4}{n-2}}-1)\wh{\nu}(W)\; dA_{g}=4\frac{n-1}{n-2}\lim_{R\to \infty}\int_{\partial B_R}(\phi-1)\wh{\nu}(W)\; dA_{g}.
            \end{equation*}
            Putting everything together, and using \eqref{eq_conformal_law} and \eqref{eq_potential_eq_2}, we get
            \begin{align*}
                m_{g}\p{h,\varphi} &=  4\frac{n-1}{n-2}\int_{M}(W\Delta_{g}\phi -(\phi-1)\Delta_{g}W)dV_{g} \\
                &=\int_M W\left(\phi^{\frac{n+2}{n-2}}\scal_{\wt{g}}-\phi\scal_{g}+4n\frac{n-1}{n-2}(\phi-1)\phi^p\right)dV_{g}\\
                &=\int_M W\left(\phi^{\frac{n+2}{n-2}}(\scal_{\wt{g}}+n(n-1))+n(n-1)F(\phi)\right)dV_{g},
            \end{align*}
            where we used that $\scal_{g}=-n(n-1)$, and define $F$ by 
                $$F(x)=x-x^{\frac{n+2}{n-2}}+\frac{4}{n-2}(x-1)x^p.$$                
            One computes $F(1)=0$ and $F'(0)=1$ for any choice of $p > 1$. Furthermore, we have
            \begin{align*}
                F''(x)&= \frac{4}{n-2}\left[(p+1)px^{p-1}-p(p-1)x^{p-2}-\frac{n+2}{n-2}x^{\frac{6-n}{n-2}}\right]\\
                &\geq \frac{4}{n-2}\left[2px^{p-1}-\frac{n+2}{n-2}x^{\frac{6-n}{n-2}}\right]
            \end{align*}          
            for $x\geq 1$. Therefore, if $p$ is chosen so large that 
            \begin{equation*}
                p> \max\left\{\frac{4}{n-2},\frac{1}{2}\frac{n+2}{n-2}\right\},
            \end{equation*}
            then $F''(x)>0$ for all $x\geq 1$, so that $F(x)> 0$ for all $x\geq1$ by the above. Then, since  $\scal_{\wt{g}}\geq-n(n-1)$, $\phi\geq 1$ and $W\geq0$, we get that
            \begin{equation*}
                m_{g}\p{h,\varphi} = \int_M W\left(\phi^{\frac{n+2}{n-2}}(\scal_{\wt{g}}+n(n-1))+F(\phi)\right)dV_{g}\geq 0,
            \end{equation*}
            and equality can only hold if  $\scal_{\wt{g}}=-n(n-1)$ and $\phi=1$. This finishes the proof of the proposition.
        \end{proof}
    \subsection{Proof of Theorem \ref{mainthm_local_positive_mass}}
        We now prove some technical estimates we will need when proving our local positive mass theorem (Theorem \ref{mainthm_local_positive_mass}/\ref{thm:local_PMT}). 

        \begin{lemma} \label{lemma_needed_for_third_var}
            Let $\p{M,\wh{g},\wh{\varphi}}$ be a static triple. Then, for $\ell > \frac{n}{2} + 2$, there exists a neighbourhood $\mathcal{U}\subset \mathcal{R}^{\ell}_{\wh{\varphi}}\p{M,\wh{g}}\times \mathcal{R}^{\ell}_{\wh{\varphi}}\p{M,\wh{\varphi}}$ and a constant $C>0$, such that for any $\mathfrak{g}=(g,\varphi)\in \mathcal{U}$, the following estimates hold:
            \begin{align*}
                \norm{\left.\frac{d}{ds}\right|_{s=0} u_{\mathfrak{g}+s\mathfrak{h}}}_{C^{2,\alpha}}&\leq C\norm{\mathfrak{h}}_{C^{2,\alpha}},\\
                \norm{\left.\frac{d}{ds}\right|_{s=0} u_{\mathfrak{g}+s\mathfrak{h}}}_{H^1_{\mathfrak{g}}}\:\; &\leq C\norm{\mathfrak{h}}_{H^1_{\mathfrak{g}}},\\ \norm{\left.\frac{d}{ds}\right|_{s=0}f_{\mathfrak{g}+s\mathfrak{h}}}_{H^2_\mathfrak{g}}\;\;&\leq C\norm{\mathfrak{h}}_{H^2_\mathfrak{g}}\\
                \norm{\left.\frac{d^2}{ds^2}\right|_{s=0} u_{\mathfrak{g}+s\mathfrak{h}}}_{H^1_{\mathfrak{g}}}\, &\leq C\norm{\mathfrak{h}}_{C^{2,\alpha}}\norm{\mathfrak{h}}_{H^1_{\mathfrak{g}}},
            \end{align*}
            where $f_{\mathfrak{g}}=f_g$ is the unique minimiser of the entropy given in Theorem \ref{global_minimiser}, $u_{\mathfrak{g}}:=f_{\mathfrak{g}}-\varphi$, and $\mathfrak{h} = \p{h,\psi}\in H^\ell_{\mfg}$.
        \end{lemma}

        \begin{proof} 
            The static metric will stay mostly in the background for this proof. However, we use the fact that any pair $\p{g,\varphi}\in \mathcal{R}^{\ell}_{\wh{\varphi}}\p{M,\wh{g}}\times \mathcal{R}^{\ell}_{\wh{\varphi}}\p{M,\wh{\varphi}}$ has $u_{g,\varphi}\in H^{\ell}_\varphi\subset C^{2,\alpha}$ (see Propositions \ref{wgtd_sob_emb} and \ref{entropy_analytic_metric_dep}), and that the auxiliary Ricci curvature of $\p{g,\varphi}$ has finite $C^{0,\alpha}$-norm. With that out of the way, we may suppress the $g,\varphi$ subscripts to simplify notation. Recall that $u$ satisfies \eqref{modified_el_eqn}:
            \begin{equation*}
                \Lap u + \frac{1}{2}\abs{d u}^2 + \g{d \varphi}{ d u} + nu = -\p{\Lap \varphi + \abs{d \varphi}^2 - n} + \frac{1}{2}\p{\scal_g + n\left(n-1\right)},
            \end{equation*}
            or, in terms of the auxiliary metric $\mathfrak{g}=\p{g,\varphi}$, as 
            \begin{equation}\label{aux_u_EL}
                \Lap_{\mathfrak{g}} u + nu + \frac{1}{2}\abs{d u}^2 = \frac{1}{2}\p{\scal_{\mathfrak{g}} + n(n+1)}.
            \end{equation}    
            Taking the variation of this equation, and denoting the variation of $u$ by $u'$, yields
            \begin{equation}\label{EL_variation}\begin{split}
                \p{\Lap_{\mathfrak{g}}+n}u' ={}& \frac{1}{2}\left.\frac{d}{ds}\right\vert_{s=0} \scal_{\mathfrak{g} + s\mathfrak{h}} - \left(\left.\frac{d}{ds}\right\vert_{s=0} \Lap_{\mathfrak{g} + s\mathfrak{h}}\right) u  + \frac{1}{2}\mathfrak{h}\left(\nabla u,\nabla u\right)-\g{du'}{du}\\
                ={}& \frac{1}{2}\left(\Lap_{\mathfrak{g}}\tr_{\mathfrak{g}}\mathfrak{h} + \delta_{\mathfrak{g}} \delta_{\mathfrak{g}} \mathfrak{h} - \left<\Ric_{\mathfrak{g}}, \mathfrak{h}\right>\right) - \left<\mathfrak{h}, \nabla^2 u\right>\\
                &+ \left<\delta_{\mathfrak{g}} \mathfrak{h} + \frac{1}{2}d\tr_{\mathfrak{g}}\mathfrak{h}, d u\right> + \frac{1}{2}\mathfrak{h}\left(\nabla u,\nabla u\right)-\g{du'}{du}\\ 
                ={}&\nabla^2_{\mathfrak{g}}\ast \mathfrak{h}+\nabla_{\mathfrak{g}}\mathfrak{h}\ast \nabla u+\p{\Ric_{\mathfrak{g}}+\nabla^2u+\nabla u\ast \nabla u}\ast \mathfrak{h} +\nabla u'\ast \nabla u.\end{split}
            \end{equation}
            Since $\Lap_{\mathfrak{g}} + n: C^{2,\alpha} \rightarrow C^{0,\alpha}$ is an isomorphism by Proposition \ref{isomorph_prop}, we have
            \begin{align*}
                \norm{u'}_{C^{2,\alpha}} \leq{} &C\norm{\nabla^2_{\mathfrak{g}}\ast \mathfrak{h}+\nabla_{\mathfrak{g}}\mathfrak{h}\ast \nabla u+\p{\Ric_{\mathfrak{g}}+\nabla^2u+\nabla u\ast \nabla u}\ast \mathfrak{h} +\nabla u'\ast \nabla u}_{C^{0,\alpha}}\\ \leq{} &C'\p{\norm{u}_{C^{2,\alpha}}^2+1}\norm{\mathfrak{h}}_{C^{2,\alpha}}+C\norm{u}_{C^{1,\alpha}}\norm{u'}_{C^{1,\alpha}}.
            \end{align*}        
            Since the $C^{2,\alpha}$-norm of $u$ is controlled by the proximity to the background (which has $u\equiv0$), we may assume that we are working in a neighbourhood so small that $\norm{u}_{C^{2,\alpha}}\leq \frac{1}{2C}$. Moving the final term in the above equation to the left hand side, we obtain the desired:
            \begin{equation*}
                \frac{1}{2}\norm{u'}_{C^{2,\alpha}} \leq C'\p{\p{\frac{1}{2C''}}^2+1}\norm{\mathfrak{h}}_{C^{2,\alpha}}.
            \end{equation*}
            As for the corresponding estimate involving $H^1_{\mathfrak{g}}$-norms, we can integrate \eqref{EL_variation} against $u'$ to arrive at 
            \begin{equation}\label{eq:u'_H^1_bound}
                \norm{u'}_{H^1_\mfg}^2\leq \Big(u',\nabla^2_{\mathfrak{g}}\ast \mathfrak{h}+\nabla_{\mathfrak{g}}\mathfrak{h}\ast \nabla u+\p{\Ric_{\mathfrak{g}}+\nabla^2u+\nabla u\ast \nabla u}\ast \mathfrak{h} +\nabla u'\ast \nabla u\Big)_{L^2_{\mfg}}.     
            \end{equation}

            The next step is a little delicate, as we will have to get rid of the second order derivative of $\mfh$. In this case, we may recall the precise form of those terms, which will allow us to integrate by parts:
            \begin{align*}
                \Big(u',\nabla^2_\mfg\ast \mfh\Big)_{L^2_\mfg}&=\frac{1}{2}\Big(u',\delta_\mfg\p{\delta_\mfg\mfh+d\tr_\mfg\mfh}\Big)_{L^2_\mfg}\\&=\frac{1}{2}\Big(du',\delta_\mfg\mfh+d\tr_\mfg\mfh\Big)_{L^2_\mfg}-\frac{1}{2}\lim_{R\to\infty}\int_{\partial B_R}u'\g{\delta_{\mfg}\mfh+d\tr_\mfg\mfh}{\nu}\; dA_\mfg.
            \end{align*}

            By the first proven inequality of this lemma, and the bounded $H^1_\mfg$-norm of $\mfh$, the boundary integral vanishes in the limit. Returning to \eqref{eq:u'_H^1_bound}, we proceed by applying Cauchy--Schwarz and factoring out bounded H\"older norms of $u$ and $\Ric_\mfg$:
            \begin{align*}
                \norm{u'}_{H^1_\mfg}^2&\leq \norm{u'}_{H^1_\mfg}\Big(\norm{\nabla_{\mathfrak{g}}\ast \mathfrak{h}+\nabla_{\mathfrak{g}}\mathfrak{h}\ast \nabla u+\p{\Ric_{\mathfrak{g}}+\nabla^2u+\nabla u\ast \nabla u}\ast \mathfrak{h}+\nabla u'\ast \nabla u}_{L^2_\mfg} \Big)\\ &\leq \wt{C}\norm{u'}_{H^1_\mfg}\Big(1+\norm{u}_{C^{2,\alpha}}\Big)\norm{\mfh}_{H^1_\mfg}+\norm{u}_{C^{2,\alpha}}\norm{u'}_{H^1_\mfg}^2
            \end{align*}

            After possibly shrinking the size of the neighbourhood we are working in, we can ensure $\norm{u}_{C^{2,\alpha}} \leq \frac{1}{2}$, so that we may absorb the final term on the right into the left hand side. Dividing by $\norm{u'}_{H^1_\mfg}$, all that remains is
            $$\norm{u'}_{H^1_\mfg}\leq 3\wt{C}\norm{\mfh}_{H^1_\mfg},$$
            which proves the second inequality.\\

            For the third inequality, we rearrange \eqref{EL_variation} slightly, by absorbing $\p{\Delta_\mfg+n}\psi$ into the right hand side, to yield:
            \begin{align*}
                \p{\Lap_{\mfg}+n} v  & = \nabla^2_{\mfg}\ast \mfh+\nabla_{\mfg}\mfh\ast \nabla u+\p{\Ric_{\mfg}+\nabla^2u+\nabla u\ast \nabla u+n}\ast \mfh +\nabla u'\ast \nabla u.
            \end{align*}
            
            Then, using that $\Lap_{\mfg} + n : H^2_{\mfg} \rightarrow L^2_{\mfg}$ is an isomorphism by Proposition \ref{isomorph_prop}, we have
            \begin{align*}
                \norm{v}_{H^2_{\mfg}} &\leq \wh{C}\norm{\nabla^2_{\mfg}\ast \mfh+\nabla_{\mfg}\mfh\ast \nabla u+\p{\Ric_{\mfg}+\nabla^2u+\nabla u\ast \nabla u+n}\ast \mfh +\nabla u'\ast \nabla u}_{L^2_{\mfg}}\\ &\leq \wh{C}'\p{1+\norm{u}_{C^{2,\alpha}}}\norm{\mfh}_{H^2_{\mfg}}+\wh{C}\norm{u}_{C^{2,\alpha}}\norm{u'}_{H^1_{\mfg}}\\
                &\leq \frac{3\wh{C}'+\wh{C}}{2}\norm{\mfh}_{H^2_\mfg},
            \end{align*}
            where the final inequality follows from the previous estimate, and the assumptions made during its proof.\\
    
            For the final estimate of the lemma we take the variation of \eqref{EL_variation} to obtain
            \begin{align*}
                \left(\Lap_{\mfg} + n\right) u'' ={}&\frac{1}{2}\left.\frac{d^2}{ds^2}\right|_{s=0}\scal_{\mfg+s\mfh}-\p{\left.\frac{d^2}{ds^2}\right|_{s=0}\Delta_{\mfg+s\mfh}}u-\p{\left.\frac{d}{ds}\right|_{s=0}\Delta_{\mfg+s\mfh}}u'\\ &+2\mfh\p{\nabla u',\nabla u}-\frac{1}{2}\p{\mfh\times \mfh}\p{\nabla u,\nabla u}-\abs{du'}^2-\g{du''}{du}\\={}&
                \nabla^2_{\mfg} \mfh \ast \mfh +  \p{\nabla_{\mfg} \mfh+\mfh\ast \nabla u+\nabla u'}\ast \nabla_{\mfg} \mfh+\nabla u'\ast \nabla u'\\& +\p{R_{\mfg}\ast \mfh+\nabla^2u'+\nabla u'\ast \nabla u+\nabla u\ast \nabla u\ast \mfh}\ast \mfh+\nabla u''\ast \nabla u.
            \end{align*}
            where the schematic expression in the last two lines can be derived from \cite[Lemma A.3]{KK_stability}. It is written in a deliberate way, according to the method of estimation in the following. \\
        
            By integrating against $u''$ and using H\"older's inequality, one sees that
            \begin{align*}
                \norm{u''}^2_{H^1_{\mfg}}\leq{} & \overline{C} \big|\big|\nabla^2_{\mfg} \mfh \ast \mfh +  \p{\nabla_{\mfg} \mfh+\mfh\ast \nabla u+\nabla u'}\ast \nabla_{\mfg} \mfh+\nabla u'\ast \nabla u'\\& \qquad\quad +\p{R_{\mfg}\ast \mfh+\nabla^2u'+\nabla u'\ast \nabla u+\nabla u\ast \nabla u\ast \mfh}\ast \mfh\big|\big|_{L^2_{\mfg}}\norm{u''}_{L^2_\mfg}\\ \leq{} &\overline{C}'\bigg(\norm{\nabla^2_{\mfg}\mfh+\nabla_{\mfg} \mfh+\mfh\ast \nabla u+\nabla u'}_{C^{0,\alpha}}\norm{\mfh}_{H^1_{\mfg}}+\norm{\nabla u'}_{C^{0,\alpha}}\norm{u'}_{H^1_{\mfg}}\\ &\qquad \qquad +\norm{R_{\mfg}\ast \mfh+\nabla^2u'+\nabla u'\ast \nabla u+\nabla u\ast \nabla u\ast \mfh}_{C^{0,\alpha}}\norm{\mfh}_{H^1_{\mfg}}\bigg)\norm{u''}_{H^1_\mfg}.
            \end{align*}
            
            Combining this with the previous estimates and the uniform $C^{2,\alpha}$-bound on $u$, and then dividing through by $\norm{u''}_{H^1_\mfg}$, completes the proof.
        \end{proof}

        \begin{proposition}\label{third_var_estim}
            Let $\p{M,\wh{g},\wh{\varphi}}$ be a static triple and $\ell > \frac{n}{2} + 2$. Then there exists a neighbourhood $\mathcal{U}\subset \mathcal{R}^{\ell}_{\wh\varphi}\p{M,\wh{g}}\times \mathcal{R}^{\ell}_{\wh\varphi}\p{M,\wh\varphi}$ and a constant $C>0$, such that, for any $\p{g,\varphi}\in \mathcal{U}$, we have 
            \begin{align*}
                \abs{\left.\frac{d^3}{ds^3}\right|_{s=0}\mu_{\mathrm{AH}}\p{\mfg+s\mathfrak{h}}}\leq C\norm{\mathfrak{h}}_{C^{2,\alpha}}\norm{\mathfrak{h}}_{H^1_\mfg}^2,
            \end{align*}
            where $\mfg=\p{g,\varphi}$ and $\mathfrak{h}=(h,\psi)$ in the usual sense (of Section \ref{sec_second_var}).
        \end{proposition}

        \begin{proof}
            As we are interested in the auxiliary norms, we extend the first variation \eqref{aux_first_var}, to an integral over $\mathfrak{M}=\mathbb{S}^1\times M$:
            \begin{align*}
                -4n\pi D_{\mfg}\mu_{\mathrm{AH}}\p{\mathfrak{h}}=\int_\mathfrak{M}\g{\Ric_{\mfg}+n\mfg+\nabla^2_{\mfg}u_{\mfg}}{\mathfrak{h}}_{\mfg}e^{-u_{\mfg}}\; dV_{\mfg}.
            \end{align*}
            where $u_\mfg=f_{\mfg}-\varphi$ as in Lemma \ref{lemma_needed_for_third_var}. Writing $k_{\mfg}:=\Ric_{\mfg}+n\mfg+\nabla^2_{\mfg}u_{\mfg}$, we find that
            \begin{align*}
                -4n\pi\left.\frac{d^3}{ds^3}\right|_{s=0}&\mu_{\mathrm{AH}}\p{\mfg+s\mathfrak{h}}\\={}&\int_\mathfrak{M}\g{k_{\mfg}''}{\mathfrak{h}}_{\mfg}e^{-u_{\mfg}}\; dV_{\mfg}\\
                &+\int_\mathfrak{M}\g{k_{\mfg}'}{\tr_{\mfg}\mathfrak{h}-2\p{\mathfrak{h}\times \mathfrak{h}}-2u_{\mfg}'\mathfrak{h}}_{\mfg}e^{-u_{\mfg}}\; dV_{\mfg}\\
                &+\int_\mathfrak{M}\g{k_{\mfg}}{\p{\mathfrak{h}\times \mathfrak{h}\times \mathfrak{h}}-\p{\tr_{\mfg}\mathfrak{h}-2u_{\mfg}'}\p{\mathfrak{h}\times \mathfrak{h}}}_{\mfg}e^{-u_{\mfg}}\; dV_{\mfg}\\
                &+\int_\mathfrak{M}\g{k_\mfg}{\p{\frac{1}{2}\tr_{\mfg}\mathfrak{h}-u_{\mfg}'}^2\mathfrak{h}-\frac{1}{2}\abs{\mathfrak{h}}^2_{\mfg}\mathfrak{h}-u_{\mfg}''\mathfrak{h}}_{\mfg}e^{-u_{\mfg}}\; dV_{\mfg},
            \end{align*}
            where we use the convention:
                $$h\times k\times l=h\times \p{k\times l}+\p{h\times k}\times l+\p{h\times l}\times k.$$
            Using Hamilton star notation and referencing \cite[Appendix A]{KK_stability}, we may write 
            \begin{align*}
                k_{\mfg}''={}&\nabla^2_{\mfg}\mathfrak{h}\ast \mathfrak{h}+\nabla_{\mfg}\mathfrak{h}\ast \nabla_{\mfg}\mathfrak{h}+\nabla_{\mfg}\mathfrak{h}\ast \mathfrak{h}+R_{\mfg} \ast \mathfrak{h} \ast \mathfrak{h}\\
                &+\nabla_{\mfg}^2u_{\mfg}''+\nabla_{\mfg}u_{\mfg}'\ast \nabla_{\mfg}\mathfrak{h}\\
                k_{\mfg}'={}& \nabla^2_{\mfg}\ast\mathfrak{h}+R_{\mfg}\ast \mathfrak{h}+\nabla^2_{\mfg}u_\mfg'+\nabla_{\mfg}u_{\mfg}\ast \nabla_{\mfg} \mathfrak{h},
            \end{align*}          
            where $R_{\mfg}$ is a catch-all for curvature terms and constants. The terms $k_\mfg$, $R_\mfg$, $u_\mfg$ and $\nabla_\mfg u_\mfg$ are uniformly bounded due to proximity to a static background. By counting the number and derivatives of $\mathfrak{h}$ and $u_\mfg'$ in each term, it is a simple task to construct the claimed inequality via Lemma \ref{lemma_needed_for_third_var}, H\"older's inequality, and the definition of the neighbourhood $\mathcal{U}$. At least for every term save one. The only term that requires special attention is 
                $$\int_{\mathfrak{M}}\g{\nabla^2_{\mfg}u_{\mfg}''}{\mathfrak{h}}_{\mfg}e^{-u_{\mfg}}\; dV_{\mfg}.$$        
            Using Lemma \ref{lemma_needed_for_third_var}, we deduce that $u''_\mfg$ has bounded $H^1_{\mfg}$-norm, as long as $\mathfrak{h}$ has the same. As such, no boundary term arises from integration by parts:
                $$\int_{\mathfrak{M}}\g{\nabla^2_{\mfg}u_{\mfg}''}{\mathfrak{h}}_{\mfg}e^{-u_{\mfg}}\; dV_{\mfg}=\int_\mathfrak{M}\g{du_{\mfg}''}{\delta_\mfg\mathfrak{h}}_{\mfg}e^{-u_{\mfg}}+\mathfrak{h}\p{\nabla u_{\mfg}'',\nabla u_\mfg}e^{-u_{\mfg}}\; dV_{\mfg}.$$  
            The proposition now follows by a further use of H\"older's inequality and Lemma \ref{lemma_needed_for_third_var}.
        \end{proof}

    We can now prove Theorem \ref{mainthm_local_positive_mass}, which we restate here in an expanded version for the reader's convenience.

        \begin{theorem}\label{thm:local_PMT}
            Let $\p{M,\wh{g},\wh\varphi}$ be a static triple, $\ell>\frac{n}{2}+2$ and $\frac{n}{2}<\beta<n$. Then the following are equivalent:
            \begin{enumerate}
            \item $\p{\wh{g},\wh{\varphi}}$ is a local minimiser of the functional $(g,\varphi)\mapsto m_{\wh{g}}(g-\wh{g},\varphi)$ on
            \begin{equation*}
                \mathcal{D} := \big\{\p{g,\varphi} \; :\; \scal_{g} \geq -n\p{n-1}\big\}.
            \end{equation*}                  
            \item $\p{\wh{g},\wh{\varphi}}$ is a local minimiser of the functional $(g,\varphi)\mapsto m_{\wh{g}}(g-\wh{g},\varphi)$ on
            \begin{align*}
                \mathcal{C}:=\big\{(g,\varphi) \; : \; \scal_{g}=-n\p{n-1}, \;\Delta_g\varphi+\abs{d\varphi}^2_g=n\big\}.
            \end{align*}        
            \item $\p{\wh{g},\wh{\varphi}}$ is a local maximiser of $\mu_{\mathrm{AH}}$ on $\mathcal{C}$.           
            \item $\p{\wh{g},\wh{\varphi}}$ is a local maximiser of $\mu_{\mathrm{AH}}$.
            \end{enumerate}
            The statements relate to metric pairs $\p{g,\varphi}\in \p{\mathcal{R}^\ell_{\wh{\varphi}}\p{M,\wh{g}}\cap \mathcal{R}^{2,\alpha}_\beta\p{M,\wh{g}}}\times \mathcal{R}^\ell_{\wh{\varphi}}\p{M,\wh{\varphi}}$.
        \end{theorem}

        \begin{proof}
            Note that \textit{(1.)} implies \textit{(2.)}, simply as $\mathcal{C}$ is contained in $\mathcal{D}$.
            To show that \textit{(2.)} implies \textit{(1.)}, consider the pair $\p{g_{\mathcal{D}},\varphi_{\mathcal{D}}} \in \mathcal{D}$. By the resolution of the Yamabe problem (see Theorem \ref{Yamabe}), we may find a CSC metric $g_{\mathcal{C}}\in \left[g_{\mathcal{D}}\right]$, the conformal class of $g_{\mathcal{D}}$ in $\mathcal{R}^\ell_{\wh{\varphi}}\p{M,\wh{g}}\cap \mathcal{R}^{2,\alpha}_\beta\p{M,\wh{g}}$. Relative to this metric, we may even find an exact potential $\varphi_{\mathcal{C}}$ (see Remark \ref{perturbation_theory}) such that $\p{g_{\mathcal{C}},\varphi_{\mathcal{C}}}\in \mathcal{C}$. Then, since the mass is linear in its first argument, we can compute as follows:
            \begin{align*}
                m_{\wh{g}}\p{g_{\mathcal{D}} - \wh{g}, \varphi_{\mathcal{D}}} &= m_{\wh{g}}\p{g_{\mathcal{D}} - \wh{g}, \varphi_{\mathcal{C}}}\\
                &=m_{\wh{g}}\p{g_{\mathcal{D}} - g_{\mathcal{C}}, \varphi_{\mathcal{C}}} +  m_{\wh{g}}\p{g_{\mathcal{C}} - \wh{g}, \varphi_{\mathcal{C}}}\\
                &\geq m_{g_{\mathcal{C}}}\p{g_{\mathcal{D}} - g_{\mathcal{C}}, \varphi_{\mathcal{C}}}\\
                &\geq 0,
            \end{align*}
            where the first equality follows from Lemma \ref{mass_independence}, the first inequality from the assumption that \textit{(2.)} holds, while the final inequality is due to Proposition \ref{conformal_PMT}. In other words, \textit{(1.)} holds.\\
    
            For an arbitrary $\p{g,\varphi} \in \mathcal{C}$, Proposition \ref{klaus_konj} tells us $f_g = \varphi$. Plugging this and $\scal_{g} = -n\p{n-1}$ into the definition of the entropy tells us $\mathcal{B}_R\p{g, f_g,\varphi} = 0$ and $\mathcal{W}_R\p{g,f_g,\varphi} = 0$. Thus   
            \begin{equation*}
                \mu_{\mathrm{AH}}\p{g,\varphi} = -m_{\wh{g}}\p{g-\wh{g},\varphi}.
            \end{equation*}

            Hence, if \textit{(2.)} holds, we have $\mu_{\mathrm{AH}}\p{g,\varphi} \leq 0$, while if \textit{(3.)} holds we have $m_{\wh{g}}\p{g-\wh{g},\varphi} \geq 0$. That is, \textit{(2.)} and \textit{(3.)} are equivalent.\\
    
            We next note that \textit{(4.)} implies \textit{(3.)}, by again using that $\mathcal{C}$ is a subset of $\mathcal{D}$.\\

            Now assume \textit{(3.)} holds and consider an arbitrary pair $\mathfrak{g}=\p{g,\varphi}$ with $g\in \mathcal{R}^\ell_{\wh{\varphi}}\p{M,\wh{g}}\cap \mathcal{R}^{2,\alpha}_\beta\p{M,\wh{g}}$ and $\varphi\in   \mathcal{R}^\ell_{\wh{\varphi}}\p{M,\wh{\varphi}}$. As before, we conclude from Theorem \ref{Yamabe} that $g$ is conformally related to a metric ${g}_{\mathcal{C}}$ with exact potential $\varphi_{\mathcal{C}}$, such that $\p{g_{\mathcal{C}},\varphi_{\mathcal{C}}}\in \mathcal{C}$. We write $g=\p{1+\eta}{g}_{\mathcal{C}}$ with $\eta\in H^{\ell}_{\wh\varphi}\p{M}\cap C^{2,\alpha}_\beta\p{M}$. By setting $\mathfrak{h}:=\p{e^{-2\varphi}-e^{-2\varphi_{\mathcal{C}}}}d\tau^2+\eta g_{\mathcal{C}}$, we may write $\mathfrak{g}={\mathfrak{g}}_{\mathcal{C}}+\mathfrak{h}$.\\
            
            To understand the relationship between the entropy of $\mathfrak{g}$ and $\mathfrak{g}_\mathcal{C}$, we Taylor expand around the latter metric with respect to the affine curve connecting them: 
            \begin{align*}
                \mu_{\mathrm{AH}}\p{\mathfrak{g}} -  \mu_{\mathrm{AH}}\p{\mathfrak{g}_{\mathcal{C}}} &= \left.\frac{d}{ds}\right\vert_{s=0} \mu_{\mathrm{AH}}\p{{\mathfrak{g}}_{\mathcal{C}}+s\mathfrak{h}} + \frac{1}{2}\left.\frac{d^2}{ds^2}\right\vert_{s=0} \mu_{\mathrm{AH}}\p{{\mathfrak{g}}_{\mathcal{C}}+s\mathfrak{h}}\\
                &+ \frac{1}{2}\int^1_0 \left(1-s\right)^2 \frac{d^3}{ds^3} \mu_{\mathrm{AH}}\p{{\mathfrak{g}}_{\mathcal{C}}+s\mathfrak{h}}\; ds.
            \end{align*}
        
            Using Propositions \ref{stability_op_prop} and \ref{third_var_estim}, we find that the first derivative vanishes, the second has a negative upper bound, and the third has a uniform bound for all $s\in [0,1]$. In particular, 
            \begin{align*}
                \mu_{\mathrm{AH}}\p{\mathfrak{g} }-  \mu_{\mathrm{AH}}\p{\mathfrak{g}_{\mathcal{C}}} 
                &\leq \frac{1}{2}\p{C_1\norm{\mathfrak{h}}_{C^{2,\alpha}}-C_2}
                \norm{\mathfrak{h}}^2_{H^1_{\mathfrak{g}_{\mathcal{C}}}},
            \end{align*}
            where $C_1,C_2$ are the constants from the propositions referenced above. By restricting our attention to a sufficiently small $H^\ell_{\wh{\mfg}}$-neighbourhood of $\wh{\mathfrak{g}}$, we may assume
                $$\norm{\mathfrak{h}}_{C^{2,\alpha}}<\frac{C_2}{C_1},$$
            which means $\mu_{\mathrm{AH}}\p{\mathfrak{g}} \leq \mu_{\mathrm{AH}}\p{\mathfrak{g}_{\mathcal{C}}}$, with equality if and only if $\mathfrak{h}=0$. This concludes the proof.
        \end{proof}

\section{The Stability and Instability Results}\label{stability_sec}
    This section is devoted to proving Theorem \ref{mainthm_stability} and Theorem \ref{mainthm_instability}. First, though, we need to prove a weighted \L ojasiewicz--Simon inequality (Theorem \ref{mainthm_loj_ineq}) for pairs $\p{g,\varphi}$ which are sufficiently close to a static pair $\left(\wh{g}, \wh{\varphi}\right)$. We remind the reader that we shall now drop the pointwise decay condition, $g\in \mathcal{R}^{2,\alpha}_\beta\p{M,\wh{g}}$, of the preceding section.

    \subsection{The \L ojasiewicz--Simon Inequality}
        As in, say, \cite{KY_PE_stability}, the proof of the \L ojasiewicz--Simon inequality will follow the framework introduced in \cite{CM_loj_simon}. This in turn relies on proving various continuity and Fredholm properties of the entropy and its derivatives, when they act as maps between suitable Sobolev spaces. We emphasise here that the norms involved in such estimates will implicitly be taken with respect to a fixed static pair $\p{\wh{g},\wh{\varphi}}$.
        
        \begin{proposition}\label{loj_needed_estims}
            For $\ell > \frac{n}{2} + 2$ and any $\left(g_1, \varphi_1\right)$, $\left(g_2, \varphi_2\right)$, which are $\eps$-close to $\left(\wh{g}, \wh{\varphi}\right)$ in the $H^\ell_{\wh{\varphi}}$-sense for $0 < \eps \ll 1$, we have
            \begin{equation}\label{mu_grad_estims}
                \norm{\nabla \mu_{\mathrm{AH}}\p{g_2,\varphi_2} - \nabla \mu_{\mathrm{AH}}\p{g_1,\varphi_1}}_{L^2_{\wh{\varphi}}} \leq C\p{\eps}\norm{\p{g_2, \varphi_2} - \p{g_1,\varphi_1}}_{H^2_{\wh{\varphi}}}.
            \end{equation}

            Furthermore, for any $h \in H^\ell_{\wh\varphi}\p{S^2 M}$ and $\psi \in H^\ell_{\wh\varphi}\p{M}$, we have the following estimate for the difference of linearisations of the gradient of the entropy:
            \begin{equation}\label{mu_hess_estims}
                \norm{D_{\p{g_2,\varphi_2}} \nabla \mu_{\mathrm{AH}} \p{h, \psi} - D_{\p{g_1,\varphi_1}} \nabla \mu_{\mathrm{AH}} \p{h, \psi}}_{L^2_{\wh{\varphi}}} \leq C\p{\eps}\norm{\p{g_2, \varphi_2} - \p{g_1,\varphi_1}}_{H^2_{\wh{\varphi}}}.
            \end{equation}
        \end{proposition}

        \begin{proof} 
        Using the auxiliary expression for the $L^2_{\mathfrak{g}}$-gradient of the entropy, provided by Corollary \ref{first_var_crit}, we may complete the proof as a (almost) faithful rendition of the proof of \cite[Proposition 5.9]{KY_PE_stability}. With the usual notation, $\mathfrak{g}=(g,\varphi)$ and $u_\mathfrak{g}=f_{\mathfrak{g}}-\varphi$, we have the gradient
            $$\nabla\mu_{\mathrm{AH}}\p{\mathfrak{g}}=-\frac{1}{2 n}\p{\Ric_{\mfg}+n\mfg+\nabla_\mfg^2u_\mfg}e^{-u_{\mfg}}.$$
            
        The proof of \eqref{mu_grad_estims} boils down to showing 
         \begin{align}
                \norm{\Ric_{\mfg_2} - \Ric_{\mfg_1}}_{L^2_{\wh{\mfg}}} &\leq C\norm{\mfg_2 - \mfg_1}_{H^2_{\wh{\mfg}}} \label{aux_ric_diff}\\
                \norm{\nabla^2_{\mfg_2}u_{\mfg_2} - \nabla^2_{\mfg_1} u_{\mfg_1}}_{L^2_{\wh{\mfg}}} &\leq C\norm{\mfg_2 - \mfg_1}_{H^2_{\wh{\mfg}}} \label{u_hess_diff}\\
                \norm{e^{-u_{\mfg_2}}-e^{-u_{\mfg_1}}}_{L^2_{\wh{\mfg}}} &\leq C\norm{\mfg_2 - \mfg_1}_{H^2_{\wh{\mfg}}}, \label{expu-diff}
        \end{align}

        For the first two, \eqref{aux_ric_diff}-\eqref{u_hess_diff}, it is advantageous to use the schematic expressions 
        \begin{align}
            \Ric_{\mfg_2}-\Ric_{\mfg_1}&=\nabla_{\mfg_2}^2\ast \p{\mfg_2-\mfg_1}+\nabla_{\mfg_2}\p{\mfg_2-\mfg_1}\ast \nabla_{\mfg_2}\p{\mfg_2-\mfg_1} \label{ricci_schematic}\\
            \nabla^2_{\mfg_2}u_{\mfg_2}-\nabla^2_{\mfg_1}u_{\mfg_1} &=\nabla^2_{\mfg_2}\p{u_{\mfg_2}-u_{\mfg_1}}+\nabla_{\mfg_2}\p{\mfg_2-\mfg_1}\ast d\p{u_{\mfg_2}-u_{\mfg_1}}\label{hessian_schematic},
        \end{align}
        where the suppressed contractions are all with respect to $\mfg_1$. The above equations can be found on page 114 of \cite{CLN}. It should be noted that the proximity of $\mfg_1$ and $\mfg_2$ to the background $\mfg$, is such as to make the difference between volume forms and covariant derivatives negligible. To put it simply, we are interested in the order of $\nabla$, not the subscript. In light of \eqref{ricci_schematic}, we see that \eqref{aux_ric_diff} holds as $\mfg_1$ and $\mfg_2$ are sufficiently close to $\wh{\mfg}$ to ensure 
            $$\norm{\nabla_{\mfg_2}^2\ast \p{\mfg_2-\mfg_1}+\nabla_{\mfg_2}\p{\mfg_2-\mfg_1}\ast \nabla_{\mfg_2}\p{\mfg_2-\mfg_1}}_{L^2_{\wh{\mfg}}}\leq C(\varepsilon)\norm{\mfg_2-\mfg_1}_{H^2_{\wh{\mfg}}}. $$

        For the difference of Hessians, \eqref{u_hess_diff}, we need to show that $\nabla^2_{\mfg_2}\p{u_{\mfg_2}-u_{\mfg_1}}$ can somehow be estimated by $\mfg_2-\mfg_1$. To this end we write $\mfg_s=\mfg_1+(s-1)\p{\mfg_2-\mfg_1}$, which means 
            $$\nabla^2_{\mfg_2}\p{u_{\mfg_2}-u_{\mfg_1}}=\int_{1}^2\nabla_{\mfg_2}^2u'_{\mfg_s}\; ds,$$
        where $u'_{\mfg_s}=\left.\frac{d}{dt}\right|_{t=0}u_{\mfg_{s+t}}$. Combined with the third inequality of Lemma \ref{lemma_needed_for_third_var}, we have the desired estimate
            $$\norm{\nabla^2_{\mfg_2}\p{u_{\mfg_2}-u_{\mfg_1}}}_{L^2_{\wh{\mfg}}}\leq \sup_{s\in[1,2]}\norm{u'_{\mfg_s}}_{H^2_{\wh{\mfg}}}\leq C\norm{\mfg_2-\mfg_1}_{H^2_{\wh{\mfg}}}.$$
        One may even use the argument to conclude that 
        $$\norm{u_{\mfg_2}-u_{\mfg_1}}_{H^2_{\wh\mfg}}\leq C\norm{\mfg_2-\mfg_1}_{H^2_{\wh\mfg}},$$
        from which \eqref{expu-diff} follows. The inequality will be needed again in the second part of this proof. From Proposition \ref{entropy_analytic_metric_dep} and the Sobolev embedding theorem, we have $\norm{u}_{C^{2,\alpha}}\leq C\norm{\mfg_2-\mfg_1}_{H^{\ell}_{\wh{\mfg}}}\leq C\varepsilon$, with which we may complete the proof of \eqref{mu_grad_estims}:
        \begin{align*}
            \norm{\nabla^2_{\mfg_2}u_{\mfg_2}-\nabla^2_{\mfg_1}u_{\mfg_1}}_{L^2_{\wh{\mfg}}} &\leq     \norm{\nabla^2_{\mfg_2}\p{u_{\mfg_2}-u_{\mfg_1}}}_{L^2_{\wh{\mfg}}}+C\norm{u_{\mfg_2}-u_{\mfg_1}}_{C^1}\norm{\nabla_{\mfg_2}\p{\mfg_2-\mfg_1}}_{L^2_{\wh{\mfg}}}\\ &\leq C(\varepsilon)\norm{\mfg_2-\mfg_1}_{H^2_{\wh{\mfg}}},
        \end{align*}    
        where that the first inequality uses \eqref{hessian_schematic}.\\
        
        As for the proof of \eqref{mu_hess_estims}, we may linearise the gradient using standard variational formulae (see Appendix \ref{appendix}). This yields
        \begin{align*}
            D_{\mfg}\nabla\mu_{\mathrm{AH}}\p{\mfh}=-\frac{e^{-u_\mfg}}{2 n}\biggl(\frac{1}{2}\Delta_{L}\mfh+\frac{1}{2}\mfh\times \nabla^2_\mfg u_\mfg+\frac{1}{2}\nabla_{\nabla_\mfg u_\mfg}\mfh-\delta_\mfg^*\p{\beta_\mfg\mfh+\mfh\cdot du_\mfg}+n\mfh \;\;&\\ +\nabla^2_\mfg u'_\mfg-u'_\mfg\p{\Ric_\mfg+n\mfg+\nabla^2_\mfg u_\mfg} &\biggr), 
        \end{align*}
        where $\beta_\mfg=\delta_\mfg+\frac{1}{2}d\tr_\mfg$ is the Bianchi operator. In light of past estimates, the only term we need to prove new estimates for is $\nabla^2_{\mfg_2} u'_{\mfg_2} - \nabla^2_{\mfg_1} u'_{\mfg_1}$. To this end, first note that by applying \eqref{hessian_schematic} we have
        \begin{equation*}
            \nabla^2_{\mfg_2} u'_{\mfg_2} - \nabla^2_{\mfg_1} u'_{\mfg_1} = \nabla^2_{\mfg_2} \p{u'_{\mfg_2} - u'_{\mfg_1}} + \nabla_{\mfg_2} \p{\mfg_2 - \mfg_1}\ast d\p{u'_{\mfg_1} - u'_{\mfg_2}}.
        \end{equation*}    
        Therefore, we can estimate as follows:
        \begin{align*}
            \norm{\nabla^2_{\mfg_2} u'_{\mfg_2} - \nabla^2_{\mfg_1} u'_{\mfg_1}}_{L^2_{\wh\mfg}}
            &\leq C\left(\norm{\nabla^2_{\mfg_2} \p{u'_{\mfg_2} - u'_{\mfg_1}}}_{L^2_{\wh\mfg}} + \norm{u'_{\mfg_2} - u'_{\mfg_1}}_{C^1}\norm{\mfg_2 - \mfg_1}_{H^1_{\wh\mfg}}\right)\\
            &\leq C\left(\norm{\nabla^2_{\mfg_2} \p{u'_{\mfg_2} - u'_{\mfg_1}}}_{L^2_{\wh\mfg}} + \norm{\mfh}_{C^{2,\alpha}}\norm{\mfg_2 - \mfg_1}_{H^1_{\wh\mfg}}\right),
        \end{align*}
        where the final line is due to the first estimate of Proposition \ref{lemma_needed_for_third_var}. Next, by taking the difference of the linearised Euler--Lagrange equations \eqref{EL_variation} for $u_{\mfg_2}'$ and $u_{\mfg_1}'$, one finds that 
        \begin{align*}
            \p{\Delta_{\mfg_2}+n}\p{u_{\mfg_2}'-u_{\mfg_1}'}={}& \nabla^2_{\mfg_2}\mfk\ast \mfh+\nabla_{\mfg_2}\mfk\ast {\big[\nabla_{\mfg_2}\mfk \ast \mfh+\nabla_{\mfg_2}\mfh+\mfh\ast \nabla_{\mfg_2}\p{u_{\mfg_1}+u_{\mfg_2}}\big]}+\mfk\ast \nabla^2_{\mfg_2}u_{\mfg_1}'\\
            &+\mfk\ast {\big[\nabla^2_{\mfg_2}\mfh+\nabla_{\mfg_2}\mfh\ast \nabla_{\mfg_2}\p{u_{\mfg_1}+u_{\mfg_2}}+\mfh\ast \p{\Ric_{\mfg_2}+\nabla_{\mfg_2}u_{\mfg_1}\ast \nabla_{\mfg_2}u_{\mfg_2}}\big]}\\ & +\nabla_{\mfg_2}^2\p{u_{\mfg_1}-u_{\mfg_2}}\ast \mfh +\nabla_{\mfg_2}\p{u_{\mfg_1}-u_{\mfg_2}}\ast {\big[\nabla_{\mfg_2}\mfh+\mfh\ast \nabla_{\mfg_2}\p{u_{\mfg_1}+u_{\mfg_2}}\big]}
        \end{align*}
        where $\mfk:=\mfg_1-\mfg_2$ and we allow $\ast$ to denote contractions in both metrics. Thus, using that $\Lap_{\mfg_2} + n: H^2_{\wh\mfg} \rightarrow L^2_{\wh\mfg}$ is an isomorphism by a slight generalisation of Proposition \ref{isomorph_prop}, we have
         \begin{align*}
             \norm{u_{\mfg_2}'-u_{\mfg_1}'}_{H^2_{\wh\mfg}} &\leq C\norm{\p{\Lap_{\mfg_2} + n} \p{u'_{\mfg_1}-u'_{\mfg_2}}}_{L^2_{\wh\mfg}}\\
             &\leq C(\eps)\p{\norm{\mfh}_{C^{2,\alpha}}+\norm{u_{\mfg_1}'}_{C^{2,\alpha}}} \norm{\mfk}_{H^2_{\wh{\mfg}}} + C(\varepsilon)\norm{\mfh}_{C^{1,\alpha}}\norm{u_{\mfg_1}-u_{\mfg_2}}_{H^2_{\wh\mfg}}.
         \end{align*}
         where we, for the sake of the reader, have already estimated the factors $\nabla_{\mfg_2}u_{\mfg_1}$ and $\nabla_{\mfg_2}u_{\mfg_2}$ using $\norm{u_{\mfg_i}}_{C^{2,\alpha}}\leq C(\varepsilon)$ for $i=1,2$. In the first half of the proof, it was shown that
         $$\norm{u_{\mfg_1}-u_{\mfg_2}}_{H^2_{\wh\mfg}}\leq C\norm{\mfg_1-\mfg_2}_{H^2_{\wh\mfg}},$$
         which, combined with the first inequality of Lemma \ref{lemma_needed_for_third_var}, proves that 
         $$\norm{u_{\mfg_2}'-u_{\mfg_1}'}_{H^2_{\wh\mfg}}\leq C(\varepsilon)\norm{\mfh}_{C^{2,\alpha}}\norm{\mfg_1-\mfg_2}_{H^2_{\wh\mfg}}.$$        
         With this in hand, we may conclude that 
         $$\norm{D_{\mfg_1}\nabla\mu_{\mathrm{AH}}\p{\mfh}-D_{\mfg_2}\nabla\mu_{\mathrm{AH}}\p{\mfh}}_{L^2_{\wh\mfg}}\leq C\p{\eps}\norm{\mfh}_{C^{2,\alpha}}\norm{\mfg_1-\mfg_2}_{H^2_{\wh\mfg}},$$
         which, due to the Sobolev embedding theorem, is simply an auxiliary version of \eqref{mu_hess_estims}.
        \end{proof}

        We will need the following slice theorem before we are ready to prove the \L ojasiewicz--Simon inequality.
        
        \begin{lemma}\label{slice_result}
            Let $\p{M,\wh{g},\wh\varphi}$ be a static triple, let $\ell > \frac{n}{2} + 2$, and define the space of transverse auxiliary metrics
            \begin{equation*}
                \mathcal{S}^{\ell}_{\wh{g},\wh{\varphi}} := \left\{\p{g,\varphi} \in \mathcal{R}^\ell_{\wh{\varphi}}\p{M,\wh{g}} \times \mathcal{R}^\ell_{\wh{\varphi}}\p{M, \wh{\varphi}}: e^{-2\varphi}d\tau^2 + g = \mathfrak{g} \in \ker\delta_{\wh{\mathfrak{g}}}\right\}.
            \end{equation*}

            Then there is a neighbourhood $\mathcal{U} \subset \mathcal{R}^\ell_{\wh{\varphi}}\p{M,\wh{g}} \times \mathcal{R}^\ell_{\wh{\varphi}}\p{M, \wh{\varphi}}$ of $\p{\wh{g},\wh{\varphi}}$, such that any pair $\p{g,\varphi} \in \mathcal{U}$ can be written uniquely as $\p{g,\varphi} = \p{\eta^\ast \wt{g}, \eta^\ast \wt{\varphi}}$ for some $\p{\wt{g},\wt{\varphi}} \in \mathcal{U} \cap \mathcal{S}^{\ell}_{\wh{g},\wh{\varphi}}$, and a diffeomorphism $\eta \in \mathrm{Diff}^{\ell + 1}\p{M}$ that is $H^{\ell + 1}_{\wh{\varphi}}$-close to the identity.
        \end{lemma}

        \begin{proof}
            As $\p{\mathfrak{M}, \wh{\mathfrak{g}}}$ is AH, one may show that the operator $\delta_{\wh\mfg}\delta_{\wh\mfg}^*$ has indicial radius $R=\frac{n}{2}$ (when $\dim\mathfrak{M}=n+1$). One way to see this is to decompose it as
            $$2\delta_{\wh{\mfg}}\delta^*_{\wh{\mfg}}X=\p{\nabla^*\nabla}_{\wh\mfg} X-\Ric_{\wh\mfg}\cdot X+\nabla_{\wh\mfg}\delta_{\wh\mfg} X,$$
            and use that $-\Ric_{\wh\mfg}\cdot X\sim nX$ (an identity in the case of the PE metric $\wh\mfg$, but asymptotically true for a general AH metric). One may then show that $R$ is determined solely as the indicial radius of $\p{\nabla^*\nabla}_{\wh\mfg}+n$, which is indeed $\frac{n}{2}$. This, and ellipticity, implies that $\mathrm{Im}\delta^*_{\wh{\mfg}}|_{H^{\ell+1}_{\beta}\p{T\mathfrak{M}}}$ is closed for $\beta$ in the Fredholm range of $\delta_{\wh\mfg}\delta_{\wh\mfg}^*$, i.e. $\abs{\beta}<\frac{n}{2}$. Specifically, for $\beta=0$, we have the decomposition 
            \begin{equation*}
                H^\ell_{\wh{\mathfrak{g}}}\p{S^2\mathfrak{M}} = \mathrm{ker} \left. \delta_{\wh{\mathfrak{g}}}\right|_{H^{\ell}_{\wh{\mathfrak{g}}}\p{S^2\mathfrak{M}}} \oplus \mathrm{Im} \left. \delta^\ast_{\wh{\mathfrak{g}}}\right|_{H^{\ell+1}_{\wh{\mathfrak{g}}}\p{T\mathfrak{M}}}.
            \end{equation*}

            By restricting the decomposition to the subspace $T_{\p{\wh{g}, \wh{\varphi}}}\p{\mathcal{R}^\ell_{\wh{\varphi}}\p{M,\wh{g}} \times \mathcal{R}^\ell_{\wh{\varphi}}\p{M, \wh{\varphi}}}$, we obtain
            \begin{align*}
                &T_{\p{\wh{g}, \wh{\varphi}}}\p{\mathcal{R}^\ell_{\wh{\varphi}}\p{M,\wh{g}} \times \mathcal{R}^\ell_{\wh{\varphi}}\p{M, \wh{\varphi}}}\\
                &= \left\{\p{h,\psi} \in H^\ell_{\wh{\varphi}}\p{S^2M} \times H^\ell_{\wh{\varphi}}\p{M} : \delta_{\wh{\mathfrak{g}}}(h,\psi)= 0\right\}\oplus \left\{\p{\delta^\ast_{\wh{g}}X, \frac{1}{2}X\p{\wh{\varphi}}} : X \in H^{\ell+1}_{\wh{\varphi}}\p{TM}\right\},
            \end{align*}
            where $\delta_{\wh{\mathfrak{g}}} (h,\psi) = \delta_{\wh{g}}h + h \cdot d\wh{\varphi} + 2\psi \;d\wh{\varphi}$ by \eqref{aux_divergence}. Note that the first factor on the right hand side above is exactly $T_{\p{\wh{g},\wh{\varphi}}}\mathcal{S}^{\ell}_{\wh{g},\wh{\varphi}}$. Now define the map
            \begin{equation*}
                \begin{cases}
                    \Phi : \mathcal{S}^{\ell}_{\wh{g}, \wh{\varphi}} \times \mathrm{Diff}^{\ell + 1}\p{M} &\longrightarrow \;\;\mathcal{R}^\ell_{\wh{\varphi}}\p{M,\wh{g}} \times \mathcal{R}^\ell_{\wh{\varphi}}\p{M, \wh{\varphi}}\\\hfill
                    \p{g,\varphi,\eta} &\longmapsto\;\; \p{\eta^\ast g, \eta^\ast \varphi}.
                \end{cases}
            \end{equation*}

            The differential of $\Phi$ at $\p{\wh{g}, \wh{\varphi}, \mathrm{Id}_M}$ is
            \begin{equation*}
                \begin{cases}
                    D_{\wh{g}, \wh{\varphi},\mathrm{Id}_M}\Phi : T_{\p{\wh{g}, \wh{\varphi}}} \mathcal{S}^{\ell}_{\wh{g}, \wh{\varphi}} \times H^{\ell + 1}_{\wh{\varphi}}\p{TM} &\longrightarrow \;\; T_{\p{\wh{g}, \wh{\varphi}}} \mathcal{S}^{\ell}_{\wh{g}, \wh{\varphi}} \oplus \mathrm{Im}\:   \delta^\ast_{\wh{\mathfrak{g}}}|_{H^{\ell+1}_{\wh{\varphi}}\p{TM}}\\
                    \hfill\p{h,\psi,X} &\longmapsto \;\; \p{h + \delta^\ast_{\wh{g}}X, \psi + \frac{1}{2}X\p{\wh{\varphi}}},
                \end{cases}
            \end{equation*}
            which is clearly bijective. The inverse function theorem then informs us that there exists a neighbourhood $\mathcal{U} \subset \mathcal{R}^\ell_{\wh{\varphi}}\p{M,\wh{g}} \times \mathcal{R}^\ell_{\wh{\varphi}}\p{M, \wh{\varphi}}$, such that every $\p{g,\varphi} \in \mathcal{U}$ can be written uniquely as $\Phi\p{\wt{g},\wt{\varphi},\eta} = \p{g,\varphi}$, for some $\p{\wt{g},\wt{\varphi}} \in \mathcal{U} \cap \mathcal{S}^{\ell}_{\wh{g},\wh{\varphi}}$ and $\eta \in \mathrm{Diff}^{\ell + 1}\p{M}$ that is $H^{\ell + 1}_{\wh{\varphi}}$-close to $\mathrm{Id}_M$.
        \end{proof}

        For the next proof, we shall need the following generalisation of the space in Lemma \ref{slice_result}. For $1\leq k\leq \ell$, set
                $$\mathcal{S}^k_{\wh{g},\wh{\varphi}}=\mathcal{S}^k_{\wh{\mfg}}:= \p{\mathcal{R}^k_{\wh{\varphi}}\p{M,\wh{g}} \times \mathcal{R}^k_{\wh{\varphi}}\p{M, \wh{\varphi}}}\cap \ker \delta_{\wh\mfg}.$$
                
        This definition is naturally extended to $k =0$, by setting 
            $$\mathcal{S}^0_{\wh{g},\wh{\varphi}}=\mathcal{S}_{\wh\mfg}^0:=\p{\mathcal{R}^k_{\wh{\varphi}}\p{M,\wh{g}} \times \mathcal{R}^k_{\wh{\varphi}}\p{M, \wh{\varphi}}}\cap \p{\mathrm{Im}\; \delta^*_{\wh{\mfg}}|_{C_c^\infty(TM)}}^\perp.$$

        The precise form of $\mathcal{S}^0_{\wh{g},\wh{\varphi}}$ is needed because of regularity considerations. We refer the reader to \cite[Lemma 5.7]{KY_PE_stability} and the associated discussion for more details. We can now prove our \L ojasiewicz--Simon inequality, which we restate here in an expanded form.
        
        \begin{theorem}\label{loj_ineq_thm}
            Given a static pair $\p{\wh{g},\wh{\varphi}}$, there is a neighbourhood $\mathcal{U} \subset \mathcal{R}^\ell_{\wh{\varphi}}\p{M,\wh{g}} \times \mathcal{R}^\ell_{\wh{\varphi}}\p{M,\wh{\varphi}}$ of $\p{\wh{g},\wh{\varphi}}$, in which the entropy $\mu_{\mathrm{AH}}=\mu_{\mathrm{AH},\wh{g}}$ satisfies the following: there exists a positive constant $C > 0$ and an exponent $\theta \in \left(0,1\right]$, such that
            \begin{equation}\label{wgtd_loj_ineq}
                \abs{\mu_{\mathrm{AH}}\p{g, \varphi} - \mu_{\mathrm{AH}}\p{\wh{g}, \wh{\varphi}}}^{2-\theta} \leq C\norm{\nabla\mu_{\mathrm{AH}}\p{g, \varphi}}^2_{L^2_{\wh{\varphi}}}.
            \end{equation}

            This weighted inequality is equivalent to an unweighted inequality with respect to the auxiliary metric $\\wh{\mathfrak{g}} = e^{-2\wh{\varphi}}d\tau^2 + \wh{g}$:
            \begin{equation}\label{aux_loj_ineq}
                \abs{\mu_{\mathrm{AH}}\p{\mathfrak{g}} - \mu_{\mathrm{AH}}\p{\wh{\mathfrak{g}}}}^{2-\theta} \leq C\norm{\nabla \mu_{\mathrm{AH}}\p{\mathfrak{g}}}^2_{L^2_{\wh{\mathfrak{g}}}}.
            \end{equation}
        \end{theorem}

        \begin{proof}
            The result will follow if we can ensure the assumptions and Condition \textit{(1.)} - \textit{(4.)} of \cite[Theorem 5.1]{KY_PE_stability} are met. The setup is slightly different from that of \cite[Theorem 5.1]{KY_PE_stability}, as we shall use $\wh{\mfg}$ in place of $0$, and 
            $$\mathcal{R}^k_{\wh\mfg}(\mathfrak{M},\wh{\mfg}):=\left\{\mfg \; :\; \mfg-\wh\mfg\in H^k_{\wh{\mfg}}(S^2\mathfrak{M})\right\}$$
            in place of $H^k$. However, this causes no issues. By Lemma \ref{slice_result} and the diffeomorphism invariance of the entropy, it is enough to prove \eqref{wgtd_loj_ineq} on a neighbourhood $\mathcal{U}\subset \mathcal{R}^\ell_{\wh\mfg}\p{\mathfrak{M},\wh{\mfg}}\cap \mathcal{S}^0_{\wh{\mfg}}=\mathcal{S}^\ell_{\wh{\mfg}}$ of $\wh{\mfg}$. Suppose $\mathcal{U}$ is chosen small enough that $\mu_{\mathrm{AH}}$ is well-defined and analytic, and the estimates of Proposition \ref{loj_needed_estims} and Lemma \ref{lemma_needed_for_third_var} hold. \\

            \noindent\textbf{Condition 1:} \textit{The $L^2$-gradient of $\mu_{\mathrm{AH}}$, $\nabla \mu_{\mathrm{AH}}:\mathcal{U}\to H^{\ell-2}_{\wh\mfg}$, satisfies $\nabla\mu_{\mathrm{AH}}(\wh{\mfg})=0$ and $\norm{\nabla\mu_{\mathrm{AH}}(\mfg_{1})-\nabla\mu_{\mathrm{AH}}(\mfg_2)}_{L^2_{\wh{\mfg}}}\leq C\norm{\mfg_1-\mfg_2}_{H^{2}_{\wh\mfg}}$, for $\mfg_1,\mfg_2\in \mathcal{U}$}.\\

            The $L^2$-gradient of $\mu_{\mathrm{AH}}$ at $\mfg\in \mathcal{U}$ is given by 
            $$\nabla \mu_{\mathrm{AH}}(\mfg)=-\frac{1}{2 n}\p{\Ric_\mfg+n\mfg+\nabla^2_\mfg u_\mfg}e^{-u_\mfg}.$$
            
            Since the Poincar\'e-Einstein metric $\wh\mfg$ is a critical point for the entropy, we have $\nabla \mu_{\mathrm{AH}}(\wh\mfg)=0$. The required $L^2_{\wh{\mfg}}$ estimate follows from Proposition \ref{loj_needed_estims}.\\

            \noindent\textbf{Condition 2:} \textit{For every $\mfh\in H^\ell_{\wh{\mfg}}$ the map $\mfg\mapsto D_\mfg\nabla\mu_{\mathrm{AH}}(\mfh):\mathcal{U}\to L^2_{\wh\mfg}$ is continuous.}\\

            This follows immediately from the second (Lipschitz) estimate of Proposition \ref{loj_needed_estims}.\\

            \noindent\textbf{Condition 3:} \textit{The linearisation $D_{\wh{\mfg}}\nabla\mu_{\mathrm{AH}}$ is formally self-adjoint, and its restriction as a map $T_{\wh\mfg}\mathcal{S}^2_{\wh\mfg}\to T_{\wh\mfg}\mathcal{S}^0_{\wh\mfg}$ is continuous and Fredholm.}\\

            A simple computation shows that the linearisation at $\wh\mfg$ agrees with the Hessian of Section \ref{sec_second_var}, i.e. \eqref{aux_mu_Hessian}:
            $$D_{\wh{\mfg}}\nabla\mu_{\mathrm{AH}}\p{\mfh}=-\frac{1}{2 n}\p{\frac{1}{2}\Delta_{E,\wh{\mfg}}\mfh-\delta^*_{\wh\mfg}\delta_{\wh\mfg}\mfh+\nabla^2_{\wh\mfg}\p{u'_{\wh\mfg}-\frac{1}{2}\tr_{\wh\mfg}\mfh}}. $$
            
            As the Hessian of a Riemannian functional, it is automatically formally self-adjoint. The restriction to $\mfh\in T_{\wh{\mfg}}\mathcal{S}^2_{\wh\mfg}\subset \ker \delta_{\wh\mfg}$ has $u'_{\wh\mfg}-\frac{1}{2}\tr_{\wh\mfg}\mfh=0$ (see \eqref{linearised_minimiser}), which leaves only 
            $$D_{\wh{\mfg}}\nabla\mu_{\mathrm{AH}}|_{T_{\wh{\mfg}}\mathcal{S}^2_{\wh\mfg}}=-\frac{1}{4 n}\Delta_{E,\wh{\mfg}}.$$
            
            Since $\wh{\mathfrak{g}}$ is a Poincar\'e--Einstein metric, \cite[Remark 2.11]{KY_PE_stability} tells us the operator is indeed Fredholm. Continuity follows directly from ellipticity, as all elliptic second order differential operators $P$ satisfy $\norm{P\mfh}_{H^{k-2}_{\wh\mfg}}\leq C\norm{\mfh}_{H^{k}_{\wh{\mfg}}}$ for all $k \geq 2$ and some constant $C > 0$.\\

            \noindent\textbf{Condition 4:} \textit{The linearisation $D_{\wh{\mfg}}\nabla\mu_{\mathrm{AH}}:T_{\wh\mfg}\mathcal{S}^\ell_{\wh\mfg}\to T_{\wh\mfg}\mathcal{S}^{\ell-2}_{\wh\mfg}$ is continuous and Fredholm.}\\

            Fredholmness follows from \cite[Remark 2.11]{KY_PE_stability} (see also \cite[Proposition D]{Lee}). Continuity again follows from ellipticity.\\
            
            Conditions 1 and 2 show that $\mfg\mapsto \nabla\mu_{\mathrm{AH}}(\mfg)$ is a $C^1$ map in a neighbourhood of $\wh\mfg$, while Conditions 3 and 4 imply that its derivative at $\wh\mfg$ is an isomorphism of the tangent space. Teed up for the use of the inverse function theorem, we can apply \cite[Theorem 5.1]{KY_PE_stability} to complete the proof of the \L ojasiewicz--Simon inequalities \eqref{wgtd_loj_ineq}-\eqref{aux_loj_ineq}.
        \end{proof}

    \subsection{Proofs of Theorems \ref{mainthm_stability} and \ref{mainthm_instability}}
        We are now finally able to prove our main stability and instability results. However, before presenting the proofs, we need to recall some short time existence and regularity results due to Bamler (see \cite{bamler_symm_spaces}) that will be useful. In particular, our proofs will be very similar to those in \cite{KY_PE_stability}, since instead of pairs $\p{g,\varphi}$ we can consider the auxiliary metric $\mathfrak{g} = e^{-2\varphi}d\tau^2 + g$. 
        
        As every static pair $\p{\wh{g}, \wh{\varphi}}$ on $M$ yields a PE metric $\wh{\mathfrak{g}}$ on $\mathfrak{M}$, the problem reduces to proving stability of PE metrics (as done in \cite{KY_PE_stability}) with a warped product structure. Importantly, metrics which are warped products remain so when evolving by Ricci flow.\\

        Throughout the proofs of the results in this section, we will mainly consider the auxiliary Ricci-DeTurck flow:
        \begin{equation}\label{rdtf}
            \begin{cases}
                \partial_t \mathfrak{g}(t) = -2\p{\Ric_{\mathfrak{g}(t)} + n \mathfrak{g}(t) + \frac{1}{2}\mathscr{L}_{X_{\wh\mfg}\p{\mathfrak{g}(t)}}\p{\mathfrak{g}(t)}}\\
                \mathfrak{g}\p{0} = \wh{\mathfrak{g}}. 
            \end{cases}
        \end{equation}
        
        Here $\wh\mfg=\p{\wh{g},\wh\varphi}$ is a static background metric, and $X_{\mathfrak{g}}\p{\mathfrak{h}}$ is the Bianchi vector field,
        \begin{equation*}
            X_{\mathfrak{g}}\p{\mathfrak{h}} :=\p{\beta_{\mfg}\mfh}^\sharp= \p{\delta_{\mathfrak{g}}\mfh}^\sharp + \frac{1}{2}\nabla_{\mathfrak{g}}\tr_{\mathfrak{g}}\mathfrak{h}.
        \end{equation*}

        We refer the reader to \cite[Section 5]{RJJ} for more on the Ricci-DeTurck flow, and the connection between $X_{\wh{\mfg}}$ and the usual DeTurck vector field. \\

        Note that, neglecting the term $\mathscr{L}_{X_{\wh{\mathfrak{g}}}\p{\mfg(t)}}\p{\mfg(t)}$, the system \eqref{rdtf} is equivalent to 
        \begin{equation*}\label{warped_product_rf}
            \begin{cases}
                \partial_t g(t) = -2\p{\Ric_{g(t)} + n g(t) - d\varphi(t) \otimes d\varphi(t) + \nabla^2 \varphi(t)}\\
                \partial_t \varphi(t) = -\p{\Lap_{g(t)} \varphi(t) + \abs{\nabla \varphi(t)}^2 - n}\\
                \p{g\p{0}, \varphi\p{0}} = \p{\wh{g}, \wh{\varphi}}.
            \end{cases}
        \end{equation*}

        This follows from the formulae of Section \ref{auxiliary_Ric}.\\

        Short time existence, as well as $L^\infty$-bounds on the perturbation $\mathfrak{h}(t) = \mathfrak{g}(t) - \wh{\mathfrak{g}}$, hold due to the following result (\cite[Proposition 2.4]{bamler_symm_spaces}):
        
        \begin{proposition}\label{rdtf_short_time}
            Let $\p{\mathfrak{M},\wh{\mathfrak{g}}}$ be a complete Riemannian manifold with its curvature tensor globally bounded in the $C^{0,\alpha}$-sense and let $\mathfrak{g}_0 = \wh{\mathfrak{g}} + \mathfrak{h}_0$  be a smooth metric on $\mathfrak{M}$. Then there are constants $C,\eps,\tau_{\mathrm{SE}} > 0$, depending only on $\mathfrak{M}$ and $\wh{\mathfrak{g}}$, such that, if $\norm{\mathfrak{h}_0}_{L^\infty} < \eps$, there is a unique $L^\infty$-bounded smooth solution $\left\{\mathfrak{h}(t)\right\}_{t \in \left[0,2\tau_{\mathrm{SE}}\right]}$ to \eqref{rdtf}. Moreover, we have
            \begin{equation*}
                \sup_{t\in\left[0,2\tau_{\mathrm{SE}}\right]}\norm{\mathfrak{h}(t)}_{L^\infty} \leq C\norm{\mathfrak{h}_0}_{L^\infty}.
            \end{equation*}
        \end{proposition}

        We are also able to get $H^\ell_{\wh{\mathfrak{g}}}$-control, as well as $C^\ell_{\wh{\mathfrak{g}}}$-control, of $\mathfrak{h}(t)$ for $\ell \in \mathbb{N}$ as large as we like. This is due to the following result (\cite[Lemma 6.2]{bamler_symm_spaces}):

        \begin{lemma}\label{rdtf_hl_bounds}
            Let $\p{\mathfrak{M},\wh{\mathfrak{g}}}$ be a PE manifold of class $C^{2,\alpha}$, let $\mathfrak{g}_0 = \wh{\mathfrak{g}} + \mathfrak{h}_0$ be a smooth metric on $\mathfrak{M}$ and denote the Ricci-DeTurck flow starting from $\mathfrak{h}_0$ by $\left\{\mathfrak{h}_t\right\}_{t \in \left[0,2\tau_{\mathrm{SE}}\right]}$, where $2\tau_{\mathrm{SE}}$ is the existence time from Proposition \ref{rdtf_short_time}. For every $\ell, b \geq 0$ there are constants $A_\ell,\eps_\ell > 0$ such that, if $\norm{\mathfrak{h}_0}_{L^\infty} < \eps_\ell$ and $\norm{\mathfrak{h}_0}_{L^2_{\wh\mfg}} \leq b$, then
            $$\norm{\mathfrak{h}(t)}_{L^2_{\wh\mfg}}\leq A_0b,\qquad \qquad \text{for all }\quad t\in [0,2\tau_{\mathrm{SE}}]$$
            and, for $\ell\geq 1$,
            \begin{equation*}
                \norm{\mathfrak{h}(t)}_{H^\ell_{\wh\mfg}} \leq A_\ell b,\qquad \qquad \text{for all }\quad t\in [\tau_{\mathrm{SE}},2\tau_{\mathrm{SE}}].
            \end{equation*}

            Additionally, for all $\ell \geq 0$, we have the pointwise estimate
            \begin{equation*}
                \abs{\nabla^\ell_{\wh{\mfg}} \mathfrak{h}(t)}_{\wh\mfg} \leq A_\ell b, \qquad \qquad \text{for all }\quad t \in \left[\tau_{\mathrm{SE}},2\tau_{\mathrm{SE}}\right].
            \end{equation*}
        \end{lemma}

        We are now ready to prove our stability result (Theorem \ref{mainthm_stability}), which we restate here in an expanded form for the reader's convenience. In the expanded statement, we shall use the notion of an \textbf{auxiliary neighbourhood} to mean a subset of the set of pairs $(g,\varphi)$, as opposed to a true neighbourhood in the space of all $(n-1)$-dimensional metrics.\\

        \begin{theorem}\label{static_stability}
            Let $\p{M,\wh{g}, \wh{\varphi}}$ be a static triple of class $C^{\ell,\alpha}$ ($\ell>\frac{n}{2}+2$) and dimension $n \geq 3$, and set $\wh{\mathfrak{g}} := e^{-2\wh{\varphi}} d\tau^2 + \wh{g}$.  If $\p{\wh{g}, \wh{\varphi}}$ is a local maximiser of the entropy $\mu_{\mathrm{AH}}=\mu_{\mathrm{AH},\wh{g}}$, then for any auxiliary $\big(L^2_{\wh{\mathfrak{g}}} \cap L^{\infty}_{\wh{\mathfrak{g}}}\big)$-neighbourhood $ \mathcal{U} $ of $\wh{\mathfrak{g}}$ there exists an auxiliary $H^\ell_{\wh{\mfg}}$-neighbourhood $\mathcal{V} \subset \mathcal{U}$ of $\wh{\mathfrak{g}}$ with the following property:

            Any Ricci flow $\mathfrak{g}(t)=\p{g(t), \varphi(t)} $ starting from some $ \mathfrak{g}_0=\p{g_0, \varphi_0}  \in \mathcal{V}$ exists for all time, and there is a family of diffeomorphisms $\Phi_t$ such that $\Phi^\ast_t \mathfrak{g}(t) = \p{\Phi_t^*g(t), \Phi^\ast_t \varphi(t)}$ stays in $\mathcal{U}$ for all time, and converges in all derivatives to a static pair $\p{g_{\infty}, \varphi_\infty} = \mathfrak{g}_\infty$ as $t \rightarrow \infty$. Moreover, the convergence is polynomial; that is, for all $\ell \in \mathbb{N}$, there exist constants $C, \beta > 0$ such that
            \begin{equation*}
                \norm{\p{g(t),\varphi(t)} - \p{g_\infty, \varphi_\infty}}_{C^\ell} \leq C\p{t - \tau_{\mathrm{SE}}}^{-\beta}
            \end{equation*}
            for all $t > \tau_{\mathrm{SE}}$, where $\tau_{\mathrm{SE}}$ is (half) the existence time from Proposition \ref{rdtf_short_time}.
        \end{theorem}

        \begin{proof} 
            For this proof, we shall require the definition of balls centred at $\wh{\mfg}$ with respect to several different metrics. Let 
                $$\mcb^{k}\p{r}:=\left\{\mfg=(g,\varphi)\in H^k_{\wh{\mfg}}\p{S^2\mathfrak{M}}\; ;\; \norm{\mfg-\wh\mfg}_{H^k_{\wh\mfg}}<r\right\}.$$

            With a very slight abuse of notation, we let $\mcb^\infty(r)$ denote the analogous ball with respect to the $L^\infty$-norm. Let $C,\varepsilon,\tau_{\mathrm{SE}}>0$ be the constants from Proposition \ref{rdtf_short_time}, and assume without loss of generality that $\mathcal{U}=\mcb^0(\varepsilon)\cap\mcb^\infty(\varepsilon)$. Let $A_0$, $A_\ell$ and $\varepsilon_\ell$ be as in Lemma \ref{rdtf_hl_bounds}, and choose $0<\delta<\varepsilon/A_0$ so small that 
            \begin{enumerate}
                \item The \L ojasiewicz--Simon inequality (Theorem \ref{loj_ineq_thm}) holds on $B^\ell(\delta)$.
                \item For all $\mfg\in \mcb^\ell(\delta)$, $\mu_{\mathrm{AH}}\p{\mfg}\leq \mu_{\mathrm{AH}}\p{\wh{\mfg}}=0$. 
                \item The closure of $\mcb^\ell(\delta)$ is contained in $\mcb^\infty\p{\min\left\{\varepsilon/C,\varepsilon_\ell,\varepsilon\right\}}$.
            \end{enumerate}
    
            Point 3. is possible since $H^\ell_{\wh\mfg}$ embeds continuously in $L^\infty$, and it ensures the existence of the Ricci-DeTurck flow \eqref{rdtf} starting at an arbitrary metric $\mfg_0\in \mcb^\ell(\delta)$. Furthermore, by design, the flow stays in $\mathcal{U}$ for $t\in [0,2\tau_{\mathrm{SE}}]$. Recall that the entropy is analytic on $H^\ell_{\wh\mfg}$, so we can find a $C_\mu>0$ such that $\abs{\mu_{\mathrm{AH}}\p{\mfg}}\leq C_\mu\norm{\mfg-\wh\mfg}_{H^\ell_{\wh\mfg}}$ for all $\mfg\in \mcb^\ell\p{\delta}$.\\

            Let $0<\delta'<\delta/2A^*_\ell$ be some even smaller radius to be determined later. Here $A_\ell^*=\max\left\{A_\ell,1\right\}$ ensures that a Ricci-DeTurck flow starting in $\mcb^\ell(\delta')$ is contained in $\mcb^\ell(\delta)$ for all $t\in [0,2\tau_{\mathrm{SE}}]$ and in $\mcb^\ell\p{\delta/2}$ for $t\in [\tau_{\mathrm{SE}},2\tau_{\mathrm{SE}}]$. Specifically, since any such Ricci-DeTurck flow is in $\mcb^\ell\p{\delta/2}$ at time $t=2\tau_{\mathrm{SE}}$, there is some $T=T(\delta') > 2\tau_{\mathrm{SE}}$ such that any Ricci-DeTurck flow starting in $\mcb^\ell(\delta')$ is contained in $\mcb^\ell(\delta)$ for all $t\in [0,T)$.  \\
    
            By the definition of $T$, we have
            \begin{equation*}
                \sup_{\mathfrak{M}} \abs{\Rm_{\mathfrak{g}(t)}}_{\mathfrak{g}(t)} \leq C_0
            \end{equation*}
            for all $t \in \left[\tau_{\mathrm{SE}},T\right)$. Shi estimates from \cite{shi} yield bounds on all derivatives of the curvature along an ordinary Ricci flow, while Lemma \ref{rdtf_hl_bounds} tells us the DeTurck diffeomorphism is as regular as we like. This implies that, for all $m \in \mathbb{N}$, there is a constant $C_m$ depending only on $\wh{\mathfrak{g}}$ and $ m$ such that
            \begin{equation}\label{shi_bounds}
                \sup_{\mathfrak{M}} \abs{\nabla^m_{\mfg(t)} \Rm_{\mathfrak{g}(t)}}_{\mathfrak{g}(t)} \leq C_m
            \end{equation}
            for all $t \in \left[\tau_{\mathrm{SE}},T\right)$.\\ 

            We now turn to controlling the derivatives of $u(t):=f_{\mfg(t)}-\varphi(t)$. By Lemma \ref{rdtf_hl_bounds}, the family of metrics $\mfg(t)$ is contained in an $H^\ell_{\wh\mfg}$-neighbourhood of $\wh\mfg$ for all $t\in [0,T)$, and by Proposition \ref{entropy_analytic_metric_dep}, $u(t) \in H^{\ell}_{\wh{\mathfrak{g}}}(\mathfrak{M})$. By the Sobolev embedding theorem we even have $u(t)\in C^{2,\alpha}\p{\mathfrak{M}}$. For the higher order derivatives, we consider only $t\in [\tau_{\mathrm{SE}},T)$.\\
            
            We want to bound the $C^m$-norm of $u\p{t}$ for all $m$, but since the metrics we are considering are of class $C^{\ell,\alpha}$, Proposition \ref{isomorph_prop} will only get us so far. We shall instead use \cite[Corollary 2.10]{KY_PE_stability}. In the following we have suppressed the time dependence to simplify the notation. Recall that $u=u(t)$ satisfies the Euler--Lagrange equation \eqref{aux_u_EL}:
            \begin{equation*}
                \p{\Lap_\mfg + n} u = \frac{1}{2}\p{\scal_\mfg + n\p{n+1}} - \frac{1}{2}\abs{du}^2.
            \end{equation*}
            where $\mfg=\mfg(t)$. Since \eqref{shi_bounds} provides uniform bounds on the curvature and its derivatives, we can use \cite[Corollary 2.10]{KY_PE_stability} to ensure that $\Lap_\mfg + n : C^{m,\alpha} \rightarrow C^{m-2,\alpha}$ is an isomorphism for all $m \geq 2$, which allow us to compute as follows:
            \begin{align*}
                \norm{u}_{C^{3,\alpha}} &\leq C\norm{\p{\Lap_\mfg + n} u}_{C^{1,\alpha}}\\
                &\leq C\p{\norm{\scal_\mfg + n\p{n+1}}_{C^{1,\alpha}} + \norm{\abs{du_g}^2}_{C^{1,\alpha}}}\\
                &<\infty,
            \end{align*}
            as $\norm{u}_{C^{2,\alpha}} <\infty$ and \eqref{shi_bounds} implies derivatives of the curvature are bounded for all $t \in \left[\tau_{\mathrm{SE}}, T\right)$. By bootstrapping we can deduce $C^{m,\alpha}$-bounds on $u$ for all $m \in \mathbb{N}$. In particular:
            \begin{equation}\label{stability_u_bounds}
                \sup_{\mathfrak{M}} \abs{\nabla^m_{\mfg(t)}u(t)}_{\mfg(t)} \leq \overline{C}_m
            \end{equation}
            for all $m \in \mathbb{N}$ and all $t \in \left[\tau_{\mathrm{SE}}, T\right)$, where $\overline{C}_m$ depends only on $m $ and $\wh{\mfg}$. Note that the estimates \eqref{shi_bounds} and \eqref{stability_u_bounds} are both diffeomorphism invariant. \\

            Now we modify the flow for $t \geq \tau_{\mathrm{SE}}$ to be more suited for the use of the \L ojasiewicz--Simon inequality of Theorem \ref{loj_ineq_thm}. For this, consider a Ricci-DeTurck flow $\mfg(t)$ starting in $\mcb^\ell(\delta')$ and the family of diffeomorphisms $\left\{\Phi_t\right\}_{t\in [0,T)}$ generated by $X_{\wh{\mfg}}\p{\mfg(t)}$, and satisfying $\Phi_0=\mathrm{Id}$. Add to this the family of diffeomorphisms $\left\{\Psi_t\right\}_{t \in \left[\tau_{\mathrm{SE}},T\right)}$ generated by $-\nabla_{\Phi^\ast_{t}\mfg(t)}u_{\Phi^\ast_{t}\mfg(t)}$ and satisfying the initial condition $\Psi_{\tau_{\mathrm{SE}}} = \Phi_{\tau_{\mathrm{SE}}}^{-1}$. Now consider the continuous flow starting in $\mcb^\ell(\delta')$, defined by
            \begin{equation*}
                \wt{\mathfrak{g}}(t): = \begin{cases}
                    \mathfrak{g}(t)~~&\mathrm{for}~t \in \left[0,\tau_{\mathrm{SE}}\right]\\
                    \p{\Phi_t\circ \Psi_t}^\ast \mathfrak{g}(t)~~&\mathrm{for}~t \in \left[\tau_{\mathrm{SE}},T\right).
                \end{cases}
            \end{equation*}
            
            By possibly decreasing it slightly, we may assume that $T$ is in fact the maximal time such that this new flow exists and stays within $\mcb^\ell(\delta)$ for all $t\in [0,T)$, and all initial metrics in $\mcb^\ell(\delta')$. This yields a \emph{modified} Ricci flow
            \begin{equation*}
                \partial_t\wt{\mathfrak{g}}(t) = \begin{cases}
                    -2\p{\Ric_{\wt\mfg(t)}+n\wt\mfg(t)+\frac{1}2 \mathscr{L}_{X_{\wh\mfg}\p{\wt\mfg(t)}}\wt\mfg(t)}~~&\mathrm{for}~t \in \left[0,\tau_{\mathrm{SE}}\right)\\
                    -2\p{\Ric_{\wt\mfg(t)}+n\wt\mfg(t)+\nabla^2_{\wt\mfg(t)}u_{\wt\mfg(t)}}~~&\mathrm{for}~t \in \p{\tau_{\mathrm{SE}},T}.
                \end{cases}
            \end{equation*}
            
            In particular, \eqref{stability_u_bounds} and the definition of $T$ allow us to replicate the arguments in \cite{bamler_cusps,bamler_symm_spaces} to deduce that Lemma \ref{rdtf_hl_bounds} still holds for the modified flow and all $t\in [0,T)$.\\
                    
            Since $\wt\mfg(t)$ has the same warped product structure as the initial metric, we can write its component-wise flow for $t\in \p{\tau_{\mathrm{SE}},T}$ as
            \begin{equation}\label{mod_warped_product_flow}
                \begin{cases}
                    \partial_t \wt{g}(t) = -2\p{\Ric_{\wt{g}(t)} + n\wt{g}(t) +\nabla^2_{\wt{g}(t)} f_{\wt{g}(t)} - d\wt{\varphi}(t) \otimes d\wt{\varphi}(t)}\\
                    \partial_t \wt{\varphi}(t) = -\p{\Delta_{\wt{g}(t)} \wt{\varphi}(t) + \g{df_{\wt{g}(t)}}{d\wt{\varphi}(t)} - n}.
                \end{cases}
            \end{equation}

            We note that, importantly, this is \eqref{entropy_grad_flow} -- the gradient flow of the entropy $\mu_{\mathrm{AH}}$. \\

            Moving on, we can use Corollary \ref{gagliardo_nirenberg_cor} and \eqref{sobolev_norm_equiv} to get, for every small $\eta > 0$, an $\ell'\in \mathbb{N}$ with $\ell'>\ell$ and constants $C_1, C_2 > 0$, depending on $ \wh{\mathfrak{g}}, \ell$ and $\ell'$, such that for $t\in \p{\tau_{\mathrm{SE}},T}$
            \begin{equation}\label{stability_interp_bounds}
                \norm{\partial_t \wt{\mfg}(t)}_{H^\ell_{\wh{\mathfrak{g}}}} \leq C_1\norm{\partial_t \wt{\mfg}(t)}^\eta_{H^{\ell'}_{\wh{\mathfrak{g}}}} \norm{\partial_t \wt{\mfg}(t)}^{1-\eta}_{L^2_{\wh{\mathfrak{g}}}} \leq C_2 \norm{\partial_t \wt{\mfg}(t)}^{1-\eta}_{L^2_{\wh{\mathfrak{g}}}}.    
            \end{equation}
            
            The second inequality and use of Theorem \ref{gagliardo_nirenberg_cor} are justified by Lemma \ref{rdtf_hl_bounds}. Next, let $\theta$ be the exponent from Theorem \ref{loj_ineq_thm} and choose $\eta>0$ such that $\sigma := \frac{\theta}{2} - \eta + \frac{\theta \eta}{2} > 0$. Note that $1-\sigma = \p{1-\frac{\theta}{2}}\p{1+\eta}$. Then \eqref{stability_interp_bounds} and the \L ojasiewicz--Simon inequality imply that
            \begin{align*}
                \frac{d}{dt}\mu_{\mathrm{AH}}\p{\wt{\mathfrak{g}}(t)} &\geq C_4\norm{\partial_t \wt{\mfg}(t)}^{1+\eta}_{L^2_{\wh{\mathfrak{g}}}} \norm{\partial_t\wt{\mfg}(t)}^{1-\eta}_{L^2_{\wh{\mathfrak{g}}}}\\
                &\geq \frac{C_4}{C_2}\norm{\partial_t \wt{\mfg}(t)}^{1+\eta}_{L^2_{\wh{\mathfrak{g}}}} \norm{\partial_t\wt{\mfg}(t)}_{H^\ell_{\wh{\mathfrak{g}}}}\\
                &\geq \frac{C_4}{C_2\sqrt{C_3}}\abs{\mu_{\mathrm{AH}}\p{\wt{\mathfrak{g}}(t)} - \mu_{\mathrm{AH}}\p{\wh{\mathfrak{g}}}}^{\p{1-\frac{\theta}{2}}\p{1+\eta}}\norm{\partial_t\wt{\mfg}(t)}_{H^\ell_{\wh{\mathfrak{g}}}}\\
                &= :\wt{C}\abs{\mu_{\mathrm{AH}}\p{\wt{\mathfrak{g}}(t)} - \mu_{\mathrm{AH}}\p{\wh{\mathfrak{g}}}}^{1-\sigma}\norm{\partial_t\wt{\mfg}(t)}_{H^\ell_{\wh{\mathfrak{g}}}},
            \end{align*}
            where $C_3$ is the constant from \eqref{aux_loj_ineq}, while the first inequality and constant $C_4$ comes from $\partial_t \wt{\mfg}(t)$ satisfying \eqref{mod_warped_product_flow}, albeit with respect to a differently weighted but easily comparable $L^2$-norm. We shall keep the vanishing entropy $\mu_{\mathrm{AH}}\p{\wh\mfg}=0$ throughout the proof, as it appears as such in the \L ojasiewicz--Simon inequality. Since all the points from the definition of $\delta$ hold along the flow, we have $\mu_{\mathrm{AH}}\p{\wt{\mathfrak{g}}(t)} \leq \mu_{\mathrm{AH}}\p{\wh{\mathfrak{g}}} = 0$ for $t\in[0,T)$, and we may deduce from the inequality above that
            \begin{align*}
                -\frac{d}{dt}\abs{\mu_{\mathrm{AH}}\p{\wt{\mathfrak{g}}(t)} - \mu_{\mathrm{AH}}\p{\wh{\mathfrak{g}}}}^\sigma &= \sigma \abs{\mu_{\mathrm{AH}}\p{\wt{\mathfrak{g}}(t)} - \mu_{\mathrm{AH}}\p{\wh{\mathfrak{g}}}}^{\sigma-1} \frac{d}{dt} \mu_{\mathrm{AH}}\p{\wt{\mathfrak{g}}(t)}\\
                &\geq \sigma\wt{C}\norm{\partial_t \wt{\mfg}(t)}_{H^\ell_{\wh{\mathfrak{g}}}}.
            \end{align*}

            Integrating this from $\tau_{\mathrm{SE}}$ to $T$ yields
            \begin{equation}\label{loj_stability_application}
                \int^{T}_{\tau_{\mathrm{SE}}} \norm{\partial_t \wt{\mfg}(t)}_{H^\ell_{\wh{\mathfrak{g}}}} dt \leq \frac{1}{\sigma\wt{C}}\abs{\mu_{\mathrm{AH}}\p{\wt{\mathfrak{g}}\p{\tau_{\mathrm{SE}}}} - \mu_{\mathrm{AH}}\p{\wh{\mathfrak{g}}}}^\sigma\leq\frac{C_\mu}{\sigma\wt{C}}\norm{\wt{\mfg}\p{\tau_{\mathrm{SE}}}-\wh{\mfg}}_{H^\ell_{\wh{\mfg}}}^\sigma,
            \end{equation}
            where the second inequality is again a consequence of the choice of $\delta$. This gives the bound
                $$\norm{\wt{\mfg}(T)-\wt{\mfg}(\tau_{\mathrm{SE}})}_{H^\ell_{\wh{g}}}\leq \int^{T}_{\tau_{\mathrm{SE}}} \norm{\partial_t \wt{\mfg}(t)}_{H^\ell_{\wh{\mathfrak{g}}}} dt\leq \frac{C_\mu}{\sigma\wt{C}}\norm{\wt{\mfg}\p{\tau_{\mathrm{SE}}}-\wh{\mfg}}_{H^\ell_{\wh{\mfg}}}^\sigma.$$
            
            We are finally in a position to specify $\delta'$ further. Choose it so small that $A_\ell\delta'+\frac{C_\mu}{\sigma\wt{C}}\p{A_\ell\delta'}^\sigma<\delta$. Then, for every initial metric in $\mcb^\ell\p{\delta'}$ we have $\wt{\mfg}\p{\tau_{\mathrm{SE}}}\in \mcb^\ell(A_\ell\delta')$, and
            \begin{align*}
                \norm{\wt{\mathfrak{g}}(T)-\wh{\mfg}}_{H^\ell_{\wh{\mathfrak{g}}}} &\leq \norm{ \wt{\mathfrak{g}}(\tau_{\mathrm{SE}})-\wh{\mathfrak{g}}}_{H^\ell_{\wh{\mathfrak{g}}}} + \norm{\wt{\mathfrak{g}}(T)-\wt{\mathfrak{g}}(\tau_{\mathrm{SE}}) }_{H^\ell_{\wh{\mathfrak{g}}}}\\ &\leq \norm{ \wt{\mathfrak{g}}(\tau_{\mathrm{SE}})-\wh{\mathfrak{g}}}_{H^\ell_{\wh{\mathfrak{g}}}}+\frac{C_\mu}{\sigma \wt{C}}\norm{ \wt{\mathfrak{g}}(\tau_{\mathrm{SE}})-\wh{\mathfrak{g}}}_{H^\ell_{\wh{\mathfrak{g}}}}^\sigma\\
                & <\delta.
            \end{align*}

            This implies that no flow starting in $\mcb^\ell(\delta')$ leaves $\mcb^\ell(\delta)$ at time $t=T$, and it can not be maximal in the sense defined. The only logical conclusion is that $ T = \infty$, which then implies $\wt{\mathfrak{g}}(t) \rightarrow \mathfrak{g}_\infty$ for $t\to\infty$ and some limit metric $\mathfrak{g}_\infty \in \mathcal{U}$. \\
             
            Using the \L ojasiewicz--Simon inequality again, one can compute that, for $t > \tau_{\mathrm{SE}}$,
            \begin{align*}
                \frac{d}{dt} \abs{\mu_{\mathrm{AH}}\p{\wt{\mathfrak{g}}(t)} - \mu_{\mathrm{AH}}\p{\wh{\mathfrak{g}}}}^{\theta-1}& =(1-\theta)\abs{\mu_{\mathrm{AH}}\p{\wt\mfg(t)}-\mu_{\mathrm{AH}}\p{\wh{\mathfrak{g}}}}^{\theta-2}\frac{d}{dt}\mu_{\mathrm{AH}}\p{\wt\mfg(t)}\\ & \geq \frac{1-\theta}{C_3}\norm{\nabla \mu_{\mathrm{AH}}\p{\wt\mfg(t)}}_{L^2_{\wh{\mfg}}}^{-2}\p{\nabla \mu_{\mathrm{AH}}\p{\wt\mfg(t)},\partial_t\wt\mfg(t)}_{L^2_{\wh\mfg}}\\
                &\geq \wh{C},
            \end{align*}
            for some constant $\wh{C} > 0$. The first inequality uses the \L ojasiewicz-Simon inequality and comparability of the volume forms, while the second uses that $\partial_t\wt\mfg(t)$ agrees with the gradient flow up to a factor $e^{-u_{\wt\mfg(t)}}$. Integrating this from $t=\tau_{\mathrm{SE}}$ to $T$ yields
            \begin{equation}\label{entropy_conv_rate}
                \abs{\mu_{\mathrm{AH}}\p{\wt{\mathfrak{g}}(T)} - \mu_{\mathrm{AH}}\p{\wh{\mathfrak{g}}}} \leq \wh{C}\p{T - \tau_{\mathrm{SE}}}^{-\frac{1}{1-\theta}}.
            \end{equation}

            It is important to note that we may assume $\theta < 1$, as the \L ojasiewicz--Simon inequaltiy still holds after decreasing the exponent $\theta$, as long as the difference of the entropies is at most $1$. Sending $t \rightarrow \infty$ tells us $\mu_{\mathrm{AH}}\p{\mathfrak{g}_\infty} = \mu_{\mathrm{AH}}\p{\wh{\mathfrak{g}}}$. Therefore, since $\wh{\mathfrak{g}}$ is a local maximum of the entropy, so too is $\mathfrak{g}_\infty$. We may then use \cite[Corollary 5.15]{DKM} to deduce that $\mathfrak{g}_\infty$ must be a Poincar\'e--Einstein metric. Moreover, since warped product structure is preserved by the modified Ricci flow, we know that $\mfg_\infty=\p{g_\infty, \varphi_\infty}$ is a static pair.\\

            Finally, combining \eqref{loj_stability_application} and \eqref{entropy_conv_rate} yields the claimed convergence rate:
            \begin{equation*}
                \norm{\wt{\mathfrak{g}}(T) - \wh{\mathfrak{g}}}_{H^\ell_{\wh{\mathfrak{g}}}} \leq \frac{1}{\sigma\wt{C}}\abs{\mu_{\mathrm{AH}}\p{\wt{\mathfrak{g}}(T)} - \mu_{\mathrm{AH}}\p{\wh{\mathfrak{g}}}}^\sigma \leq \frac{\wh{C}^\sigma}{\sigma \wt{C}}\p{T - \tau_{\mathrm{SE}}}^{-\frac{\sigma}{1-\theta}}.
            \end{equation*}
            
            Using Lemma \ref{rdtf_hl_bounds} and the Sobolev embedding theorem yields the desired convergence in all derivatives. Note also that by diffeomorphism invariance the above results for the modified Ricci flow hold for the standard Ricci flow 
            $$\mfg_{\mathrm{RF}}(t):=\begin{cases}\p{\Phi_t}_\ast\wt{\mfg}(t) &\text{for }t\in[0,\tau_{\mathrm{SE}}]\\ \p{\Phi_t \circ \Psi_t \circ \Phi_t}_\ast \wt\mfg\p{t}&\text{for }t\in[\tau_{\mathrm{SE}},\infty)\end{cases},$$
            satisfying $\partial_t\mfg_{\mathrm{RF}}(t)=-2\p{\Ric_{\mfg_{\mathrm{RF}}(t)}+n\mfg_{\mathrm{RF}}(t)}$ for all $t\geq 0$.
        \end{proof}

        Now for the instability theorem (Theorem \ref{mainthm_instability}), which we restate here in an expanded form.

        \begin{theorem}\label{static_instability}
            Let $\p{M,\wh{g}, \wh{\varphi}}$ be a static triple of class $C^{\ell,\alpha}$ ($\ell>\frac{n}{2}+2$) and dimension $n \geq 3$, and set $\wh{\mathfrak{g}} := e^{-2\wh{\varphi}}\;d\tau^2 + \wh{g}$. Assume that $\p{\wh{g}, \wh{\varphi}}$ is \textbf{not} a local maximiser of the entropy $\mu_{\mathrm{AH}}=\mu_{\mathrm{AH},\wh{g}}$. Then there exists a non-trivial ancient Ricci flow $\left\{\mathfrak{g}(t)\right\}_{t \in (-\infty,0]} = \left\{\p{g(t), \varphi(t)}\right\}_{t \in (-\infty,0]}$ and a family of diffeomorphisms $\left\{\Phi_t\right\}_{t \in (-\infty,0]}$, such that $\Phi_t^*\mathfrak{g}(t) \to \wh{\mathfrak{g}}$ in all derivatives as $t \rightarrow -\infty$.
        \end{theorem}

        \begin{proof}
            Consider $\left\{\p{g_i,\varphi_i}\right\} = \left\{\mathfrak{g}_i\right\} \subset \mathcal{R}^\ell_{\wh\mfg}\p{\mathfrak{M},\wh{\mathfrak{g}}}$, a sequence of smooth metrics on $\mathfrak{M}=\mathbb{S}^1\times M$ such that $\mathfrak{g}_i \rightarrow \wh{\mathfrak{g}}$ in the $\big(L^2_{\wh{\mathfrak{g}}}\cap L^\infty\big)$-sense. Furthermore, suppose $\mu_{\mathrm{AH}}\p{\mathfrak{g}_i} > \mu_{\mathrm{AH}}\p{\wh{\mathfrak{g}}} = 0$ for all $i$, and let $\wt{\mathfrak{g}}_i(t)$ be the modified Ricci flow starting at $\mathfrak{g}_i$, as defined during the proof of Theorem \ref{static_stability}. Importantly, since $\mu_{\mathrm{AH}}$ is monotonically increasing under the flow, we have 
            $$\mu_{\mathrm{AH}}\p{\wt\mfg_i(t)}=\mu_{\mathrm{AH}}\p{\wt\mfg_i(t)}-\mu_{\mathrm{AH}}\p{\wh\mfg}>0$$
            for all $t\geq 0$. Furthermore, for $t\in [\tau_{\mathrm{SE}},2\tau_{\mathrm{SE}}]$
            \begin{equation}\label{instab_sob_conv}
                \wt{\mathfrak{g}}_i(t) \overset{L^\infty}\longrightarrow \wh{\mathfrak{g}}\qquad \text{for}\qquad i\to \infty
            \end{equation}
            by Lemma \ref{rdtf_hl_bounds}. The same convergence is true in the $H^\ell_{\wh{\mfg}}$-norm.\\

            Now, take $\mathcal{U} =\mcb^\ell(2\varepsilon)$ with $\eps > 0$ small enough that Theorem \ref{loj_ineq_thm} holds on $\mathcal{U}$. For large enough $i$, convergence implies $\wt\mfg_i(t) \in \mathcal{U}$ for all $t \in [\tau_{\mathrm{SE}},2\tau_{\mathrm{SE}}]$. For these $i$ and $t$, we can apply the \L ojasiewicz--Simon inequality to deduce that
            \begin{equation*}
                \frac{d}{dt}\p{\mu_{\mathrm{AH}}\p{\wt{\mathfrak{g}}_i(t)} - \mu_{\mathrm{AH}}\p{\wh{\mathfrak{g}}}}^{\theta-1} \geq -C_\theta,
            \end{equation*}
            where $\theta$ is the exponent from \eqref{aux_loj_ineq} and $C_\theta=\frac{1-\theta}{C} \geq 0$. Integrating this inequality in time from $t = s$ to $ T$, we find that
            \begin{equation}\label{entropy_diff_bound}
                \left[\p{\mu_{\mathrm{AH}}\p{\wt{\mathfrak{g}}_i(s)} - \mu_{\mathrm{AH}}\p{\wh\mfg}}^{\theta-1} - C_\theta\p{T-s}\right]^{-\frac{1}{1-\theta}} \leq \mu_{\mathrm{AH}}\p{\wt{\mathfrak{g}}_i\p{T}} - \mu_{\mathrm{AH}}\p{\wh{\mathfrak{g}}},
            \end{equation}
            provided $s$ is such that $\wt{\mathfrak{g}}_i\p{s} \in \mathcal{U}$. Let $T_i$ be the first time that the flow $\wt{g}_i(t)$ leaves $\mcb^\ell(\varepsilon)$; that is, the first time at which
            \begin{equation*}
                \norm{\wt{\mathfrak{g}}_i\p{T_i} - \wh{\mathfrak{g}}}_{H^\ell_{\wh{\mathfrak{g}}}} = \eps.
            \end{equation*}

            \begin{claim}
                We have $T_i \rightarrow \infty$ as $i \rightarrow \infty$.
            \end{claim} 
            
            \begin{proof}[Proof of claim.]
                Assume for contradiction that $T_i$ (possibly after passing to a subsequence) is a bounded sequence, with $T_\infty = \lim_{i \rightarrow \infty} T_i$. Using Corollary \ref{gagliardo_nirenberg_cor} like before, we have, for any given $\eta > 0$, that for $t\in [0,T_i]$
                \begin{equation*}\label{instability_interp_bounds}
                    \norm{\partial_t \wt\mfg_i(t)}_{H^\ell_{\wh{\mathfrak{g}}}} \leq C_1\norm{\partial_t \wt\mfg_i(t)}^\eta_{H^{\ell'}_{\wh{\mathfrak{g}}}} \norm{\partial_t \wt\mfg_i(t)}^{1-\eta}_{L^2_{\wh{\mathfrak{g}}}} \leq C_2 \norm{\partial_t \wt\mfg_i(t)}^{1-\eta}_{L^2_{\wh{\mathfrak{g}}}}
                \end{equation*}
                for some $ \ell < \ell' \in \mathbb{N}$, and constants $C_1, C_2 > 0$ depending only on $ \wh{\mathfrak{g}}$, $ \ell$ and $\ell'$. Note that the second inequality and the application of Theorem \ref{gagliardo_nirenberg_cor} are justified by Lemma \ref{rdtf_hl_bounds}, which implies the $H^{\ell'}_{\wh\mfg}$-norm of $\partial_t \wt\mfg_i(t)$ is finite as long as the flow stays in $\mathcal{U}$. Let $\eta>0$ be small enough so that $\sigma := \frac{\theta}{2} - \eta + \frac{\theta \eta}{2} > 0$. Then we can compute as we did in the proof of Theorem \ref{static_stability} to see that
                \begin{equation}\label{integrate_this}
                    \frac{d}{dt}\p{\mu_{\mathrm{AH}}\p{\wt{\mathfrak{g}}_i(t)} - \mu_{\mathrm{AH}}\p{\wh\mfg} }^\sigma \geq C_\sigma\norm{\partial_t \wt\mfg_i(t)}_{H^\ell_{\wh{\mathfrak{g}}}}.
                \end{equation}
    
                Integrating this new bound with respect to time from $\tau_{\mathrm{SE}}$ to $T_i$ allows us to deduce that
                \begin{equation}\label{non_trivial_rf_estim}
                    \eps = \norm{\wt{\mathfrak{g}}_i\p{T_i} - \wh{\mathfrak{g}}}_{H^\ell_{\wh{\mathfrak{g}}}} \leq \norm{\wt{\mathfrak{g}}_i\p{\tau_{\mathrm{SE}}} - \wh{\mathfrak{g}}}_{H^\ell_{\wh{\mathfrak{g}}}} + \frac{1}{C_\sigma}\p{\mu_{\mathrm{AH}}\p{\wt{\mathfrak{g}}_i\p{T_i}} - \mu_{\mathrm{AH}}\p{\wh{\mathfrak{g}}}}^\sigma.
                \end{equation}
    
                Note that \eqref{non_trivial_rf_estim} does not rely on $T_i$ (or the limit, $T_\infty$) being finite. Additionally, by \eqref{instab_sob_conv} we can make $\norm{\wt{\mathfrak{g}}_i\p{\tau_{\mathrm{SE}}}- \wh{\mathfrak{g}}}_{H^\ell_{\wh{\mathfrak{g}}}}$ as small as we want by taking $i \gg 1$. Using this and rearranging \eqref{non_trivial_rf_estim}, we obtain the uniform lower bound 
                \begin{equation}\label{unif_entropy_lower_bnd}
                    \eps' \leq \mu_{\mathrm{AH}}\p{\wt{\mathfrak{g}}_i\p{T_i}}, 
                \end{equation}
                for some $\eps'>0$. Hence, letting $i \rightarrow \infty$, we have $\eps' \leq \mu_{\mathrm{AH}}\p{\wt{\mathfrak{g}}_\infty\p{T_\infty}}$. On the other hand, by \eqref{instab_sob_conv}, we know that $\wt{\mathfrak{g}}_i\p{T_i} \rightarrow \wh{\mathfrak{g}}$ in the $H^\ell_{\wh{\mathfrak{g}}}$-norm. Thus $\norm{\wt{\mathfrak{g}}_i\p{T_i} - \wh{\mathfrak{g}}}_{C^0} \rightarrow 0$ by the Sobolev embedding theorem. Since the limit here is unique, we have $ \lim_{i \rightarrow \infty} \wt{\mathfrak{g}}_i\p{T_i}=\wh{\mathfrak{g}}$, which contradicts the uniformly positive lower entropy bound \eqref{unif_entropy_lower_bnd}. We must therefore have $\lim_{i\to\infty}T_i =\infty$.
            \end{proof}

            We now define a new sequence of metrics by applying a time shift. Let $\wt{\mathfrak{g}}^*_i(t) := \wt{\mathfrak{g}}_i\p{t + T_i}$ for $t \in \left[\tau_{\mathrm{SE}}-T_i,0\right]$. Then, by the previous part of the proof, we have
            \begin{align*}
                \norm{\wt{\mathfrak{g}}^*_i(t) - \wh{\mathfrak{g}}}_{H^\ell_{\wh{\mathfrak{g}}}} \leq \eps\qquad \qquad &\text{for}\;\; t\in\left[\tau_{\mathrm{SE}}-T_i,0\right],\\
                \wt{\mathfrak{g}}^*_i\p{\tau_{\mathrm{SE}}-T_i} \overset{H^\ell_{\wh\mfg}}\longrightarrow \wh{\mathfrak{g}}\qquad \qquad &\text{for}\;\; i\to \infty.
            \end{align*}
             
             Additionally, note that we can apply Hamilton's compactness theorem (see \cite{hamilton_compactness}) to pass to a convergent subsequence, so that $\wt{\mathfrak{g}}^*_i(t) \rightarrow \wt{\mathfrak{g}}(t)$ in $H^\ell_{\wh{\mathfrak{g}}}$ as $i\to \infty$, where $\left\{\wt{\mathfrak{g}}(t)\right\}_{t \in (-\infty,0]} = \left\{\p{\wt{g}(t), \wt{\varphi}(t)}\right\}_{t \in (-\infty,0]}$ is an ancient modified Ricci flow satisfying the system \eqref{mod_warped_product_flow}.\\

            Next, for $t \in (-\infty,0]$, let $ \left\{\Psi_t\right\}_{t\in(-\infty,0]}$ be the family of diffeomorphisms generated by $-\nabla_{\wt\mfg(t)}u_{\wt\mfg(t)}$. Then $\mathfrak{g}(t) := \Psi^\ast_t \wt{\mathfrak{g}}(t)$ is a Ricci flow. Furthermore, sending $i \to \infty$ in \eqref{non_trivial_rf_estim}, we have
            \begin{equation*}
                \eps \leq \frac{1}{C_\sigma}\p{\mu_{\mathrm{AH}}\p{\mathfrak{g}(0)} - \mu_{\mathrm{AH}}\p{\wh{\mathfrak{g}}}}^\sigma \leq \norm{\nabla \mu_{\mathrm{AH}}\p{\mathfrak{g}(0)}}^{\frac{2\sigma}{2-\theta}}_{L^2_{\wh{\mathfrak{g}}}},
            \end{equation*}
            where the second inequality is due Theorem \ref{loj_ineq_thm}. Note that $\mathfrak{g}(0)$ appears since $\wt{\mathfrak{g}}^*_i(0) = \wt{\mathfrak{g}}_i\p{T_i}$, and the latter is what is present in \eqref{non_trivial_rf_estim}. This, together with the monotonicity of the entropy (Lemma \ref{entropy_monot}), implies that $\mathfrak{g}(t)$ is a non-trivial Ricci flow. That is, it is not a stationary solution of \eqref{entropy_grad_flow}; hence, by diffeomorphism invariance, it is not a stationary solution of \eqref{eq:Listflow_withLambda}. \\

            By integrating \eqref{integrate_this} from $\tau_{\mathrm{SE}}$ to $t+T_i$, and applying \eqref{entropy_diff_bound} (with $s = t+ T_i$ and $T = T_i$), we find that, for $t \in \left[\tau_{\mathrm{SE}}-T_i, 0\right]$ and $i \gg 1$, 
            \begin{align*}
                \norm{\wt\mfg^*_i\p{\tau_{\mathrm{SE}}-T_i} - \wt\mfg^*_i(t)}_{H^\ell_{\wh\mfg}} &= \norm{\wt\mfg_i\p{\tau_{\mathrm{SE}}}-\wt\mfg_i\p{t + T_i}}_{H^\ell_{\wh{\mathfrak{g}}}} \\
                &\leq \frac{1}{C_\sigma}\int_{\tau_{\mathrm{SE}}}^{t+T_i}\frac{d}{ds}\p{\mu_{\mathrm{AH}}\p{\wt\mfg_i(s)}-\mu_{\mathrm{AH}}\p{\wh\mfg}}^\sigma\; ds\\
                &\leq \frac{1}{C_\sigma}\p{\mu_{\mathrm{AH}}\p{\wt\mfg_i\p{t+T_i}} - \mu_{\mathrm{AH}}\p{\wh\mfg}}^\sigma\\
                &\leq \frac{1}{C_\sigma}\p{-C_\theta t + \p{\mu_{\mathrm{AH}}\p{\wt\mfg_i\p{T_i}} - \mu_{\mathrm{AH}}\p{\wh\mfg}}^{\theta-1}}^{-\frac{\sigma}{1-\theta}}\\
                &\leq \p{-C_3t + C_4}^{-\frac{\sigma}{1-\theta}},
            \end{align*}
            for some $C_3,C_4 > 0$. Note that the equality is due to the definition of $\wt\mfg^*_i$, while the second inequality relies on $\mu_{\mathrm{AH}}\p{\wt\mfg_i\p{\tau_{\mathrm{SE}}}} > 0$, and the last inequality is due to \eqref{unif_entropy_lower_bnd} and $\theta < 1$, so that
            \begin{equation*}
                \p{\mu_{\mathrm{AH}}\p{\wt\mfg_i\p{T_i}} - \mu_{\mathrm{AH}}\p{\wh\mfg}}^{\theta - 1} \leq \p{\eps'}^{1-\theta}.
            \end{equation*}
            
            The triangle inequality now yields
            \begin{align*}
                \norm{ \wt\mfg(t)-\wh\mfg}_{H^\ell_{\wh\mfg}} \leq \norm{ \wt\mfg(t)-\wt\mfg^*_i(t)}_{H^\ell_{\wh\mfg}}+ \p{-C_4 t + C_5}^{-\frac{\sigma}{1-\theta}}+\norm{\wt\mfg^*_i\p{\tau_{\mathrm{SE}}-T_i}-\wh\mfg}_{H^\ell_{\wh\mfg}},
            \end{align*}
            and by letting $i\to \infty$, it reduces to
            \begin{align*}
                \norm{ \wt\mfg(t)-\wh\mfg}_{H^\ell_{\wh\mfg}} \leq  \p{-C_4 t + C_5}^{-\frac{\sigma}{1-\theta}}.
            \end{align*}
            
           Finally, sending $t \rightarrow -\infty$ we have $\p{\Psi_t}_\ast \mathfrak{g}(t)=\wt\mfg(t) \rightarrow \wh\mfg$ in $H^\ell_{\wh\mfg}$. Convergence in all derivatives follows from Lemma \ref{rdtf_hl_bounds} and the Sobolev embedding theorem. 
        \end{proof}
    
\pagebreak    
\appendix
\section{Appendix}\label{appendix}
    This appendix is devoted to recording some variational formulas we use throughout the paper. The proofs can be found in, say, \cite{besse}.

    \begin{proposition}
        Let $\left(M,g\right)$ be a Riemannian manifold. Then, with $g_s = g + sh$ for $\abs{s} \ll 1$, we have
        
        \begin{align}
            \left. \frac{d}{ds}\right\vert_{s=0} \Ric_{g_s} &= \frac{1}{2}\Lap_L h -\delta^*_g\delta_g h - \frac{1}{2}\nabla^2_g \tr_g\left(h\right), \label{ric_var}\\
            \left. \frac{d}{ds}\right\vert_{s=0} \scal_{g_s} &= \Lap_g\tr_g\left(h\right) + \delta_g \delta_g h - \left<\Ric_g, h\right>. \label{scal_var}
        \end{align}
    \end{proposition}

    \begin{proposition}   
        Let $\left(M,g\right)$ be a Riemannian manifold, $f,v,\varphi,\psi \in C^1\p{M}$, and $h$ a symmetric $2$-tensor on $M$. Then, with $g_s = g + sh$ (for $\abs{s} \ll 1$), $f_s = f + sv$, and $\varphi_s = \varphi + s\psi$, we have
        \begin{align}
            \left. \frac{d}{ds}\right\vert_{s=0} e^{-f_s} dV_{g_s} &= \left(-v + \frac{1}{2}\tr_g\left(h\right)\right) e^{-f} dV_g, \label{f_wgtd_vol_var}\\
            \left. \frac{d}{ds}\right|_{s=0} \nabla^{2}_{g_s} f_s &= \nabla^2_g v -\delta^*_g\p{h\cdot df}+\frac{1}{2}h\times \nabla^2_g f + \frac{1}{2}\nabla_{g,\nabla_g f} h, \label{hess_var}\\
            \left. \frac{d}{ds}\right|_{s=0} \Lap_{g_s} \varphi_s &= \Lap_g \psi + \g{h}{\nabla^2_g \varphi}-\g{\delta_g h+\frac{1}{2}d\tr_g h}{d\varphi}, \label{lap_var}\\
            \left. \frac{d}{ds}\right|_{s=0} \g{d\varphi_s}{df_s}_{g_s} &= -\g{h}{df \otimes d\varphi}_g+\g{dv}{d\varphi}_g+\g{df}{d\psi}_g. \label{inner_prod_var}
        \end{align}

         Here $\times $ denotes the symmetrised contraction $\p{h\times k}_{ij}=g^{lm}\p{h_{il}k_{mj}+h_{jl}k_{mi}}$ for symmetric $2$-tensors $h,k$.
    \end{proposition}

    Note that \eqref{hess_var} is written in a slightly non-standard way compared to the version in \cite{besse}. However, it is equivalent and is convenient for our purposes.

    \begin{proposition}
        Let $h$ be a symmetric $2$-tensor on a Riemannian manifold $\left(M,g\right)$. Further let $\omega$ be a $1$-form on $M$ and write $\omega_s = \omega + s\omega'$. Then
        \begin{align}
            \left(\left. \frac{d}{ds}\right\vert_{s=0} \delta_{g_s} h\right) &= -\frac{1}{2} \delta_g\left(h \times h\right) - \frac{1}{2}h \cdot d\tr_g\left(h\right) + \frac{1}{4}d\abs{h}^2_g \label{symm2_div_var}\\
            \left. \frac{d}{ds}\right\vert_{s=0} \delta_{g_s} \omega_s &= -\left(g^{ij}\nabla_i \omega_i\right)' = \delta_g \omega' + \frac{1}{2}g^{ij}\left(2\nabla_i h_{j\ell} - \nabla_\ell h_{ij}\right) \omega^\ell + \left<h, \nabla_g \omega\right> \label{oneform_var}\\
            \left. \frac{d}{ds} \right\vert_{s=0}\left( \delta_{g_s}\delta_{g_s}h\right) &= \delta_g \left(\left.\frac{d}{ds}\right\vert_{s=0} \delta_{g_s} h\right) - \abs{\delta_g h}^2_g - \frac{1}{2}\left<\delta_g h, d\tr_g\left(h\right)\right>_g + \left<h, \nabla_g h\right>_g \label{div_squared_var}
        \end{align}
    \end{proposition}

    \begin{theorem}[Yamabe Problem]\label{Yamabe}
        Let $\ell>\frac{n}{2}+2$, $\frac{n}{2}<\beta<n$, and $g\in \mathcal{R}^{\ell}_{\frac{1}{2}}\p{M,\wh{g}}\cap \mathcal{R}^{2,\alpha}_{\beta}\p{M,\wh{g}}$. Then there is a unique $\omega\in H^{\ell}_{\frac{1}{2}}\p{M}\cap C^{2,\alpha}_\beta(M)$ such that $\wt{g}=e^{2\omega}g$ has constant scalar curvature $\scal_{\wt{g}}=-n(n-1)$.
    \end{theorem}
    \begin{proof} The heavy lifting is done by \cite[Theorem 1.1/6.8]{ALM}, though the right application requires a little finesse. With regards to the assumptions of that theorem, it will only work in our setting when the dimension $n$ is odd. In that case $g\in \mathcal{R}^\ell_{\frac{1}{2}}\p{M,\wh{g}}$ is of class $\mathscr{H}^{\ell,2;\frac{n+1}{2}}$, and the indices satisfy the requirements of the theorem. However, when $n$ is even, the best fortified space we can achieve is $\mathscr{H}^{\ell,2;\frac{n}{2}}$, which is just shy of the requirements. Fortunately, a closer inspection of the proof of \cite[Theorem 6.8]{ALM} reveals that we may bypass the notion of fortified spaces, as in our case $g$ satisfies the slightly stronger condition $g-\wh{g}\in H^{\ell}_{m-\frac{n}{2}}$ (and in particular $\scal_g+n(n-1)\in H^{\ell-2}_{m-\frac{n}{2}}$) for $m=\frac{n+1}{2}$. The proof proceeds to construct a solution $u\in H^{\ell}_{m-\frac{n}{2}}\p{M}$ to the Yamabe equation 
    $$4\frac{n-1}{n-2}\Delta_gu+\scal_g(1+u)=-n(n-1)(1+u)^{\frac{n+2}{n-2}}.$$

    Setting $\wt{g}:=\p{1+u}^{\frac{4}{n-2}}g$, we have found our constant scalar curvature metric, though we still need to check that the conformal factor has the right decay. By defining $\omega:=\frac{2}{n-2}\ln\p{1+u}$, we obtain another conformal factor in $H^\ell_{\frac{1}2}\p{M}$, satisfying the transformed Yamabe equation
    $$\p{\Delta_g+n}\omega =\frac{n-2}{2}\abs{d\omega}^2_g-\frac{1}{2(n-1)}\p{\scal_g+n(n-1)}-\frac{n}{2}F(\omega),$$
    where $F(\omega)=e^{2\omega}-1-2\omega=\mathcal{O}\p{\omega^2}$. By Lemma \ref{wgtd_sob_emb}, we have $\omega\in C^{2,\alpha}_{1/2}\p{M}$. It follows that $\abs{d\omega}^2_g\in C^{1,\alpha}_1\p{M}$ and $F(\omega)\in C^{2,\alpha}_{1}\p{M}$. Since $g\in \mathcal{R}^{2,\alpha}_{\beta}\p{M,\wh{g}}$, and since $\beta>\frac{n}{2}>1$ we have 
    $$\frac{n-2}{2}\abs{d\omega}^2_g-\frac{1}{2(n-1)}\p{\scal_g+n(n-1)}-\frac{n}{2}F(\omega)\in C^{0,\alpha}_{1}\p{M}.$$

    It follows from Proposition \ref{isomorph_prop}, that $\omega\in C^{2,\alpha}_{1}\p{M}$. We may now bootstrap up to $C^{2,\alpha}_{\beta}\p{M}$, as in \cite[Proposition 4.3]{DKM}.   
    \end{proof}

    Finally, we will need the following version of the Gagliardo--Nirenberg inequality.

    \begin{theorem}\label{gagliardo_nirenberg} Let $g$ be AH of class $C^{\ell,\alpha}$ and let $T\in H^{k}_{1/2}(S^2 M)\cap L^q_{1/2}(S^2 M)$ for $k\leq \ell$ and $k+\frac{n}{q}\geq \frac{n}{2}$. Then, for every $0\leq i\leq k$, there is a constant $C=C(k,q,n)$ such that 
    $$\norm{\nabla^{i}T}_{L^p_{1/2}}\leq C\norm{\nabla^k T}_{L^2_{1/2}}^\theta \norm{T}_{L^q_{1/2}}^{1-\theta}$$
    for every pair $(p,\theta)$ such that 
    $$\frac{n}{2}-\frac{n}{p}\leq k-i\qquad \& \qquad \theta\p{k-\frac{n}{2}+\frac{n}{q}}=i-\frac{n}{p}+\frac{n}{q}$$
    and $\frac{i}{k}\leq \theta \leq 1$.
    \end{theorem}
    \begin{proof}
        By \cite[Lemma 3.9]{Lee} we can approximate $T$ by a sequence $\p{T_j}_{j}\subset C^\infty_c(S^2 M)$, so $T_j\to T$ in $H^k_{1/2}(S^ 2M)$. By the classical Gagliardo-Nirenberg estimate for tensors (see \cite[Corollary 12.7]{hamilton_interpolation} - although it is for compact manifolds, the argument is the same for compactly supported tensors), there is a constant $C > 0$ depending only on $k,q,n$ such that for any pair $(p,\theta)$ with 
        $$\theta\p{k-\frac{n}{2}+\frac{n}{q}}=i-\frac{n}{p}+\frac{n}{q}$$
        the following holds:
        $$\norm{\nabla^{i}T_j}_{L^p_{1/2}}\leq C\norm{\nabla^k T_j}_{L^2_{1/2}}^\theta \norm{T_j}_{L^q_{1/2}}^{1-\theta}.$$
        As $T\in H^{k}_{1/2}(S^2 M)\cap L^q_{1/2}(S^2 M)$, the norms on the right converge to the corresponding norms of $T$. By the Sobolev embedding theorem (see \cite[Lemma 3.6(c)]{Lee}), the norm on the left converges to the corresponding norm of $T$ under the assumption that $\frac{n}{2}-\frac{n}{p}\leq k-i$.
    \end{proof}

    \begin{corollary}\label{gagliardo_nirenberg_cor}
        With assumptions as in Theorem \ref{gagliardo_nirenberg}, we have
        $$\norm{\nabla^{i}T}_{L^2_\varphi}\leq C\norm{\nabla^k T}_{L^2_\varphi}^{i/k}\norm{T}_{L^2_\varphi}^{1-i/k}\qquad 0\leq k\leq \ell$$
        $$\norm{\nabla^{i}T}_{L^{2k/i}_\varphi}\leq C\norm{\nabla^k T}_{L^2_\varphi}^{i/k}\norm{T}_{L^\infty}^{1-i/k}\qquad \frac{n}{2}\leq k\leq \ell$$
        where $\varphi\in C^{\ell,\alpha}(M)$ is a function with asymptotics $\varphi=-r+\mathcal{O}\p{e^{-r}}$, with respect to the distance function described in Remark \ref{distance_function}. Furthermore, $L^2_\varphi$ denotes the standard $L^2$ space, with respect to the weighted volume form $e^{-\varphi}\;dV_{g}$.
    \end{corollary}

\end{document}